\documentclass[preprint,3p]{elsarticle1}

\usepackage{graphicx}

\usepackage{amssymb}
\usepackage{comment}

\usepackage{lineno}

\usepackage{amsmath}
\usepackage{amsthm}

\usepackage{mathrsfs}

\usepackage{psfrag}
\usepackage{color}

\journal{Linear Algebra and its Applications}

\newtheorem{theorem}{Theorem}
\newtheorem{lemma}[theorem]{Lemma}
\newtheorem{example}[theorem]{Example}
\newtheorem{definition}[theorem]{Definition}
\newtheorem{claim}{Claim}

\newtheorem{corollary}[theorem]{Corollary}
\newtheorem{remark}{Remark}
\newtheorem{question}{Question}
\newtheoremstyle{algstyle}%
  {10mm}       
  {10mm}       
  {\tt}   
  {0pt}        
  {\bfseries}  
  {\newline}   
  {10mm}       
  {\thmname{#1}\thmnumber{ #2}\thmnote{ (#3)}}          
\theoremstyle{algstyle}
\newtheorem{algorithm}{Algorithm}

\newtheoremstyle{algdashstyle}%
  {10mm}       
  {10mm}       
  {\tt}   
  {0pt}        
  {\bfseries}  
  {\newline}   
  {10mm}       
  {\thmname{#1}\thmnumber{ #2}$'$\thmnote{ (#3)}}          
\theoremstyle{algdashstyle}

\newcommand{\nw}[1]{%
\textbf{#1}%
}

\newcommand{\mnw}[1]{%
\boldsymbol{#1}%
}

\newcommand{\bbmatrix}[1]{%
\begin{bmatrix} #1 \end{bmatrix}%
}

\newcommand{\ppmatrix}[1]{%
\begin{pmatrix} #1 \end{pmatrix}%
}

\newcommand{\mydot}[1]{%
\stackrel{\text{\Large .}}{#1}%
}

\newcommand{\lrar}{\leftrightarrow}

\newcommand{\subseteqn}{\hspace{0.1cm}\subseteq \hspace{0.1cm}}
\newcommand{\supseteqn}{\hspace{0.1cm}\supseteq \hspace{0.1cm}}
\newcommand{\equaln}{\hspace{0.1cm} = \hspace{0.1cm}}
\newcommand{\plusn}{\hspace{0.1cm} + \hspace{0.1cm}}
\newcommand{\inn}{\hspace{0.1cm} \in \hspace{0.1cm}}

\newcommand{\equivn}{\hspace{0.1cm} \equiv  \hspace{0.1cm}}
\newcommand{\lrarn}{\hspace{0.1cm} \lrar  \hspace{0.1cm}}

\newcommand{\V}{\mbox{$\cal V$}} 
\newcommand{\lnew}{\mbox{$ L$}} 
\newcommand{\F}{\mbox{$\cal F$}} 
\newcommand{\Vm}{{\mathcal V}_{M}}            			
\newcommand{\Va}{{\cal V}_{A}}              			

\newcommand{\Vabt}{{\cal V}_{AB}^T}
\newcommand{\Vbct}{{\cal V}_{BC}^T}
\newcommand{\transp}{{^T}}
\newcommand{\Vap}{{\cal V}_{AP}}            

  \newcommand{\Fx}{{\cal F}_{X}} 
  \newcommand{\Fy}{{\cal F}_{Y}}
   
    \newcommand{\0}{{\mathbf 0}}        
        
 \newcommand{\Vab}{{\cal V}_{AB}}   			
 \newcommand{\Vcd}{{\cal V}_{CD}}   			
 \newcommand{\Vaadash}{{\cal V}_{AA'}}   			
 \newcommand{\Vadashbdash}{{\cal V}_{A'B'}}   			
\newcommand{\Vhab}{\hat{{\cal V}}_{AB}}   			
   
\newcommand{\Vhatwodashbtwodash}{\hat{{\cal V}}_{A"B"}}   
 
\newcommand{\Vs}{{\cal V}_{S}}             
\newcommand{\Iaa}{{ I}_{AA'}}           			
\newcommand{\Iww}{{ I}_{WW'}}           			
\newcommand{\Iwminusw}{{ I}_{\dw(-\mydot{W'})}}           			
\newcommand{\Iwdw}{{I}_{W\mydot{W}}}  			
\newcommand{\Vsp}{{\cal V}_{SP}}           			
\newcommand{\Ksp}{{\cal K}_{SP}}           			
\newcommand{\Ks}{{\cal K}_{S}}           			
\newcommand{\Kps}{{\cal K}_{PS}}           			
\newcommand{\Kp}{{\cal K}_{P}}           			
\newcommand{\Kpq}{{\cal K}_{PQ}}           			
\newcommand{\Ksq}{{\cal K}_{SQ}}           			
\newcommand{\Lpq}{{ L}_{PQ}}           			
\newcommand{\Q}{{\mathbb{Q}}}
    
\newcommand{\Vp}{{\cal V}_{P}}              
\newcommand{\Vbc}{{\cal V}_{BC}}              
     
\newcommand{\Vabperp}{{\cal V}_{AB}^{\perp}}  
 
\newcommand{\Vb}{{\cal V}_{B}}              			
\newcommand{\Vwdwm}{{\cal V}_{W\mydot{W}M}}  			
\newcommand{\Vadjwdw}{{\cal V}^a_{W'\mydot{W'}}}  			
\newcommand{\Vpdwmu}{{\cal V}_{P\mydot{W}M_u}}  			
\newcommand{\Xwdw}{{\cal X}_{W\mydot{W}}}  			
\newcommand{\Ywdw}{{\cal Y}_{W\mydot{W}}}  			
\newcommand{\Vdwmy}{{\cal V}_{\mydot{W}M_y}}  			
\newcommand{\Vtdwmy}{{\tilde{\cal V}}_{\mydot{W'}M_y'}}  			
\newcommand{\Vtdpmy}{{\tilde{\cal V}}_{\mydot{P'}M_y'}}  			
\newcommand{\Vwdwmumy}{{\cal V}_{W\mydot{W}M_uM_y}}  			
\newcommand{\Vwdwmu}{{\cal V}_{W\mydot{W}M_u}}  			
\newcommand{\Vdpmy}{{\cal V}_{\mydot{P}M_y}}  			
\newcommand{\Vpmu}{{\cal V}_{PM_u}}  			
\newcommand{\Vpdpmumy}{{\cal V}_{P\mydot{P}M_uM_y}}  			
\newcommand{\Vpdpmu}{{\cal V}_{P\mydot{P}M_u}}  			
\newcommand{\Vtwdw}{{\tilde{\cal V}}_{W'\mydot{W'}}}  			
\newcommand{\Vtwdwm}{{\tilde{\cal V}}_{W'\mydot{W'}M_u'M_y'}}  			
\newcommand{\Vtpdp}{{\tilde{\cal V}}_{P'\mydot{P'}}}  			
\newcommand{\Vtpdpm}{{\tilde{\cal V}}_{P'\mydot{P'}M_u'M_y'}}  			
\newcommand{\Vwmu}{{\cal V}_{WM_u}}  			
\newcommand{\Vmydw}{{\cal V}_{\mydot{W}M_y}}  			
\newcommand{\Vdwdp}{{\cal V}_{\mydot{W}\mydot{P}}}  			
\newcommand{\Vonedwdp}{{\cal V}^1_{\mydot{W}\mydot{P}}}  			
\newcommand{\Vonetildedwdp}{{\cal V}^1_{\mydot{\tilde{W}}\mydot{{P}}}}  			
\newcommand{\Vtwotildedwdp}{{\cal V}^2_{\mydot{\tilde{W}}\mydot{{P}}}}  			
\newcommand{\tildedwdp}{{\mydot{\tilde{W}}\mydot{{P}}}}  			
\newcommand{\wtildedw}{{W\mydot{\tilde{W}}}}  			
\newcommand{\tildedw}{{\mydot{\tilde{W}}}}  			
\newcommand{\Vtwodwdp}{{\cal V}^2_{\mydot{W}\mydot{P}}}  			
\newcommand{\Vtwowp}{{\cal V}^2_{{W}{P}}}  			
\newcommand{\Vtwodwdq}{{\cal V}^2_{\mydot{W}\mydot{Q}}}  			
\newcommand{\Vtwodpdq}{{\cal V}^2_{\mydot{P}\mydot{Q}}}  			
\newcommand{\Vwp}{{\cal V}_{WP}}  			
\newcommand{\Vonewp}{{\cal V}^1_{WP}}  			
\newcommand{\Vonepq}{{\cal V}^1_{PQ}}  			
\newcommand{\Vonewq}{{\cal V}^1_{WQ}}  			
\newcommand{\Vtonewp}{{\tilde{\cal V}}^1_{W'P'}}  			
\newcommand{\Vttwodwdp}{{\tilde{\cal V}}^2_{\dwd\dPd}}  			
\newcommand{\Vw}{{\cal V}_{W}}  			
\newcommand{\Vtildedw}{{\cal V}_{\mydot{\tilde{W}}}}  			
\newcommand{\Vwdw}{{\cal V}_{W\mydot{W}}}  			
\newcommand{\Vonewdw}{{\cal V}^u_{W\mydot{W}}}  			
\newcommand{\oVonewdw}{{\cal V}^1_{W\mydot{W}}}  			
\newcommand{\Vtwowdw}{{\cal V}^l_{W\mydot{W}}}  			
\newcommand{\oVtwowdw}{{\cal V}^2_{W\mydot{W}}}  			
\newcommand{\wdw}{{W\mydot{W}}}  			
\newcommand{\Vwdwk}{({\cal V}_{W\mydot{W}})^{(k)}}  			
\newcommand{\Vdw}{{\cal V}_{\mydot{W}}}  			
\newcommand{\Vpdpm}{{\cal V}_{P\mydot{P}M}}  			
\newcommand{\Vqdqm}{{\cal V}_{Q\mydot{Q}M}}  			
\newcommand{\Vpdp}{{\cal V}_{P\mydot{P}}}  			
\newcommand{\Vqdq}{{\cal V}_{Q\mydot{Q}}}  			
\newcommand{\Vpdpk}{({\cal V}_{P\mydot{P}})^{(k)}}  			
\newcommand{\dw}{{\mydot{W}}}  			
\newcommand{\dws}{{\mydot{w}}}  			
\newcommand{\dwd}{{\mydot{W'}}}  			
\newcommand{\dW}{{\mydot{w}}}  			
\newcommand{\dP}{{\mydot{P}}}  			
\newcommand{\dPd}{{\mydot{P'}}}  			
\newcommand{\dQ}{{\mydot{Q}}}  			
\newcommand{\dwdp}{{\mydot{W}\mydot{P}}}  			
\newcommand{\dotp}{{\mydot{P}}}  			
\newcommand{\dwsmall}{{\mydot{w}}}  			

\newcommand{\E}{\mbox{$\cal E$}} 
 
\newcommand{\G}[0]{{\cal G}}                       
  
\newcommand{\K}[0]{{\cal K}}                       

\newcommand{\KSP}{\mbox{${\cal K}_{SP}$}}    		
\newcommand{\KSQ}{\mbox{${\cal K}_{SQ}$}}    		
\newcommand{\KPQ}{\mbox{${\cal K}_{PQ}$}}    		
\newcommand{\M}{\mbox{$\cal M$}}

\newcommand{\W}[0]{{\cal W}}                       
 
\newcommand{\bD}{\mbox{${\bf B}$}}  					
\newcommand{\B}{\mbox{${\cal B}$}}  				

\newcommand{\fS}{\mbox{${\bf f}_{S}$}}  				
\newcommand{\fP}{\mbox{${\bf f}_{P}$}}  				
\newcommand{\fQ}{\mbox{${\bf f}_{Q}$}}  				

\newcommand{\pa}{\mbox{${+}_{a}$}}      				
\newcommand{\pb}{\mbox{${+}_{b}$}}      				
\newcommand{\pc}{\mbox{${+}_{c}$}}      				
\newcommand{\pdw}{\mbox{${+}_{\mydot{w}}$}}      				
\newcommand{\pdwd}{\mbox{${+}_{\mydot{w}'}$}}      				
\newcommand{\pdp}{\mbox{${+}_{\mydot{p}}$}}      				

\newcommand{\x}{\mbox{${\bf x}$}}             		

\newcommand{\al}{\Box\,}

\newcommand{\ldw}{\lambda ^{\mydot{w}}}
\newcommand{\ldws}{\lambda ^{\mydot{w}}}
\newcommand{\ldwd}{\lambda ^{\mydot{w}'}}

\newcommand{\ldp}{\lambda ^{\mydot{p}}}

\begin{document}

\begin{frontmatter}



\title{On the linking of number lattices}

\author[hn]{H. Narayanan\corref{cor1}}
\ead{hn@ee.iitb.ac.in}
\cortext[cor1]{Corresponding author}
\author[hari]{Hariharan Narayanan}
\ead{hariharan.narayanan@tifr.res.in}
\address[hn]{Department of Electrical Engineering, Indian Institute of Technology Bombay}
\address[hari]{School of Technology and Computer Science, Tata Institute of Fundamental Research}

\begin{abstract}
In this paper we study ideas which have proved useful in topological 
network theory \cite{HNarayanan1986a,HNarayanan,narayanan1987topological,HNarayanan1997} in the context of lattices of numbers.
A number lattice $\lnew_S$ is a collection of row vectors, over $\mathbb{Q}$ on a finite column set $S,$
generated by integral linear combination of a finite set of row vectors.
A generalized number lattice $\K_S$ is the sum of a number lattice $\lnew_S$
and a vector space $\V_S$ which has only the zero vector in common with  it.
The dual $\K^d_S$ of a generalized number lattice is the collection
of all vectors whose dot product with vectors in $\K_S$ are integral
and is another generalized number lattice.
 
We consider a linking operation ('matched composition`) between generalized number lattices $\K_{SP},\K_{PQ}$
(regarded as collections of row vectors on column sets $S\cup P, P\cup Q,$
respectively with $S,P,Q$ disjoint)
defined by
$\K_{SP}\leftrightarrow \K_{PQ}\equiv
\{(f_S,h_Q):((f_S,g_P)\in \K_{SP}, (g_P,h_Q) \in \K_{PQ}\},$
and another ('skewed composition`)
defined by
$\K_{SP}\rightleftharpoons \K_{PQ}\equiv
\{(f_S,h_Q):((f_S,g_P)\in \K_{SP}, (-g_P,h_Q) \in \K_{PQ}\}.$

We show that these basic operations  together with contraction 
and restriction, and the results, the implicit inversion theorem 
(which gives simple conditions for the equality
$\K_{SP}\lrar(\K_{SP}\lrar \K_S)= \K_S,$ to hold)  and implicit duality theorem
($(\K_{SP}\leftrightarrow \K_{PQ})^d= (\K_{SP}^d\rightleftharpoons \K_{PQ}^d$)),
are both relevant and useful in suggesting problems concerning
 number lattices and their solutions.
While sets of generalized number lattices are closed under matched and skewed 
composition and the dualization operation,
not all sets of number lattices are. However, it is easy to extract, from  a generalized number lattice, its  number 
lattice part.

Using the implicit duality theorem,
we give  simple 
methods of constructing new {\bf self dual} lattices
from old.

We also give  new and efficient algorithms for the following
problems:

\begin{itemize}
\item To construct a block diagonal  basis
for a number lattice, if it exists.
\item Given $\V_{SP},\K_P,$
 such that $\V_{SP}\lrar(\V_{SP}\lrar \K_P)= \K_P,$
where $\V_{SP}$ is a vector space with  a totally unimodular basis matrix,
to construct

1. {\bf reduced bases} for the number lattice part of $\V_{SP}\lrar\K_P, \K_P^d, (\V_{SP}\lrar\K_P)^d,$
from  a reduced basis for the number lattice part of $\Kp;$

2. to construct
{\bf approximate shortest vectors} of
$L_P$ from approximate
shortest vectors of the number lattice part of $\V_{SP}\lrar L_P,$
when $L_P\subseteq \Vsp\lrar(\V_{SP}\lrar L_P).$
%
\end{itemize}

\end{abstract}


\end{frontmatter}


\section{Introduction}
\label{sec:intro}
Number lattices are established areas of research in mathematics and 
computer science (\cite{cassels}, \cite{micciwasser}). The notion of
`short vectors' has been of importance for number lattices
both from a theoretical as well as a computational point of view (\cite{LLL},\cite{babai1}).
They are often studied in terms of dual pairs and there 
are fundamental results relating the lengths of short vectors
in such pairs (\cite{banas},\cite{lagarias},\cite{regev},\cite{micci},\cite{peikert}).
However, `linked' number lattices,
related to each other through linear relations, do not appear to have been  paid attention to, 
in the literature.

In this paper, we introduce techniques for studying  number lattices
related to each other through linear relations and through dualization.
These techniques have been used successfully over many years to study electrical
networks (\cite{HNarayanan1997}). In this paper we use them, for instance, to 
efficiently construct reduced bases for 
a number lattice using another such, for a related number lattice.
We also bring out some analogies that exist between electrical networks
and linked number lattices, for instance, that dual number lattices
are similar to dual electrical networks and therefore that reciprocal
networks are analogous to self-dual number lattices. 

The conventional notion of a dual number lattice is natural for
the case of full dimensional lattices - the dual lattice is simply the collection
of all vectors whose dot product with the vectors of the original 
lattice is integral. For the case where it is not full
dimensional, we need to introduce an additional condition that we work
only within the span of the original number lattice.
If we drop this additional condition, the dual of a number lattice would,
 in general, be the sum of a number lattice and a vector space (Theorem \ref{thm:gennlchar}).

The linking operation that we use is the `matched composition' between  $\K_{SP},\K_{PQ}$
(regarded as collections of row vectors on column sets $S\cup P, P\cup Q,$
respectively with $S,P,Q$ disjoint)
defined by \\
$\K_{SP}\leftrightarrow \K_{PQ}\equiv
\{(f_S,h_Q):((f_S,g_P)\in \K_{SP}, (g_P,h_Q) \in \K_{PQ}\}.$
This operation is usually performed with $\Ksp$ as a vector space, which does 
the linking, and with $\Kpq \equiv \Kp,$ which gets linked to $\K_{SP}\leftrightarrow \K_{P}.$ 
We use the `$\lrar$' operation because it occurs naturally in many physical
systems, such as electrical networks, systems defined through
linear block diagrams, etc. It also lends itself to being treated implicitly,
without eliminating variables.
 But, in general, even if $\K_{P}$ is a number lattice, when $\Ksp $ is a vector space, $\K_{SP}\leftrightarrow \K_{P}$ would be the sum of a vector space
and a number lattice.

We therefore choose to work with a `generalized number lattice' $\Kp,$ defined to
be the sum of a number lattice $L_P$  and a vector space $\Vp.$
There is no loss in generality if we assume $\Vp$ and $L_P$ to be orthogonal.
In this case, both $\Vp$ and $L_P$ are unique for $\Kp$ and are also
easy to extract from it. Further, 
if $\Ks=\Vsp\lrar \Kp,$ and $\Ks=\Vs+L_S,$ with $\Vs,L_S,$  orthogonal,
 under some weak conditions on $\Vsp,$ one can algorithmically relate 
vectors in $L_S$ uniquely to those in $L_P$ (Theorems \ref{thm:inversevsnl2}, \ref{thm:inverselength}).

If we take the vectors in the generalized number lattice $\Kp$ to be  row
 vectors, $\Vsp\lrar \Kp$ can be regarded as a generalization of
 post multiplication by a matrix.
There are essential differences:
every vector of $\Kp$  would not get acted upon and a vector, in general,
would get `transformed' to a non trivial  affine subspace. 
Linking has the technical advantage, over post multiplication by a matrix,
that it can be handled implicitly. It is often associated with graphs,
which are algorithmically easy to process.
We choose to work, more generally, in terms of 
a `regular' vector space $\V_{SP},$ which is defined to be spanned by the rows of a totally unimodular matrix
 (such as the incidence matrix of a graph),
with $\K_{P}$ being a generalized number lattice. The greater generality,
thus available, is theoretically
more convenient while algorithmically the situation is almost as good as
working with graphs.

The theme of this paper is that the linking operation using regular vector spaces, has some of the 
properties of post multiplication by a $0,\pm 1$ matrix,
which are desirable from the point of view of number lattices, such as 
short vectors getting linked.
Finding the shortest (least norm) vector in a number lattice is known to be  a hard problem (\cite{arora},\cite{khot},\cite{micciwasser}).
Therefore, if we have somehow found a short vector in the number lattice 
part of $\Kp,$
it is interesting to note that a related vector in the number lattice part 
of $\Vsp\lrar \Kp$ is also short, the length being within a factor $|S|\times |P|,$  when $\Vsp$ is regular.

The outline of the paper follows.

Section \ref{sec:Preliminaries} is on preliminary definitions and results.

Section \ref{sec:basic} is on basic operations used in the paper.
These are, sum and intersection of generalized number lattices, generalized to 
include their being defined on different sets, restriction and 
contraction   and the dualization
of generalized number lattices.

Section \ref{sec:basisfromgen} deals with algorithms 
for constructing a basis for a number lattice, given a generating set for it.
We mainly use the algorithms available for building the Hermite normal
form (HNF) of an integral matrix. An instance of the usefulness of 
HNF - to detect if the given number lattice is the direct sum of 
smaller number lattices - is also described.

Section \ref{sec:linkgen} is on the fundamental results available 
for linking or dualizing generalized number lattices.
These are:

 the {\it implicit inversion theorem} which gives necessary and sufficient
conditions for the statement 

`$\Ksp\lrar(\Ksp\lrar\Ks)=\Ks$' to hold and

the {\it implicit duality theorem} which states that `$(\Ksp\lrar \Kp)^d=
\Ksp^d\lrar\Kp^d$'.
\\
The implicit inversion theorem is used to show that, when $\Vsp\lrar \Kp=\Ks$
and certain simple conditions on $\Vsp,\Kp$ are satisfied, the number lattice
parts of $\Kp$ and $\Ks$ have invertible maps between them.
\\
The implicit duality theorem is used to build new self dual number 
lattices from old.

Section \ref{sec:link} is a discussion of how linkages permit wide ranging
generalizations of the notion of maps.

Section \ref{sec:uniquelinkage}
contains results on lengths of corresponding vectors in the number 
lattice parts of generalized number lattices $\Kp,\Ks,$ linked through a regular vector
space  $\Vsp,$ which satisfies some simple conditions in relation to them.
Among other results, it is proved that a shortest vector in one of 
the number lattices can be transformed into a vector, in the other,
whose length is no more than the shortest vector of that lattice, by a factor of $|S|\times |P|.$

Section \ref{sec:approximatedual} discusses how LLL-reduced basis for a
 dual number lattice can be built from an LLL-reduced basis for the 
primal, efficiently.

Section \ref{sec:summary}
summarizes the results on approximate shortest vectors in number lattices
related to each other through linking by regular vector spaces and through dualization.

Section \ref{sec:closest} relates shortest vectors of special kinds to
the vectors closest to related number lattices.

Section \ref{sec:conclusion} is on conclusions.
\section{Preliminaries}
\label{sec:Preliminaries}
The preliminary results and the notation used are from \cite{HNarayanan1997}.
A \nw{vector} $\mnw{f}$ on a finite set $X$ over $\mathbb{F}$ is a function $f:X\rightarrow \mathbb{F}$ where $\mathbb{F}$ is a field. In this paper, we work only  with the rational field $\mathbb{Q}.$
The \nw{length} of a vector $x$ is the Euclidean norm $||x||$ of $x.$

The size of a set $X$ is denoted by $\mnw{|X|}.$
When $X$, $Y$ are disjoint, $\mnw{X\uplus Y}$ denotes the disjoint 
union of $X$ and $Y.$ A vector $f_{X\uplus  Y}$ on $X\uplus Y$ would be written as $\mnw{f_{XY}}.$ 
The {\bf sets} on which vectors are defined would  always be {\bf finite}. When a vector $x$ figures in an equation, we use the 
convention that $x$ denotes a column vector and $x^T$ denotes a row vector such as
in `$Ax=b,x^TA=b^T$'. Let $f_Y$ be a vector on $Y$ and let $X \subseteq Y$. The \textbf{restriction $f_Y|_X$} of $f_Y$ to $X$ is defined as follows:
$f_Y|_X \equiv g_X, \textrm{ where } g_X(e) = f_Y(e), e\in X.$

When $f$ is on $X$ over $\mathbb{F}$, $\lambda \in \mathbb{F},$ then  the \nw{scalar multiplication} $\mnw{\lambda f}$ of $f$ is on $X$ and is defined by $(\lambda f)(e) \equiv \lambda [f(e)]$, $e\in X$. When $f$ is on $X$ and $g$ on $Y$ and both are over $\mathbb{F}$, we define $\mnw{f+g}$ on $X\cup Y$ by \\
$(f+g)(e)\equiv f(e) + g(e),e\in X \cap Y,\ (f+g)(e)\equiv  f(e), e\in X \setminus Y,
\ (f+g)(e)\equiv g(e), e\in Y \setminus X.
$

When $X, Y, $ are disjoint,  $f_X+g_Y$ is written as  $\mnw{(f_X, g_Y)}.$ When $f,g$ are on $X$ over $\mathbb{F},$ the \textbf{dot product} $\langle f, g \rangle$ of $f$ and $g$ is defined by 
$ \langle f,g \rangle \equiv \sum_{e\in X} f(e)g(e).$

We say $f$, $g$ are \textbf{orthogonal} (orthogonal) iff $\langle f,g \rangle$ is zero.

An \nw{arbitrary  collection} of vectors on $X$ with $0_X$ as a member would be denoted by $\mnw{\mathcal{K}_X}$.

A collection $\K_X$ is a \nw{vector space} on $X$ iff it is closed under 
addition and scalar multiplication. 
For any collection $\K_X,$  $span(\K_X)$ is the collection of all
linear combinations of vectors in it.

For a vector space  $\V_X,$ since we take $X$ to be finite,
any maximal independent subset of $\V_X$ has size less than or equal to $|X|$ and this 
size can be shown
to be unique. A maximal independent subset of a vector
space $\V_X$ is called its \nw{basis} and its  size 
is called the  {\bf dimension} of $\V_X$ and denoted by ${\mnw{dim}(\V_X)}$
 or by ${\mnw{r}(\V_X)}.$
For any collection of vectors $\K_X,$
$\mnw{r}(\K_X)$
is defined to be $dim(span(\K_X)).$
The collection of all linear combinations of the rows of a matrix $A$ is a vector space 
that is denoted by $row(A).$

For any collection of vectors
$\mathcal{K}_X,$   the collection $\mnw{\mathcal{K}_X^{\perp}}$ is defined by
$ {\mathcal{K}_X^{\perp}} \equiv \{ g_X: \langle f_X, g_X \rangle =0\},$
It is clear that $\mathcal{K}_X^{\perp}$ is a vector space for 
any $\mathcal{K}_X.$ When $\mathcal{K}_X$ is a vector space,
 and the underlying set $X$ is finite, it can be shown that $({\mathcal{K}_X^{\perp}})^{\perp}= \mathcal{K}_X$ 
and  $\mathcal{K}_X,{\mathcal{K}_X^{\perp}}$ are said to be \nw{complementary orthogonal}. 
The symbol $0_X$ refers to the \nw{zero vector} on $X$ and $\mnw{0_X}$  refers to the \nw{zero vector space} on $X.$ The symbol $\mnw{\F_X}$  refers  to the collection of all vectors on $X$ over the field in question.
The notation $\mnw{\V_X}$ would always denote
a vector space on $X.$

A collection  $\mathcal{K}_X$ on $X$ is a \nw{number lattice} iff there exists a set of vectors $b_X^1, \cdots , b_X^n$ 
in $\Q^m, $ such that $\mathcal{K}_X\equiv \{\lambda _1b_X^1 + \cdots +\lambda _n b_X^n,\lambda_i \in \mathbb{Z}\}.$ 
Such a set is said to be a {\bf generating set} for the number lattice. If the generating set of vectors is independent, it is called a \nw{basis}
of the number lattice.
Starting from a generating set for a number lattice,
it is possible to build a basis for it through efficient algorithms.


A collection  $\mathcal{K}_X$ on $X$ is a \nw{generalized number lattice} iff there exists a set of  vectors\\ $b_X^1, \cdots , b_X^n, c_X^1, \cdots ,c_X^k$ 
in $\Q^m, $ such that $\mathcal{K}_X\equiv \{\lambda _1b_X^1 + \cdots +\lambda _n b_X^n+\mu_1c_X^1+ \cdots \mu_kc_X^k,\lambda_i \in \mathbb{Z}, \mu_j\in \Q\}.$
The set of vectors $\{b_X^1, \cdots , b_X^n, c_X^1, \cdots ,c_X^k\}$      is 
called a set of \nw{generating vectors} for $\K_X.$
We say  $\K_X$
is full dimensional iff $dim(span(\K_X))=|X|.$
Let $c_X^1, \cdots ,c_X^k,$ in the above definition, span the vector space 
$\V_X.$ We  call this the \nw{vector space part} of $\K_X$.
Clearly it is unique for $\K_X,$
being the collection of all vectors $c_X\in \K_X,$
such that $\alpha c_X\in \K_X, \alpha \in \mathbb{Q}.$

Each $b_X^i$ in the above generating set can be resolved into components $b_X^{i1}, b_X^{i2}$
in $\V^{\perp}_X, \V_X$ respectively. 
Therefore  $\mathcal{K}_X= \{\lambda _1b_X^{11} + \cdots +\lambda _n b_X^{n1}+\mu_1c_X^{1}+ \cdots \mu_kc_X^{k},\lambda_i \in \mathbb{Z}, \mu_j\in \Q\}.$
Let $\lnew_X$ be the number lattice generated by
$b_X^{11}, \cdots , b_X^{n1}.$
It is clear that $\lnew_X\subseteq \V^{\perp}_X.$
The generating set for $\K_X,$ can be taken without loss of generality to 
be the set $\{b_X^1, \cdots , b_X^n, c_X^1, \cdots ,c_X^k\},$ 
where the $b_X^i, i=1, \cdots n,$ form a basis for $\lnew_X$
and 
the $c_X^i, i=1, \cdots ,k,$ form a basis for $\V_X.$
Thus, we can take $\K_X\equiv L_X+\V_X,$ where $L_X,\V_X,$
are orthogonal. In this  case, both $L_X$ and $\V_X$
would be unique for the given $\K_X.$

%
%

Thus  generalized number lattices are  generalizations of finite
dimensional vector spaces as well as of number lattices.
They arise naturally from 
number lattices through simple operations like dualization.

It is clear that a generalized number lattice that does not contain
a nontrivial vector space (i.e., has the vector space part as the zero vector space)
is a number lattice.

We  use the symbol $\mnw{L_X}$ for the number lattice on $X$ as opposed to $\mathcal{K}_X$ for arbitrary collections of vectors on $X$ with a zero vector as a member. 

A matrix of full row rank, whose rows generate a vector space $\V_X,$
is called a \nw{representative matrix} for $\V_X.$
A representative matrix which can be put in the form $(I\ |\ K)$ after column
permutation is called a \nw{standard representative matrix}.

If the rows  of a matrix generate a number lattice $L_X,$
by integral linear combination, then the matrix is called a \nw{generating matrix}
for $L_X.$ If further the rows are linearly independent, the generating
matrix is called a \nw{basis matrix} for $L_X.$

When $X$, $Y$ are disjoint we usually write $\mathcal{K}_{XY}$ in place of $\mathcal{K}_{X\uplus Y}$.\\
The collection
$\{ (f_{X},\lambda f_Y) : (f_{X},f_Y)\in \mathcal{K}_{XY} \}$
is denoted by
$ \mathcal{K}_{X(\lambda Y)}.
$
When $\lambda = -1$ we would write more simply $\mathcal{K}_{X(-Y)}.$
Observe that $(\mathcal{K}_{X(-Y)})_{X(-Y)}=\mathcal{K}_{XY}.$

We say sets $X$, $X'$ are \nw{copies of each other} iff they are disjoint and there is a bijection, usually clear from the context, mapping  $e\in X$ to $e'\in X'$.
When $X,X'$ are copies of each other, the vectors $f_X$ and $f_{X'}$ are said to be copies of each other with  $f_{X'}(e') \equiv  f_X(e), e \in X.$ 
The copy $\K_{X'}$ of $\K_X$ is defined by
 $\K_{X'}\equiv\{f_{X'}:f_X\in \K_X\}.$
\subsection{Complexity convention}
For any $g,h:\Re^m\rightarrow \Re ,$ $g= \tilde{O}(h)$ denotes $g= O(h\log ^c(h)),$ for some constant $c>0.$

We take  the multiplication of two 
$n\times n$ matrices over a ring $R$ to be $O(n^{\theta})$
ring operations over $R.$
It is known that $\theta \leq 2.373.$
\section{Basic operations}
\label{sec:basic}
The basic operations we use in this paper are as follows:

\subsection{Sum and Intersection}
Let $\mathcal{K}_{SP}$, $\mathcal{K}_{PQ}$ be collections of vectors on sets $S\cup P,$ $P\cup Q,$ respectively, where $S,P,Q,$ are pairwise disjoint. The \nw{sum} $\mnw{\mathcal{K}_{SP}+\mathcal{K}_{PQ}}$ of $\mathcal{K}_{SP}$, $\mathcal{K}_{PQ}$ is defined over $S\cup P\cup Q,$ as follows:
\begin{align*}
 \mathcal{K}_{SP} + \mathcal{K}_{PQ} &\equiv  \{  (f_S,f_P,0_{Q}) + (0_{S},g_P,g_Q), \textrm{ where } (f_S,f_P)\in \mathcal{K}_{SP}, (g_P,g_Q)\in \mathcal{K}_{PQ} \}.
 \end{align*}
When $S$, $Q,$ are pairwise disjoint, $\mathcal{K}_S + \mathcal{K}_{Q}$ is usually written in this paper as $\mnw{\mathcal{K}_S \oplus \mathcal{K}_Q}$ and is called the \nw{direct sum}.
Thus,
\begin{align*}
\mathcal{K}_{SP} + \mathcal{K}_{PQ} &\equiv (\mathcal{K}_{SP} \oplus \0_{Q}) + (\0_{S} \oplus \mathcal{K}_{PQ}).
\end{align*}

The \nw{intersection} $\mnw{\mathcal{K}_{SP} \cap \mathcal{K}_{PQ}}$ of $\mathcal{K}_{SP}$, $\mathcal{K}_{PQ}$ is defined over $S\cup P\cup Q,$ where $S,P,Q,$ are pairwise disjoint, as follows:
\begin{align*}
\begin{split}
\mathcal{K}_{SP} \cap \mathcal{K}_{PQ} \equiv \{ f_{SPQ} &: f_{S P Q} = (f_S,h_P,g_{Q}),
\\& \textrm{ where } (f_S,h_P)\in\mathcal{K}_{SP}, (h_P,g_Q)\in\mathcal{K}_{PQ}.
\}.
\end{split}
\end{align*}
Thus,
\begin{align*}
\mathcal{K}_{SP} \cap \mathcal{K}_{PQ}\equiv (\mathcal{K}_{SP} \oplus  \F_{Q}) \cap (\F_{S} \oplus \mathcal{K}_{PQ}).
\end{align*}

It is immediate from the definition of the sum operation that sum of generalized number lattices is a generalized number lattice.
In Subsection \ref{subsec:dualize}, we show, by using the notion of dualization,
 that intersection of generalized number lattices is also 
a generalized number lattice.
\subsection{Restriction and contraction}
The \nw{restriction}  of $\mnw{\mathcal{K}_{SP}}$ to $S$ is defined by
$\mnw{\mathcal{K}_{SP}\circ S}\equiv \{f_S:(f_S,f_P)\in \mathcal{K}_{SP}\}.$
The \nw{contraction}  of $\mnw{\mathcal{K}_{SP}}$ to $S$ is defined by
$\mnw{\mathcal{K}_{SP}\times S}\equiv \{f_S:(f_S,0_P)\in \mathcal{K}_{SP}\}.$

Here again $\mnw{\mathcal{K}_{SPZ}\circ SP}$, $\mnw{\mathcal{K}_{SPZ} \times SP}$, respectively
when $S,P,Z,$ are pairwise disjoint,  denote\\  $\mnw{\mathcal{K}_{SPZ}\circ (S\uplus P)}$, $\mnw{\mathcal{K}_{SPZ} \times (S \uplus P)}.$

\subsubsection{Visibility of restriction and contraction  in  bases of $\V_{SP}$
}
\label{subsubsec:visible1}
Let  $(B_S\vdots B_P)$ be a representative matrix  for
vector space $\V_{SP}.$ 
By invertible row operations on $(B_S\vdots B_P),$
we can obtain a matrix of the form 
\begin{align*}
\ppmatrix{C_{SP}}\equiv
\ppmatrix{
        C_{1S} & \vdots & \0_{1P}\\
        C_{2S} & \vdots &C_{2P} 
},
\end{align*}
where the rows of $(C_{SP})$ form a basis for $\V_{SP}$
(therefore  the rows of $(C_{1S})$ are linearly independent),
and  the  rows of $(C_{2P})$ are linearly independent.
Whenever $(f_S,f_P)$ is a vector in $\V_{SP},$
it is clear that $f_P$ is linearly dependent on the rows of $(C_{2P}).$
Since these rows are independent, if $(f_S,0_P)$ is a vector in $\V_{SP},$
$f_S$ must be  linearly dependent on the rows of $(C_{1S}).$

We conclude that  $(C_{1S})$ is a representative matrix  for
$\V_{SP}\times S$ and that  $(C_{2P})$ is a representative matrix  for
$\V_{SP}\circ P$ 
and say that these latter are \nw{visible} in the
representative matrix  $(C_{SP})$
of $\V_{SP}.$

\subsubsection{Visibility of restriction and contraction in  bases of 
$L_{SP}$}
\label{subsec:visible}
The visibility of  restriction and contraction  in basis matrices
of number lattices is similar to the case of vector spaces
except that instead of invertible linear operations on the rows,
we have to perform `integral invertible row operations'.
These are integral operations whose inverses are also integral
and correspond to premultiplication by integral matrices whose
inverses are also integral. Such matrices are said to be {\bf unimodular}. 

Given a generating matrix $(C_P),$  for a number 
lattice $L_P,$ it is well known that, by integral invertible  row operations, we can get a matrix of the form 
\begin{align*}
\ppmatrix{
         \0_{1P}\\
        C_{2P} 
},
\end{align*}
in which the rows of $(C_{2P})$ are linearly independent 
(see Section \ref{sec:basisfromgen}).

Starting from  a basis matrix 
$(C_S\vdots C_P)$  for
the number lattice $L_{SP},$
by integral invertible  row operations on $(C_S\vdots C_P),$
we can obtain a basis matrix for $L_{SP},$ of the form
\begin{align*}
\ppmatrix{C_{SP}}\equiv
\ppmatrix{
        C_{1S} & \vdots & \0_{1P}\\
        C_{2S} & \vdots &C_{2P} 
},
\end{align*}
where 
the  rows of $(C_{2P})$ are linearly independent.
Whenever $(f_S,f_P)$ is a vector in $L_{SP},$
it is clear that $f_P$ is integrally linearly dependent on the rows of $(C_{2P}).$
If $(f_S,0_P)$ is a vector in $\V_{SP}.$
it can be expressed as an integral linear combination of the
rows of $(C_{SP}).$
This integral linear combination cannot involve the rows 
of $( C_{2S}  \vdots C_{2P}),$ 
since the rows of $(C_{2P})$ are linearly independent.
Therefore, $f_S$ is integrally linearly dependent on the rows of
$(C_{1S}).$ Further, we know that the rows of $(C_{1S})$ are linearly
independent.

We conclude that $(C_{1S})$ is a basis matrix for
$L_{SP}\times S$ and  $(C_{2P})$ is a basis matrix for
$L_{SP}\circ P$
and say that these are \nw{visible} in the basis 
matrix  $(C_{SP})$ of $L_{SP}.$

\subsection{Dualization}
\label{subsec:dualize}

For any generalized number lattice 
$\mathcal{K}_S,$   $\mnw{\mathcal{K}_S^d}$ is defined by
\begin{align*}
 \mathcal{K}^d_S \equiv \{ g_S: \langle f_S, g_S \rangle \ \mbox{an  integer},f_S\in \mathcal{K}_S \}.
\end{align*}

We have the following useful characterization of ${\mathcal{K}_S^d}$
(\cite{HNarayanan1997}). 
\begin{theorem}
\label{thm:gennlchar}
Let $\K_S\equiv \lnew^{(1)}_S+\V^{(1)}_S,$ where $\lnew^{(1)}_S$ is a number lattice and
$\V^{(1)}_S,$ a vector space orthogonal to it.
Let  $B_1$ be a basis matrix for the number lattice 
$\lnew^{(1)}_S,$
let  $C_1,D_1$ be representative matrices  respctively for the vector spaces $\V^{(1)}_S,$ $(\lnew^{(1)}_S+\V^{(1)}_S     )^{\perp}.$

Let 
\begin{align}
\label{eqn:kdual}
\ppmatrix{
        B_2 \\ C_2\\ D_2
}^T=
\ppmatrix{
        B_1 \\ C_1\\ D_1
}^{-1} .
\end{align}

Then 
\begin{enumerate}
\item
Rows of $B_1,B_2$ span the same vector space and $D_2$ is the representative
matrix for 
$\V^{(2)}_S\equiv (\lnew^{(1)}_S+\V^{(1)}_S)^{\perp}.$
\item
${\mathcal{K}_S^d}$ is the generalized number lattice which is equal to $ \lnew^{(2)}_S+\V^{(2)}_S,$ 
where   $B_2$ is a basis matrix for the number lattice $\lnew^{(2)}_S.$
\item $({\mathcal{K}_S^d})^d= {\mathcal{K}}_S.$


\end{enumerate}
\end{theorem}

\begin{proof}
Part 1 is straightforward.

2.  It is clear that ${\mathcal{K}_S^d}\supseteq \lnew^{(2)}_S+\V^{(2)}_S,$
Let $y_S\in {\mathcal{K}_S^d}.$ 
Since $\langle x_S,y_S \rangle, x_S\in \V^{(1)}_S ,$ is an integer
we must have  $y_S\in (\V^{(1)}_S)^{\perp}=span(\lnew^{(2)}_S+\V^{(2)}_S).$  
We can therefore write $y_S= y^1_S+y^2_S, y^1_S\in span(\lnew^{(2)}_S),
y^2_S\in span(\V^{(2)}_S)= (\lnew^{(1)}_S+\V^{(1)}_S)^{\perp}.$
By part 1, rows of $B_2$ generate $span(\lnew^{(1)}_S).$
Let $(y^1_S)^T=\lambda ^TB_2.$
We will show that $\lambda ^T$ is integral.
If $x^1_S\in \lnew^{(1)}_S, \mbox{\ i.e,\ } x^1_S= (B_1)^T\mu, \mu \ \mbox{integral},$ then since $
\langle y^1_S +y^2_S, x^1_S\rangle = \langle y^1_S,x^1_S \rangle $
must be an integer,  we must have
$\langle y^1_S,x^1_S \rangle = \lambda^TB_2(B_1)^T\mu= \lambda^T(I)\mu=\lambda^T\mu,$ integral 
for arbitrary $\mu .$
This is possible only if $\lambda^T$ is integral.
Thus $y_S\in L^{(2)}_S+\V^{(2)}_S.$
We conclude that ${\mathcal{K}_S^d}=\lnew^{(2)}_S+\V      ^{(2)}_S.$

3. This is immediate from part 2 above.
\end{proof}

The following corollary is immediate
\begin{corollary}
\label{cor:vsnldual}
\begin{enumerate}
\item
When $\K_S$ is a full dimensional number lattice (i.e., dimension = $|S|$),
$\mathcal{K}_S^d$ is also a full dimensional number lattice.
\item
When $\K_S$ is a vector space, $\mathcal{K}_S^d$ is the
space $\K_S^{\perp}$ complementary orthogonal to it.
\end{enumerate}
\end{corollary}

The following result is  easy to see.

If $\K_S , \widehat{\K}_{S}$ are  generalized number lattices, 
$$ (\K_S + \widehat{\K}_{S})^d  = \K_S^d\cap \widehat{\K}_{S}^d.$$ 
Using  $(\K^d)^d = \K,$ 
we get
$$(\K_S \cap \widehat{\K}_S)^d  = \K_S^d + \widehat{\K}_S^d.
$$
We will show that this result is true even if the generalized number lattices are defined on different sets with appropriate modification of the dualization operation.\\
When $S$, $P$ are disjoint, and $ \K_S ,\K_P$ are generalized number  lattices, it is easily verified that 
$$(\K_S \oplus \K_P)^d =\K_S^d \oplus \K_P^d.$$ 
(Here we abuse notation for better readability, The `$d$' on the left hand side
is with respect to $S\uplus P$ while the ones on the right hand side are with respect
to $S$ and $P,$ respectively.)\\
When $S$, $P$ are not disjoint, $\K_S + \K_P \equiv (\K_S \oplus \0_{P\setminus S}) + (\K_P \oplus \0_{S\setminus P}) $. If $ \K_S ,\K_P$ are generalized number lattices  on $S,P,$ respectively, we have
\begin{align*}
 (\K_S + \K_P)^d &= 
(\K_S \oplus \0_{P\setminus S}+\K_P \oplus \0_{S\setminus P})^d \\
&=(\K_S \oplus \0_{P\setminus S}) ^d \cap (\K_P \oplus \0_{S\setminus P})^d \\
&=  (\K_S^d\oplus (\0_{(P\setminus S)})^d) \cap 
(\K_P^d\oplus (\0_{(S\setminus P)})^d)\\
&=  (\K_S^d\oplus \F_{(P\setminus S)}) \cap 
(\K_P^d\oplus \F_{(S\setminus P)})\\
&=\K_S^d\cap \K_P^d,
\end{align*}
by the definition of intersection of generalized number lattices on two distinct sets. 

Using $(\K^d)^d = \K$ for generalized number lattices, we have that $\K_S \cap \K_P  = (\K_S^d + \K_P^d)^d,$  is a generalized number lattice when $\K_S,\K_P$ are generalized number lattices.

The following results for generalized number lattices can also be easily verified:
\begin{align*}
 (\K_{SP}\circ S )^d &= \K_{SP}^d \times S \\
 (\K_{SP}\times S )^d & = \K_{SP}^d \circ S \textrm{ using } (\K^d)^d = \K. 
\end{align*}
The above pair of results will be referred to as 
the \textbf{dot-cross duality}.  

When $\K_{SP}$ is a generalized number lattice, it is immediate
from the definition that so is $\K_{SP}\circ S.$
It follows that  $\K_{SP}\times S=(\K^d_{SP}\circ S)^d$ is also
a generalized number lattice.

\section{Constructing a basis from a generating set}
\label{sec:basisfromgen}
A convenient basis for a number lattice, for many purposes, is the one
in Hermite Normal Form (HNF). This is unique for a given number lattice.
Given an integral generating matrix of a number lattice,
one can construct an HNF basis for the latter, 
by building the HNF of the generating matrix. When the matrix is rational,
but not integral, one multiplies all the entries by an integer $k,$ so that
they become integers, builds the HNF of the resulting integral matrix,
and divides all the entries by $k.$ We will call the resulting matrix,
which is unique for the given rational generating matrix, its HNF.
\subsection{Hermite Normal Form}
The Hermite Normal form (HNF) of an integral matrix is the number lattice 
analogue of the row reduced echelon form for  matrices over $\Q.$
Our definition is row based. The column based  HNF can be defined similarly.
We remind the reader that a square matrix is said to be {\bf unimodular}
iff it has integral entries and has determinant equal to one.
It is clear that the inverse of a unimodular matrix is unimodular
and that product of unimodular matrices is also unimodular.
We say that two matrices are {\bf integrally row equivalent}
iff each can be obtained from the other by integral row operations.
When the matrices have the same number of rows, we say that 
 each can be obtained from the other by {\bf integral invertible  
operations}
iff each can be obtained from the other by premultiplication by a unimodular
matrix. 
HNF is defined as follows:

\begin{definition}
An integral matrix of full column rank is said to be in the 
{\bf Hermite Normal Form} iff it has the form 
$ \left[
\begin{array}{lll}
 \bD \\
 {\bf 0} 
\end{array} \right] $ 
where $\bD$ satisfies the following: \\
\noindent{\bf i.}  it is an upper triangular, integral, nonnegative matrix; \\
\noindent{\bf ii.} its diagonal entries are positive and have the unique highest 
     magnitude in their columns. 
\end{definition}

When a matrix $K$ has dependent columns, one first picks the 
sequence of columns $c_{i_1}, \cdots ,c_{i_k},$ such that,
scanning from the left,
$c_{i_1}$ is the first nonzero column, 
$c_{i_j} $ is independent of  all columns  occurring before it
and all columns occurring after $c_{i_k}$ are dependent on
$c_{i_1}, \cdots ,c_{i_k}.$
We will call such a column basis {\bf lexicographically earliest}.

Let the submatrix composed of these columns be $M$ and let $T$ be a unimodular 
matrix such that $TM$
is in HNF form. Then $TK$ is said to be the HNF for $K.$
Column $c_j$ will have no nonzero entries in rows after $r$ if
$j< i_{r+1}.$

The HNF of a matrix can be seen to be unique for a given matrix, by observing that
once a matrix is in the HNF form, it cannot be put in another such form
by integral invertible  row operations.
If two matrices with independent rows are integrally row equivalent,
it is clear that they have the same HNF matrix. Therefore all basis matrices
of a number lattice have the same HNF. 

A naive algorithm for constructing the HNF of a matrix is as follows.
(It is  naive because it does not guarantee that numbers encountered during
intermediate states of the algorithm do not grow exponentially large
in terms of the size of the matrix (\cite{kannan}).)

Let $c_{i_1}, \cdots ,c_{i_k},$ be the lexicographically earliest column
basis for the matrix. Let us suppose columns $c_{i_1}, \cdots ,c_{i_r},
r<k,$ of this matrix
satisfy properties $(i)$ and $(ii)$ above of HNF matrices. Note that 
the submatrix of this matrix composed of columns $t<i_{r+1}$  
has rows $j>r$ as zero rows. 

Let $c'_{i_{r+1}},$
denote the column vector composed of  the entries $(j,i_{r+1}),j\geq r+1,$
of the column $c_{i_{r+1}}.$ Perform integral invertible operations
on the rows $j\geq r+1$ of the matrix to bring the gcd, say $d,$ of the entries
of $c'_{i_{r+1}}$ 
to the $(r+1,i_{r+1})$ position, all other entries being zero.
Subtract integral multiples of the present $(r+1)^{th}$ row of 
the matrix from earlier rows so that entries $(j,i_{r+1}), j< r+1,$
are all less than $d.$
(Note that these operations do not disturb the submatrix composed of columns $t<i_{r+1}.$)     
This completes the processing of column $c_{i_{r+1}},$
of the matrix.
It can be seen that $c_{i_1}, \cdots ,c_{i_{r+1}}$ of the resulting matrix satisfy properties $(i)$ and $(ii)$ above of HNF matrices.

The present fastest algorithm for computing the HNF $H$ of an $m\times n,$ 
rank $n$ 
integral matrix $A,$ appears to be the one
in \cite{storjohann}. This algorithm uses the ideas in \cite{hafner} and has complexity 
$\tilde{O}(mn^{\theta}log(max(|A_{ij}|))$ time. 

The algorithm also produces, in addition to the HNF $H,$ a unimodular matrix 
$R,$ such that $H=RA.$
The matrix $R$ has entries of bit size $\tilde{O}(nlog(max(|A_{ij}|))).$
We will call this the SL-algorithm for HNF.

Next, let us consider the case of a number lattice that is not full dimensional. Let  $(C_{1S}|C_{1P})$ be an integral matrix with linearly independent
rows and maximal
 independent columns corresponding to set $S.$
If now we have to find a basis for the number lattice $L_{SP}$ generated by rows of the
$(m\times n)$
matrix
\begin{align*}
C_{SP}\equiv \bbmatrix{C_{1S}&\vdots& C_{1P}\\
C_{2S}&\vdots& C_{2P}},
\end{align*}
where the second set of rows are linearly dependent upon the first,
we can use the SL-algorithm
on the set of columns $S$ and
obtain the $(m\times m)$ unimodular matrix $R$ in 
$\tilde{O}(m|S|^{\theta}log(\kappa))$ time, where $\kappa=max(|C_{SP}(i,j)|).$
The bit size of entries in $R$ will be $\tilde{O}(|S|\times log(\kappa)).$
We can premultiply the
matrix $C_{SP}$ by  $R$ to obtain
the matrix
\begin{align}
\label{eqn:form}
\hat{C}_{SP}\equiv  RC_{SP}= \bbmatrix{\hat{C}_{1S}&\vdots & \hat{C}_{1P}\\
\0_{2S}&\vdots & \0_{2P}},
\end{align}
where the first set of rows of the matrix constitute a basis matrix for
the number lattice $L_{SP}.$
Multiplication of the  set of columns of $C_{SP}$ by $R$ can be
carried out in $\tilde{O}(m|S|^{\theta}log(\kappa)) +
\tilde{O}(m^2|P| \times |S|log(\kappa))$ time.

 In Subsection \ref{subsec:visible} we have discussed
the usefulness of putting,  by integral invertible  row operations,
  a   basis matrix
$(C_S\vdots C_P)$  for
the number lattice $L_{SP}$
into the form
\begin{align}
\label{eqn:form2}
\ppmatrix{C_{SP}}\equiv
\ppmatrix{
        C_{1S} & \vdots & \0_{1P}\\
        C_{2S} & \vdots &C_{2P} 
}.
\end{align}
We have shown there that,  if the  rows of $(C_{2P})$ are linearly independent,
then rows of $(C_{1S}),(C_{2P}) $ respectively form bases for the number lattices
$L_{SP}\times S$ and $L_{SP}\circ P.$
By using the SL-algorithm, the basis of the form in
Equation \ref{eqn:form2} can be computed in $\tilde{O}(m|P|^{\theta}log(\kappa))$
+ $\tilde{O}(m^2n|P|log(\kappa))$ time,
where $(C_S\vdots C_P)$ is an $m\times n$ matrix and $log(\kappa)$
is the maximum bit size  of entries in $C_{SP}.$
The second term in the complexity calculation is that of the multiplication  $RC_{SP}.$ Here, $R$ has bit size of entries $\tilde{O}(|P|log(\kappa)).$

\ref{sec:Var} contains a discussion of a few common variations
on the problem of finding a basis of a number lattice.

\subsubsection{Connectedness of number lattices through HNF}
A number lattice $L_S$ can sometimes be regarded as the direct sum 
$\bigoplus _iL_{S_i}, i=1,2,\cdots ,k .$
of
number lattices defined over the blocks of a partition $\{S_1, \cdots , S_k\}$ of $S.$
This would clearly make computations with the 
lattice much easier. In this section, we discuss how to recognize this situation 
through the use of the Hermite normal form of a basis matrix 
for $L_S.$ 

From the definition of contraction and restriction it is clear
that if $L_S=\bigoplus _iL_{S_i}, i=1,2,\cdots ,k ,$
we must have $L_{S}\times S_i=L_{S}\circ S_i, i= 1, \cdots k.$

We now prove the converse.
Suppose $L_S=L_{S_1}\oplus L_{S_2}.$
Let us suppose, wlog,  that the elements of $S$ are ordered such that
those in $S_1$ occur before those in $S_2.$
It is clear by the discussion in Subsection \ref{subsec:visible},
that we can build a basis matrix for $L_S=L_{S_1S_2},$
which has the form
\begin{align*}
C_{S_1S_2}
\equiv
\ppmatrix{
        C_{1S_1} & \vdots & \0_{1S_2}\\
        C_{2S_1} & \vdots &C_{2S_2} 
},
\end{align*}
with rows of $(C_{2S_2})$ linearly independent.
We know that $(C_{1S_1})$ is a basis matrix for $L_{S}\times S_1=L_{S}\circ S_1.$
Therefore by integral invertible  row operations, we can reduce
$C_{S_1S_2}$ to the form 
\begin{align}
\label{eqn:form3}
\hat{C}_{S_1S_2}
\equiv
\ppmatrix{
        C_{1S_1} & \vdots & \0_{1S_2}\\
        \0_{2S_1} & \vdots &C_{2S_2} 
},
\end{align}
Thus $L_S= L_{S}\times S_1\oplus L_{S}\times S_2 =L_{S}\circ S_1\oplus L_{S}\circ S_2.$
The result follows by induction on $k.$

In general, we will be working with a column permuted version 
of the matrix in Equation \ref{eqn:form3}.
Therefore it is clear that $L_S= L_{S_1}\oplus\cdots\oplus  L_{S_m}$
iff it has a basis matrix where no row contains nonzero 
entries from columns corresponding to more than one of the $S_i.$
We will say that such a basis matrix has sets of columns $S_i, i=1, \cdots m,$ {\bf decoupled.}

We will now show that it is adequate for us to check if the HNF of any basis
matrix of $L_S$ has this property.

Let $c_{i_1}, \cdots ,c_{i_k},$ be the lexicographically earliest column basis
for a basis matrix of $L_S,$ which has sets of columns $S_i, \  i= 1, \cdots , m,$
 decoupled.

Suppose,  in the conversion to  HNF, as we scan the columns of the lexicographically earliest column basis from the left,  we have processed columns up to and including $c_{i_j}.$
Let the resulting matrix have sets of columns $S_i, i=1, \cdots m,$ decoupled.
We will show that this situation continues when we process $c_{i_{j+1}}$
also.

Let $r_z$ denote the $z^{th}$ row of the matrix at the current stage.
Let $c_{i_{j+1}}$
correspond to $e_{i_{j+1}}\in S_h$ and let $e_p\in S_q, \ h\ne q.$

If column $c_{i_{j+1}},$
 has a zero entry in  $r_t,$ or in $r_s, $
we would not be adding a multiple of one of them to the other.
Therefore we need only consider the case where the column has both the 
entries nonzero.

But in this case, the column $c_p$ corresponding to $e_p$
must have zero entry in both rows. So when we add an integral
mutiple of say $r_t$ to $r_s$ the column $c_p$ will 
continue to have zero entries in both rows.
So sets of columns $S_i, i=1, \cdots m,$ remain decoupled after this operation too.

Since processing column $c_{i_{j+1}}$
is made up of only such elementary operations we conclude sets of columns $S_i, i=1, \cdots m,$ remain decoupled even after this processing is complete.
%
%
%
%
It is therefore clear that during the conversion of a basis matrix of $L_S$
which has sets of columns $S_i, i=1, \cdots m,$
 decoupled,  
into the HNF, the matrices encountered will all have 
$S_i, i=1, \cdots m,$
 decoupled and this would be true for the HNF too.

But the HNF of all basis  matrices of a number lattice has to be the same.
Therefore if some basis matrix of $L_S$ has sets of 
columns $S_i, i=1, \cdots m,$
 decoupled,
  so will be the 
HNF basis matrix of $L_S.$ 

We say $S_1\subseteq S$ is a \nw{separator} for $L_S,$
iff $L_S=L_{S_1}\oplus L_{S-S_1}.$
A separator of $L_S$ is \nw{elementary} iff it does not contain another 
separator of $L_S$ as a proper subset.
The following simple algorithm detects the elementary 
separators of $L_S$ from its HNF.

Build the graph $\G_{L_S}$ from the HNF of $L_S$ as follows.
Let $S\equiv \{e_1, \cdots , e_n\}.$
Take $S$ to be the set of nodes of $\G_{L_S}.$
Join $e_i,e_j$ by an edge if there is a row in the HNF where
both columns have nonzero entries.
Let $S_i, \ i=1, \cdots , m$ be the node sets of the connected components of 
$\G_{L_S}.$ 
Then $S_i, \ i=1, \cdots , m$ are the elementary separators of $L_S.$

\section{Linking generalized number lattices}
\label{sec:linkgen}
In this section, we introduce two basic operations, viz., matched and skewed composition, for `linking' generalized number lattices.
We prove two basic results `implicit inversion (IIT)' and `implicit duality (IDT)'
(Theorems \ref{thm:inversevsnl1} and \ref{thm:idt})
involving these operations.
From IIT, we are able to show that number lattice parts of linked generalized
number lattices have similar properties. From IDT, among other things,
we develop a technique for building new self dual lattices from old.

\subsection{Matched and Skewed Composition}
Let $\Ksp,\Kpq,$ be collections of vectors respectively on $S\uplus P,P\uplus Q,$ with $S,P,Q,$ being pairwise disjoint.
Further, let $\0_{SP}\in \Ksp, \0_{PQ}\in \K_{PQ}.$

The \nw{matched composition} $\mnw{\mathcal{K}_{SP} \leftrightarrow \mathcal{K}_{PQ}}$ is on $S\uplus Q$ and is defined as follows:
\begin{align*}
 \mathcal{K}_{SP} \leftrightarrow \mathcal{K}_{PQ} 
  &\equiv \{
                 (f_S,g_Q): (f_S,h_P)\in \Ksp, 
(h_P,g_Q)\in \Kpq\}.
\end{align*}
Matched composition is referred to as matched sum in \cite{HNarayanan1997}.


The \nw{skewed composition} $\mnw{\mathcal{K}_{SP} \leftrightarrow \mathcal{K}_{PQ}}$ is on $S\uplus Q$ and is defined as follows:
\begin{align*}
 \mathcal{K}_{SP} \rightleftharpoons \mathcal{K}_{PQ} 
  &\equiv \{
                 (f_S,g_Q): (f_S,h_P)\in \Ksp, 
(-h_P,g_Q)\in \Kpq\}.
\end{align*}

When $S$, $Y$ are disjoint, both the matched and skewed composition of 
$\K_S,\K_Y,$ correspond to the direct sum $\K_S\oplus \K_Y$.

It is clear from the definition of matched composition and that of restriction
and contraction, that 
$$ \Ks\circ (S-T) =\Ks\lrar \F_T, \Ks\times (S-T)=\Ks\lrar \0_T, T\subseteq S.$$

When $\mathcal{K}_{SP}$, $\mathcal{K}_P$ are generalized number lattices,  observe that $(\mathcal{K}_{SP}\leftrightarrow \mathcal{K}_P) = (\mathcal{K}_{SP}\rightleftharpoons \mathcal{K}_P).$
%
%

When $S,P,Z,$ are pairwise disjoint,  
we have
\begin{align*}
(\mathcal{K}_{SPZ} \leftrightarrow  \mathcal{K}_{S}) \leftrightarrow  \mathcal{K}_{P} = (\mathcal{K}_{SPZ} \leftrightarrow  \mathcal{K}_{P}) \leftrightarrow  \mathcal{K}_{S} = \mathcal{K}_{SPZ} \leftrightarrow  (\mathcal{K}_{S} \oplus  \mathcal{K}_{P}).
\end{align*}

When $ \mathcal{K}_{S}\equiv \0_S, \mathcal{K}_{P}\equiv \mathcal{K}_{SPZ}\circ P,$ the above reduces to 
$$ \mathcal{K}_{SPZ}\times PZ\circ Z\equaln \mathcal{K}_{SPZ}\circ SZ\times Z.$$
Such an object is called a \nw{minor} of $\mathcal{K}_{SPZ}.$

In the special case where $Y\subseteq S$, the matched composition $\Ks\lrar \K_Y,$ is called the 
\nw{generalized minor} of $\mathcal{K}_S $
with respect to $\mathcal{K}_Y$. 

We have already seen (Subsection \ref{subsec:dualize}) that sum, intersection, restriction  and contraction 
of generalized number lattices are also generalized number lattices.
We therefore have the following result.

%
%
\begin{theorem}
\label{thm:opsummary}
\begin{enumerate}
\item 
Let $\K_{SP},\K_{PQ}$ be generalized number lattices with $S,P,Q$ disjoint.
From the definition of matched  composition, we then have
\begin{align*}
(\mathcal{K}_{SP}\leftrightarrow \mathcal{K}_{PQ}) = (\mathcal{K}_{SP}+ \mathcal{K}_{(-P)Q}) \times (S\cup Q)
\end{align*}
and also equal to
\begin{align*}
 (\mathcal{K}_{SP}\cap \mathcal{K}_{PQ})\circ (S\cup Q).
\end{align*}
Similarly from the definition of skewed composition,
\begin{align*}
 (\mathcal{K}_{SP} \rightleftharpoons \mathcal{K}_{PQ}) = (\mathcal{K}_{SP}+ \mathcal{K}_{PQ}) \times (S\cup Q)
\end{align*}
and also equal to
\begin{align*}
 (\mathcal{K}_{SP}\cap \mathcal{K}_{(-P)Q})\circ (S\cup Q).
\end{align*}
\item If $\K_{SP},\K_{PQ}$ are generalized number lattices with $S,P,Q$ disjoint, then $\Ksp\lrar\Kpq, \Ksp\rightleftharpoons \Kpq,$ are also
generalized number lattices.
\item Let $\E(\K_S, \cdots ,\K_Q)$ be an expression involving  generalized number lattices,\\
 and the operations  $+,\cap, \leftrightarrow, \rightleftharpoons .$ Then the expression yields a 
 generalized number lattice.

\end{enumerate}
\end{theorem}

\subsection{Vector spaces and number lattices}
We have seen in Theorem \ref{thm:opsummary}, that
matched and skewed composition of generalized number lattices
yield other generalized number lattices.
When $\Ks=\Ksp\lrar \Kp,$ we say $\Ks,\Kp$ are {\bf linked} through $\Ksp.$
Below, we prove two basic results related to linking, viz.,
`implicit inversion theorem' and `implicit duality theorem'.
These enable us to relate properties of generalized number lattices
that are linked,  under certain relatively weak conditions,
 in the case of the former and, very generally, in the case of the latter.
The most natural situation is when linking is done through vector spaces
although the theorems are valid  more generally.

Vector spaces are useful for linking number lattices as the following result 
indicates.
\begin{lemma}
\label{thm:vsnl}
Let rows of $(B_S\vdots B_P),(C_P\vdots C_Q)$ be bases for 
vector space $\V_{SP}$ and number lattice $L_{PQ}$
respectively.
\begin{enumerate}
\item 
Let 
$B_S=B_PN,$ for some matrix $N,$
equivalently, let $\Vsp\times S=\0_S.$
Then $\Vsp\lrar \Lpq$ is a number lattice on $S\uplus Q.$ 
\item Let $B_S=B_PN$ and let every row of $C_P$ be a linear combination 
of rows of $B_P, $ equivalently, let $\Lpq\circ P\subseteq \Vsp \circ P.$
Then 
$(C_PN\vdots C_Q)$ is a basis matrix for the number lattice $\Vsp\lrar \Lpq.$
\end{enumerate}
\end{lemma}
\begin{proof}
1. Let $(f_S,g_Q)\in \Vsp\lrar \Lpq.$ 
Then there exist 
$(f_S,h_P)\in \Vsp$ and $(h_P,g_Q)\in \Lpq.$
Therefore $(f_S,h_P)=\sigma ^T(B_S\vdots B_P)= \sigma ^T(B_PN\vdots B_P)$ for some rational vector $\sigma$
and
$(h_P,g_Q)=\lambda ^T(C_P\vdots C_Q)$ for some integral vector $\lambda .$
We thus have $\sigma ^TB_P=\lambda ^TC_P$ so that $\sigma ^TB_PN=\lambda ^TC_PN.$ 
Therefore $(f_S,g_Q)=\lambda ^T(C_PN\vdots C_Q).$
Thus the rows of $(C_PN\vdots C_Q)$ generate all vectors in $\Vsp\lrar \Lpq$
through integral linear combination.
Therefore $\Vsp\lrar \Lpq$
does not contain a nontrivial vector subspace.
On the other hand $\Vsp\lrar \Lpq$
is a generalized number lattice.
Since it does not contain a nontrivial vector subspace we conclude that it is a number lattice.

2. We saw in the proof of part 1, that the number lattice generated by the 
rows of $(C_PN\vdots C_Q)$ contains the number lattice $\Vsp\lrar \Lpq.$
Now let $(f_S,g_Q)=\lambda ^T(C_PN\vdots C_Q),$ where $ \lambda ^T$
is an integral vector and let 
$h_P\equiv \lambda ^TC_P$ so that $(h_P,g_Q)=\lambda ^T(C_P\vdots C_Q).$
By the hypothesis, there exists a rational vector $\sigma $ such that $\sigma ^TB_P=\lambda ^TC_P.$
\\
Now $\sigma ^T(B_S\vdots B_P)=\sigma ^T(B_PN\vdots B_P)=\lambda^T(C_PN\vdots C_P)=(f_S,h_P).$
Hence, $(f_S,h_P)\in \Vsp, (h_P,g_Q)\in \Lpq$ so that 
$(f_S,g_Q)\in \Vsp\lrar\Lpq.$
\end{proof}
\begin{remark}
Lemma \ref{thm:vsnl}, is intended to motivate the notion of linking as well
as the definition of generalized number lattice. Its hypothesis is strong
and the conclusion, not surprising.
However, if we allow the result of the linking to be 
a generalized number lattice, by using the implicit inversion theorem (IIT)
(Theorem \ref{thm:inversevsnl1}), we can obtain Theorem \ref{thm:inversevsnl2},
where with weaker hypothesis, we obtain a result of greater applicability.
\end{remark}
\begin{example}
\label{eg:port}
Let $Z\equiv S\uplus P$ and let $\V_Z$ be a vector space on $Z.$
Let $Z\equiv \{e_1, \cdots e_m\}$ and let
$Z'\equiv \{e'_1, \cdots e'_m\}$
be a disjoint copy of $Z.$
Let $\V_Z$ be a vector space on $Z$
and $\V_{Z'}$ be a copy on $Z'.$
The simplest such vector space is the coboundary space spanned by the incidence
matrix of a graph ${\G},$ with directed edges $Z.$
Let $\V_{ZZ'}\equiv \V_Z\oplus \V^{\perp}_{Z'}.$

Let  $(Q_S|Q_P)$ be a representative matrix  for $\V_Z,$
and let $(B_S|B_P)$ be a representative matrix  for $\V^{\perp}_Z.$
Let a maximal linearly independent set of columns of $(Q_S)$
be also a maximal
 linearly independent set of columns of $(Q_S|Q_P).$
Further let the columns of $(Q_P)$ be linearly independent.
It can be shown that a maximal
 linearly independent set of columns of $(B_S)$
will also be a maximal
 linearly independent set of columns of $(B_S|B_P)$
and the columns of $(B_P)$ will be linearly independent.
If $\V_Z$ is the coboundary space of graph $\G,$ the above conditions mean that
 the edges in
$P$ can be included in a tree as well as in the complement of a tree
(i.e., $P$ contains no loops or cutsets of $\G$).

A consequence is  that there will exist matrices $M,N$ such that
$(Q_P)= (Q_S)M$ and $(B_P)= (B_S)N.$
Therefore $\V_Z\times P=\0_P$ and $\V^{\perp}_{Z}\times P=\0_{P},$
i.e., 
$\V^{\perp}_{Z'}\times P'=\0_{P'},$
so that $\V_{ZZ'}\times PP'=\0_{PP'}.$

If $\lnew_{SS'}$ is a number lattice on $S\uplus S',$ by Theorem \ref{thm:vsnl}
,
$\V_{ZZ'}\lrar \lnew_{SS'}$ is a number lattice on $P\uplus P'.$

\end{example}
\subsubsection{Implicit Inversion}
\label{subsec:impinv}
Consider the equation
$$\K_{SP}\lrar \K_{PQ}=\K_{SQ}.$$
In this subsection we examine when, given $\K_{SP},\K_{SQ},$ the equation has 
some solution $\K_{PQ},$ and when the solution is unique.

The following lemma is a generalization of a result (Problem 7.5) in \cite{HNarayanan1997}. 
The proof is relegated to the appendix.
\begin{lemma}
\label{lem:Kgenminor}
1. Let $\KSP,\KSQ $ be collections of vectors on $S \uplus P,S\uplus Q $  respectively.
Then there exists a collection of vectors
$\KPQ$ on $P\uplus Q$ s.t. ${\bf 0}_{PQ} \in \KPQ $ and $\KSP \lrar \KPQ = \KSQ,$ 
only if
$$\KSP\circ S \supseteq \KSQ \circ S\ \mbox{and}\ 
 \KSP \times S \subseteq \KSQ \times S.$$
2. Let $\KSP$ be a collection of vectors closed under subtraction  on $S \uplus P$ 
and let  $\KSQ$ be a collection of vectors on $S\uplus Q,$
closed under addition. 
Further let $\KSP\circ S \supseteq \KSQ \circ S$ and let $ \KSP \times S \subseteq \KSQ \times S.$ 
Then, $\0_{SQ}\in \KSQ,$ and the collection of vectors
$\KSP \lrar \KSQ,$ is closed under addition with ${\bf 0}_{PQ}$ as a member. 
Further we have that $\KSP \lrar (\KSP\lrar \KSQ)= \KSQ.$\\
3. Let $\KSP$ be a collection of vectors closed under subtraction, 
and let $\KSQ$ satisfy
the conditions, 
closure under addition, 
$\KSP\circ S \supseteq \KSQ \circ S$ and
$ \KSP \times S \subseteq \KSQ \times S.$ 
Then the equation
$$\KSP\lrar \KPQ = \KSQ,$$
where $\KPQ$ has to satisfy
closure under addition, 
as well as the conditions 
$$\KSP \times P \subseteq \KPQ \times P \ and\  \KSP\circ P \supseteq \KPQ \circ P,$$
has $\hat{\K}_{PQ}\equiv \KSP\lrar\KSQ,$ as the unique solution.
\end{lemma}

\begin{remark}
\begin{itemize}
\item
We note that the  collections of vectors in Lemma \ref{lem:Kgenminor} can be over rings rather than over fields - in particular over the ring of integers.
\item The hypotheses of Lemma \ref{lem:Kgenminor} are clearly true when $\KSP, \KPQ, \KSQ$
are generalized number lattices. So the lemma holds for them.
\item
Only $\KSP$ has to be closed over subtraction.
The other two  collections $\KPQ,\KSQ$ have to be  only closed over addition with a zero vector as a member. In particular $\KSP$ could be a vector space over rationals while
 $\KPQ,\KSQ$  could be cones over rationals. 
\end{itemize}
\end{remark}
We specialize Lemma \ref{lem:Kgenminor} to the case of generalized 
number lattices in the following theorem.
We will call this the {\bf implicit
inversion theorem}(IIT) for generalized number lattices.
\begin{theorem}
\label{thm:inversevsnl1}
Consider the equation
$$\Ksp \lrar \Kpq =\Ksq, $$
where $\Ksp, \Kpq, \Ksq$ are generalized number lattices.
Then
\begin{enumerate}
\item given $\Ksp, \Ksq,  \exists \Kpq, $ satisfying the equation iff $\Ksp\circ S\supseteq \Ksq \circ S$ and $\Ksp\times S \subseteq \Ksq \times S.$ 
\item  given $\Ksp, \Kpq,\Ksq$ satisfying the equation, we have $\Ksp \lrar \Ksq =\Kpq $ iff
$\Ksp\circ P\supseteq \Kpq \circ P$ and $\Ksp\times P\subseteq \Kpq \times P.$
\item given $\Ksp, \Ksq,$ assuming that the equation  $\Ksp \lrar \Kpq =\Ksq $
is satisfied by some $\hat{\K}_{PQ} $ it is unique under the condition $\Ksp\circ P\supseteq \Kpq \circ P$ and $\Ksp\times P\subseteq \Kpq \times P.$
\end{enumerate}
\end{theorem}

The following theorem, 
which relates number lattice parts of linked generalized lattices,
is a consequence.

\begin{theorem}
\label{thm:inversevsnl2}


Let $\Vsp$ be a vector space, $\Kp$ be a generalized number lattice
with $\Kp= \Vp+L_{P}, \Vp\cap L_P=\0_P.$
Let $\Ks\equiv \Vsp \lrar \Kp ,$
and let $\Vs\equiv \Vsp\lrar \Vp. $
Let  $\Vsp\times P\subseteq \Vp$   and let 
$ \Vsp\circ P\supseteq \Kp.$   
Then
\begin{enumerate}
\item $\Vsp \lrar \Ks = \Kp;$
$\Vsp \lrar \Vs = \Vp$ ; 
\item 
 Let $B_P$ be a basis matrix for $L_P$ with rows $x_{1P}, \cdots, x_{nP}.$
Then there exist  vectors $\hat{x}_{iS}, i=1, \cdots, n,$ such that $(\hat{x}_{iS},x_{iP})\in \Vsp, i=1, \cdots, n.$

Let $B_S$ be the matrix with $\hat{x}_{iS}, i=1, \cdots, n,$ as  the
$i^{th}$ row. Then $B_S$ has independent rows and is a basis matrix for 
a number lattice $L_S \subseteq \Vsp\lrar L_P,$ 
such that 
$L_P\subseteq \Vsp\lrar L_S, \Vs\cap L_S=\0_S,$ and such that
$\Ks=\Vs+L_S.$

\item Let $L'_S \subseteq \Vsp\lrar L_P$ be such that $L'_S\cap \Vs=\0_S$
and such that $L_P\subseteq \Vsp\lrar L'_S.$

(a) If $x_P\in L_P, $ and $x_P\ne 0_P,$ then there is a unique $\hat{x}_S\in L'_S,  $ such
that $(\hat{x}_S,x_P)\in \Vsp.$ Further, $\hat{x}_S\ne 0_S. $ 


(b) Let $B_P$ be a basis matrix for $L_P$ with rows $x_{1P}, \cdots, x_{nP}.$
Let  $\hat{x}_{iS}\in L'_S, i=1, \cdots, n,$ be such that $(\hat{x}_{iS},x_{iP})\in \Vsp, i=1, \cdots, n.$
Let $B_S$ be the matrix with $\hat{x}_{iS}, i=1, \cdots, n,$ as  the
$i^{th}$ row. Then $B_S$ is a basis matrix for $L'_S.$

(c) $\Ks=L'_S+\Vs.$




%

\end{enumerate}
\end{theorem}
\begin{proof}
1. We have $\Vsp\times P\subseteq \Vp\subseteq \Kp$ and $\Vsp\circ  P\supseteq \Kp\supseteq \Vp.$
Therefore, by Theorem \ref{thm:inversevsnl1}, 
$$\Vsp\lrar(\Vsp\lrar\Kp)=\Vsp\lrar \Ks=\Kp \  
\mbox{and }
\Vsp\lrar(\Vsp\lrar\Vp)=\Vsp\lrar \Vs=\Vp.$$  


2.
We have, $L_P\subseteq \Kp\subseteq \Vsp\circ P.$
Therefore, if $x_{ip}\in L_P,i = 1, \cdots, n,$ there exist $\hat{x}_{iS},
i=1, \cdots, n,$  such that $(\hat{x}_{iS},x_{ip})\in \Vsp.$

 Let $L_S$ be the number lattice generated by $ \hat{x}_{iS},i=1, \cdots ,n.$
By the definition of the $\lrar $ operation, $L_S\subseteq \Vsp\lrar L_P$ 
and since we also have $ \hat{x}_{iP},i=1, \cdots , n,$ as basis vectors for $L_P,$
we must have $L _P\subseteq \Vsp\lrar L_S.$ 
Suppose a nontrivial linear combination of the vectors $ \hat{x}_{iS}$ lies in $\Vs.$
Then, since $\Vsp \lrar \Vs=\Vp,$ the same linear combination of 
 $ \hat{x}_{iP},i=1, \cdots ,n,$
would lie in $\Vp.$ 
But $\Vp\cap L_P=\0_P$ so that the nontrivial linear combination yields a zero vector,  
a contradiction.
Therefore, $L_S\cap \Vs=\0_S$ and $ \hat{x}_{iS},i=1, \cdots , n,$ 
are basis vectors for $L_S.$
%
%
%
%
%

It is clear that $\Ks\supseteq L_S+\Vs.$
To see the reverse containment, 
 let 
 $\tilde{x}_S\in \Ks.$ \\There exists $\tilde{x}_P\in \Vsp\lrar \Ks=\Kp,$ 
such that $(\tilde{x}_S,\tilde{x}_P)\in \Vsp.$
%
%
%
Since $\tilde{x}_P\in \Kp=\Vp+L_P,$ there exist
$x'_P,x"_P,$ such that $\tilde{x}_P=x'_P+x"_P, x'_P\in \Vp,
x"_P\in L_P.$ 
Let $x'_S,x"_S,$ be such that 
$(x'_S,x'_P),(x"_S,x"_P)\in \Vsp.$ 
We know that there exists $\hat{x}_S\in L_S,$ such that 
$(\hat{x}_S,x"_P)\in \Vsp.$\\ 
Therefore, we must have 
$(\hat{x}_S- x"_S)\in \Vsp\times S\subseteq \Vsp\lrar \Vp=\Vs,$
so that, $x"_S\in L_S+\Vs.$\\
Further, $ x'_S\in \Vsp \lrar \Vp=\Vs.$ Therefore $x'_S+x"_S\in L_S+\Vs.$
\\
We have that $(x'_S+x"_S,\tilde{x}_P),(\tilde{x}_S,\tilde{x}_P)\in \Vsp,$
so that 
$\tilde{x}_S\in (x'_S+x"_S)+\Vsp\times S,$ \\
i.e., $\tilde{x}_S\in L_S+\Vs+ 
\Vsp\times S= L_S+\Vs.$


 3. 
(a)  Let  $x_P\in L_P, $ with $x_P\ne 0_P.$ 
Since $\Vsp\lrar L'_S\supseteq L_P,$
 there exists $x_S\in L'_S $ such
that $(x_S,x_P)\in \Vsp.$
The uniqueness of $x_S$ follows from the fact that $L'_S\cap (\Vsp\times S)\subseteq L'_S\cap \Vs=\0_S.$
Next, if $x_S=0_S,$ since $\Vsp\lrar \Vs=\Vp,$  we must have $x_P\in \Vp.$
But $L_P\cap \Vp=0_P$ so this means $x_P= 0_P,$ a contradiction.
We conclude that $x_S\ne 0_S.$
 
(b) Next, let $\hat{x}_{iS}, i=1, \cdots, n,$
be as in the hypothesis. Since $L'_S\subseteq \Vsp\lrar L_P,$ and since $x_{1P}, \cdots, x_{nP},$ 
generate $L_P.$ the vectors $\hat{x}_{iS}$ generate $L'_S.$
Since $\Vsp\lrar \Vs=\Vp,$
if a nontrivial linear combination of $\hat{x}_{iS}, i=1, \cdots, n,$ lies in $\Vs,$ the same combination
of $x_{1P}, \cdots, x_{nP}$ will lie in $\Vp.$
But $L_P\cap \Vp=\0_P$ and $x_{1P}, \cdots, x_{nP}$ are given to be independent.
Hence $\hat{x}_{iS}, i=1, \cdots, n,$ are independent
and $B_S$ is a basis matrix for $L'_S.$

(c) The proof is as in that of part 2 of the theorem.
\end{proof}
\begin{definition}
Let $\Ks,\Kp,\Vsp$ be such that $\Ks=\Vsp\lrar \Vp, \Kp=\Vsp\lrar \Vs.$
Let $L_S\subseteq \Ks, L_P\subseteq \Kp$ be such that for each $x_P\in L_P$
there exists a unique $x_S\in L_S$ such that $(x_S,x_P)\in \Vsp$ and 
for each $x_S\in L_S$
there exists a unique $x_P\in L_P$ such that $(x_S,x_P)\in \Vsp.$ 
We then say that $\Ks,\K_P$ are {\bf invertibly linked} through $\Vsp$
and so are $L_S,L_P.$

\end{definition}

\subsubsection{Implicit Duality}
\label{ssec:vspaceresults}
Theorem \ref{thm:idt}, is from \cite{HNarayanan1997}.
In the present paper, it  is used to derive results on self duality and short vectors of duals.
The proof in this subsection is new. It  generalizes naturally 
to the matroid case \cite{STHN2014}.


From the definition of matched  composition, when $\mathcal{K}_{SP},\mathcal{K}_{PQ}$ are generalized number lattices and $(S,P,Q)$ are disjoint, 
\begin{align*}
(\mathcal{K}_{SP}\leftrightarrow \mathcal{K}_{PQ}) = (\mathcal{K}_{SP}+ \mathcal{K}_{(-P)Q}) \times (S\cup Q)
\end{align*}
and also equal to 
\begin{align*}
 (\mathcal{K}_{SP}\cap \mathcal{K}_{PQ})\circ (S\cup Q).
\end{align*}
Similarly from the definition of skewed composition,
\begin{align*}
 (\mathcal{K}_{SP} \rightleftharpoons \mathcal{K}_{PQ}) = (\mathcal{K}_{SP}+ \mathcal{K}_{PQ}) \times (S\cup Q)
\end{align*}
and also equal to 
\begin{align*}
 (\mathcal{K}_{SP}\cap \mathcal{K}_{(-P)Q})\circ (S\cup Q).
\end{align*}
Hence we have 
\begin{align*}
 (\mathcal{K}_{SP}\leftrightarrow \mathcal{K}_{PQ})^d 
& = [(\mathcal{K}_{SP} + \mathcal{K}_{(-P)Q}) \times (S\cup Q)]^d \\
& = (\mathcal{K}_{SP} + \mathcal{K}_{(-P)Q})^d \circ (S\cup Q) \\
& =(\mathcal{K}_{SP}^d \cap \mathcal{K}_{(-P)Q}^d) \circ(S\cup Q)\\
& = \ \ \  \mathcal{K}_{SP}^d \rightleftharpoons \mathcal{K}_{PQ}^d 
\end{align*}
In particular,
\begin{align*}
 (\mathcal{K}_{SP}\leftrightarrow \mathcal{K}_{P})^d &\equaln   
\mathcal{K}_{SP}^d \leftrightarrow \mathcal{K}_{P}^d
\end{align*}
since $\mathcal{K}_P = \mathcal{K}_{(-P)}$.
To summarize, noting further that if $\Vsp$ is a vector space $\Vsp^d=\Vsp^{\perp},$
\begin{theorem}
\label{thm:idt}
Let $\mathcal{K}_{SP},\mathcal{K}_{PQ}$ be generalized number lattices. Then
$(\mathcal{K}_{SP}\leftrightarrow \mathcal{K}_{PQ})^d 
\ \equaln\ (\mathcal{K}_{SP}^d \rightleftharpoons \mathcal{K}_{PQ}^d) 
.$ 

In particular,
\begin{itemize}
\item $(\mathcal{K}_{SP}\leftrightarrow \mathcal{K}_{P})^d \ \equaln\ (\mathcal{K}_{SP}^d \leftrightarrow \mathcal{K}_{P}^d)
.$
\item $(\mathcal{V}_{SP}\leftrightarrow \mathcal{K}_{P})^d \ \equaln\ (\mathcal{V}_{SP}^{\perp} \leftrightarrow \mathcal{K}_{P}^d)
.$

\end{itemize}
\end{theorem}
The above result will be referred to as the \nw{implicit duality theorem} (IDT)
for generalized number lattices.

Observe that the dot-cross duality is also a consequence of the implicit duality theorem since
\begin{align*}
(\mathcal{K}_{SP}\times S )^d = (\mathcal{K}_{SP}\leftrightarrow \0_S )^d = \mathcal{K}_{SP}^d \leftrightarrow \mathcal{F}_S = \mathcal{K}_{SP}^d\circ S.
\end{align*}

By Corollary \ref{cor:vsnldual}, we know that the dual of a
full dimensional number lattice is also full dimensional.
When $\mathcal{K}_{SP},\mathcal{K}_{PQ}$ are full dimensional number lattices,
it is easily verified that $\mathcal{K}_{SP}\circ S, \mathcal{K}_{SP}\times S$
$
(= (\mathcal{K}^d_{SP}\circ S)^d)$
are full dimensional number lattices on $S$ and  $\mathcal{K}_{SP}+\mathcal{K}_{PQ},$
$\mathcal{K}^d_{SP}+\mathcal{K}^d_{PQ},$
$\mathcal{K}_{SP}\cap \mathcal{K}_{PQ}$ $(= (\mathcal{K}^d_{SP}+ \mathcal{K}^d_{PQ})^d)$ are full dimensional number lattices on
$S\uplus P\uplus Q.$\\ 
From the definition of matched and skewed composition, we therefore have
\begin{theorem}
\label{thm:idt2}
Let ${L}_{SP},{L}_{PQ}$ be full dimensional number lattices. Then
${L}_{SP}\leftrightarrow {L}_{PQ}$ is a full dimwnsional number lattice 
and therefore, 
$({L}_{SP}\leftrightarrow {L}_{PQ})^d$
is a full dimensional number lattice and is equal to \\
${L}_{SP}^d \rightleftharpoons {L}_{PQ}^d.$ 
In particular,
$({L}_{SP}\leftrightarrow {L}_{P})^d \ \equaln\ {L}_{SP}^d \leftrightarrow {L}_{P}^d
.$
\end{theorem}

The situation where $\V_{SP}\lrar \lnew_{PQ}$ is a full dimensional
number lattice is also of some interest and is dealt with
in the following theorem.

\begin{theorem}
Let $\V_{SP}$ be a vector space such that $\V_{SP}\circ S=\F_S,$
$\V_{SP}\times S=\0_S,$ and let $\lnew_{PQ}$ be a number lattice 
such that  
$\lnew_{PQ}\times Q$ is full dimensional and 
$span(\lnew_{PQ}\circ P)= \V_{SP}\circ P.$
Then
$\V_{SP}\lrar \lnew_{PQ}$ is a full dimensional number lattice 
and therefore so is $({\V}_{SP}\leftrightarrow {L}_{PQ})^d \ \equaln\ {\V}_{SP}^d \leftrightarrow {L}_{PQ}^d={\V}_{SP}^{\perp} \leftrightarrow {L}_{PQ}^d
.$
\end{theorem}
\begin{proof}
Let rows of $(B_S\vdots B_P),(C_P\vdots C_Q)$ be bases for
vector space $\V_{SP}$ and number lattice $L_{PQ}$
respectively. Since $\Vsp\times S=\0_S,$
there exists a matrix $N,$
such that $B_S=B_PN.$\\
By Theorem \ref{thm:vsnl}, since 
$span(\lnew_{PQ}\circ P)= \V_{SP}\circ P,$
we must have that 
$\Vsp\lrar \Lpq$ is a number lattice
for which  $(C_PN\vdots C_Q)$ is a basis matrix.
We will show that the columns of $(C_PN\vdots C_Q)$
are linearly independent.
Let  $(C"_Q)$ be  a basis matrix for $\lnew_{PQ}\times Q.$
Therefore there exists a generating matrix
(see Subsubsection \ref{subsec:visible})
\begin{align}
\label{eqn:k}
\ppmatrix{C_{SQ}} \equiv
\bbmatrix{
        C'_PN &  C'_Q \\ 0 & C"_Q
} 
\end{align}
for the number lattice
 $\Vsp\lrar \Lpq.$ 
Now $\lnew_{PQ}\times Q$ is full dimensional.
So columns of $(C"_Q)$ are linearly independent.
The matrix $(C_{SQ})$  is row equivalent to $(C_PN\vdots C_Q).$
Since 
$span(\lnew_{PQ}\circ P)= \V_{SP}\circ P,$
$(C_P),(B_P)$ are row equivalent
and therefore so also $(C_PN), (B_PN).$
Now $\V_{SP}\circ S = \F_S.$ Therefore, columns of $(B_S)=(B_PN)$
are linearly independent and therefore also those of $(C_PN).$
It follows that columns of $(C'_PN)$
are linearly independent and therefore also
columns of $(C_{SQ}).$
Since $(C_{SQ}),$ $(C_PN\vdots C_Q)$ are row equivalent,
the columns of $(C_PN\vdots C_Q)$  are also linearly independent so that
$\Vsp\lrar \Lpq$ is a full dimensional number lattice.
\end{proof}


Implicit duality theorem and its applications are dealt with in detail in \cite{HNarayanan1986a},  \cite{narayanan1987topological}, \cite{HNarayanan}, \cite{HNarayanan1997}.
More  recently, there has been  interest in the applications of this result in \cite{Forney2004},
in the context of `Pontryagin duality'.
The above proof is similar to the  one in  \cite{HNarayanan}. Versions for other contexts
are available in
\cite{HNarayanan1997}, for operators in \cite{HN2000}, \cite{HN2000a}, for matroids in \cite{STHN2014}.

\subsubsection{Self duality}
Let $Q\equiv \{e_1, \cdots e_m\}.$ Let  $p(\cdot )$ be a permutation of the elements of $Q$ such that $p^2(\cdot )=\  identity .$\\ 
Define, for a vector $f_{e_1, \cdots , e_m}$ on $Q,$ the vector 
$f_{p(e_1), \cdots , p(e_m)}$ as follows:

$f_{p(e_1), \cdots , p(e_m)}(p(e_i)) \equiv f_{e_1, \cdots , e_m}(e_i),\ i=1, \cdots , m.$
\\
Note that $\langle f_{p(e_1), \cdots , p(e_m)}, g_{e_1, \cdots , e_m} \rangle 
= \Sigma _i f_{p(e_1), \cdots , p(e_m)}(p(e_i))g_{e_1, \cdots , e_m} (e_i).$\\
Let $\K_{pQ}\equiv \{f_{p(e_1), \cdots , p(e_m)},
f_{e_1, \cdots , e_m}
\in \K_{Q}\}.$  

We say the generalized number lattice $\K_Q$ is \nw{self dual relative to the permutation $p(\cdot )$} iff $\K^d_Q= \K_{pQ}.$

We now have, as a corollary to Theorem \ref{thm:idt},
a way of constructing new self dual (generalized) number lattices from old.

\begin{corollary}
\label{cor:selfdual}
Let $p(\cdot)$ be a permutation of $S\uplus P$ such that $p^2(\cdot)= \ identity,$ $p(S)=S$ and 
$p(P)=P.$
\begin{enumerate}
\item Let ${\K}_{SP},{\K}_{P}$ be self dual generalized number lattices relative to $p(\cdot ), 
p|_P(\cdot ),$ respectively.
Then so is ${\K}_{SP}\leftrightarrow {\K}_{P}$ relative to $p|_S(\cdot ).$

\item Let ${L}_{SP},{L}_{P}$ be self dual number lattices relative to $p(\cdot ), 
p|_P(\cdot ),$ respectively.
Then so is ${L}_{SP}\leftrightarrow {L}_{P}$ relative to $p|_S(\cdot ).$
\item  
Let ${\V}_{SP}$ be a  vector space  and let  ${L}_{P}$ be a number lattice 
such that ${\V}_{SP}\lrar {L}_{P}$ is a number lattice.
Further, let ${\V}_{SP}$ be self dual relative to the the permutation $p(\cdot )$
and let  ${L}_{P}$ be self dual relative to $p|_P(\cdot ).$
Then ${\V}_{SP}\lrar {L}_{P}$ is a self dual number lattice
relative to $p|_S(\cdot ).$

\end{enumerate}
\end{corollary}
\begin{proof}
\begin{enumerate}
\item Let $\K_S\equiv {\K}_{SP}\leftrightarrow {\K}_{P}.$
We then have $$\K^d_S= ({\K}_{SP}\leftrightarrow {\K}_{P})^d= {\K}^d_{SP}\leftrightarrow {\K}^d_{P}={\K}_{p(SP)}\leftrightarrow {\K}_{pP}= \K_{pS}. $$ 
\item From part 1 above, ${L}_{SP}\leftrightarrow {L}_{P}$
is a self dual generalized number lattice relative to $p|_S(\cdot).$

We have  that ${L}_{SP},{L}_{P}$ are full dimensional number lattices
and so also (by Theorem \ref{thm:idt}) is ${L}_{SP}\leftrightarrow {L}_{P}.$
Therefore $({L}_{SP}\leftrightarrow {L}_{P})^d$ cannot contain a vector space.
It is therefore a self dual number lattice.

\item From part 1 above, it is clear that ${\V}_{SP}\lrar {L}_{P}$ 
is a self dual generalized number lattice relative to $p|_S(\cdot).$ 
We know that $\K_S\equiv {\V}_{SP}\lrar {L}_{P}$
is  a number lattice, and therefore $\K^d_S= \K_{pS}$ 
cannot contain a nontrivial vector space.
Therefore ${\V}_{SP}\lrar {L}_{P}$
is a self dual number lattice.
\end{enumerate}
\end{proof}
Corollary \ref{cor:selfdual} suggests some simple ways of 
constructing new self dual lattices from old.

1. 
Suppose $S_1\uplus P_1, \cdots , S_k\uplus P_k$
are identical copies of the set $S_1\uplus P_1.$
Let $S\equiv S_1\uplus \cdots \uplus S_k$
 and let $P\equiv P_1\uplus \cdots \uplus P_k.$
We will suppose that the sets $P$ and $S_1$ are `small'
and that we have self dual number lattices
$L_{S_1P_1}$ and $L_P.$
Let $L_{S_iP_i}, i= 2, \cdots , k$
be copies of the number lattice $L_{S_1P_1}.$
It is clear that $L_{S_1P_1}\oplus \cdots \oplus L_{S_kP_k}$
is self dual, so that
$L_{S_1P_1}\oplus \cdots \oplus L_{S_kP_k}\lrar L_P $ is self dual.
If now we define a permutation $p(\cdot)$ on $S\uplus P,$
such that $p^2(\cdot)= identity,$ $p(S_i)=S_i,p(P_i)=P_i, i=1, \cdots ,k,$
and $L_{S_iP_i}, i= 1, \cdots , k$ and $L_P$ are self dual relative 
to the permutation $p(\cdot),$
so would $L_{S_1P_1}\oplus \cdots \oplus L_{S_kP_k}\lrar L_P $ be.

2. We describe a way of building self dual number lattices analogous
to the manner in which reciprocal electrical networks are built by connecting
smaller reciprocal networks (see for instance \cite{narayanan2002some}
).

Let $\V_Z, \V_{Z'}$  be as in Example \ref{eg:port}.\\
Let $p(\cdot)$ be the permutation on $Z\uplus Z'$ defined by
$p(e_i)\equiv e_i', p(e_i')\equiv e_i, i=1, \cdots , m.$
Let $\V_{ZZ'}\equiv \V_Z\oplus \V^{\perp}_{Z'}.$\\
Clearly $\V^d_{ZZ'}=\V^{\perp}_{ZZ'}= (\V_Z\oplus \V^{\perp}_{Z'})^{\perp}= \V^{\perp}_{Z}\oplus \V_{Z'}=\V_{p(ZZ')}.$
Therefore, $\V_{ZZ'}$ is self dual relative to the permutation 
$p(\cdot ).$\\
As in Example \ref{eg:port}, 
 let $S\equiv \{e_1, \cdots , e_n\}, n\leq m,$
$P\equiv \{e_{n+1}, \cdots , e_m\}, $
$S'\equiv \{e'_1, \cdots , e'_n\}, $
$P'\equiv \{e'_{n+1}, \cdots , e'_m\}. $
Let $Z\equiv S\uplus P, Z'\equiv  S'\uplus P'.$
Let  $\V_Z$ be the coboundary space of a graph $\G,$ and let 
 the edges in
$P$ be such that they can be included in a tree as well as in the complement of a tree
(i.e., $P$ contains no loops or cutsets of $\G$).

%

If $\lnew_{SS'}$ is a number lattice on $S\uplus S',$ then by Theorem \ref{thm:vsnl},
 as we saw in  Example \ref{eg:port}, 
 $\V_{ZZ'}\lrar \lnew_{SS'}$ is also a number lattice.
 
Suppose $S_1, \cdots , S_k$
are identical copies of the set $S_1.$
Let $S\equiv S_1\uplus \cdots \uplus S_k.$
We will suppose that the set  $S_1$ is  `small'
and that we have a self dual number lattice
$L_{S_1S_1'}.$ 
Let $L_{S_iS_i'}, i= 2, \cdots , k$
be copies of the number lattice $L_{S_1S_1'}.$
It is clear that $L_{SS'}\equiv L_{S_1S_1'}\oplus \cdots \oplus L_{S_kS_k'}$
is self dual, so that
$\V_{ZZ'}\lrar L_{SS'} $ is self dual relative to the permutation $p(\cdot).$


\section{Number lattices with a specified partition of the underlying set}
\label{sec:link}
In this section we compare composition of `maps' with matched and skewed composition
of the corresponding objects, viz., `linkages'.
A {\bf linkage} is a generalized number latice $\mathcal{K}_{S}$ with a partition $\{S_1,\cdots, S_k\}$ of $S,$  specified.
It will be denoted by $\mathcal{K}_{S_1\cdots S_k}.$
When the linkage is a vector space 
it will be referred to as  a {\bf vs-linkage} $\mathcal{V}_{S_1\cdots S_k}.$ 

Observe that a map $x^TK=y^T$ can be regarded as
a vs-linkage $\mathcal{V}_{SP},$ with a typical vector $(x^T,y^T)$,  whose basis is the set of rows of $\bbmatrix{
 I & K},$ the first set of columns of the matrix being indexed by $S$ and the second set by $P.$


The matched composition operation `$\leftrightarrow $' can be regarded as a generalization of the composition operation for ordinary maps. 

Thus if $x^TK=y^T,y^TP=z^T,$ it can be seen that $\mathcal{V}_{SP}\leftrightarrow\mathcal{V}_{PQ}  $ 
would have as the basis, the rows of  $\bbmatrix{
I & KP },$ the first set of columns of the matrix being indexed by $S$ and the second set by $Q.$

\begin{remark}
\begin{enumerate}
\item When we compose matrices, the order is important. In the case of linkages, since vectors
are indexed by subsets, order does not matter. Thus $\Ksp=\Kps$ and $\Ksp\lrar \Kpq\equaln \Kpq\lrar \Kps.$

\item
$\mathcal{V}_{SP} $ mimics the map $K$ by linking 
 row vectors  $x^T,y^T$ which correspond to each other when $x^T$ is operated upon by $K,$ into 
 the long vector $(x^T,y^T).$
The collection of all such vectors forms a vector space which contains all the information that 
the map $K$ and $K^{-1}$ (if it exists)  contain and represents both of them implicitly.

But there are two essential differences: a map takes vectors in  
$\Fx$ to vectors in $\Fy,$ but takes the vector $0 _S$ only to the vector $0 _P.$
In the case of vs-linkages only a subspace of $\Fx$  may be involved as 
$\mathcal{V}_{SP}\circ S$ and the vector 
$0 _S$ may be linked to a nontrivial subspace of $\Fy.$
\end{enumerate}
\end{remark}
\subsection{Expressions of linkages}
The `$\lrar $' operation is not inherently associative.
For instance, consider the expression 
$$ \K_{BC}\lrar \0_{AB}\lrar \0_{BQ}.$$
If we treat this expression to be 
$\K_{BC}\lrar (\0_{AB}\lrar \0_{BQ}),$
it will reduce to $\K_{BC}\lrar\0_{AQ}=\K_{BC}\oplus \0_{AQ}.$
\\
If, however, we treat it to be $ (\K_{BC}\lrar \0_{AB})\lrar \0_{BQ},$
it will reduce to $\K_{BC}\times C\oplus \0_{ABQ}.$

The two reduced expressions are clearly different if $\K_{BC}\circ C\ne \K_{BC}\times C.$

The problem here is that the index set $B$ occurs more than twice.

Next consider the expression 
$$(\K_{AB}\lrar \K_{BC})\lrar(\K_{CA}\lrar \K_{PQ}).$$

 Here, no index set occurs more than twice, but if we regard the expression 
as $(\K_{AB}\lrar \K_{BC}\lrar\K_{CA})\lrar \Kpq,$
we get a subexpression $\K_{AB}\lrar \K_{BC}\lrar\K_{CA},$ where the index set becomes null
and the `$\lrar$' operation is not defined for such a situation.
We will call such expressions `null' expressions.

Let us only look at expressions 
where no index set occurs twice and which do not contain null subexpressions.

In the case of such expressions,
the brackets can be got rid of without  ambiguity.
Consider for instance the expression 
$$\K_{A_1B_1}\lrar \K_{A_2B_2}\lrar 
\K_{A_1B_2},$$ 
where all the $A_i,B_j$ are mutually disjoint. It is clear that this expression
has a unique meaning (evaluation), namely, the collection of all $(f_{A_2},g_{B_1})$ such that
there exist $ h_{A_1},k_{B_2}$ with $$(h_{A_1},g_{B_1})\in \K_{A_1B_1},\ \ (f_{A_2},k_{B_2})\in \K_{A_2B_2},\ \ (h_{A_1},k_{B_2})\in \K_{A_1B_2}.$$ Essentially the 
terms of the kind $l_{D_i}$ survive if they occur only once and they get coupled through the terms which occur twice. These latter dont survive in the final expression. Such an interpretation is valid even if we are dealing with 
signed index sets of the kind ${- A_i}.$ (We remind the reader that
$\K_{(-A)B}$ is the space of all vectors $(-f_A,g_B)$ where
$(f_A,g_B)$ belongs to $\K_{AB}.$ ) If either ${- A_i},$
or ${ A_i}$ occurs, it counts as one occurrence of the index set $  { A_i}.$
We state this result as a theorem but omit the routine proof.

Let us define a \nw{regular  expression of linkages} to be one of the form $\lrar _{A_i,B_j,C_k,\cdots}(\K_{(\pm A_i)(\pm B_     j)(\pm C_k)\cdots}),$ where the index sets $\pm A_i, \pm  B_j,\pm C_k, \cdots$ are all mutually disjoint, 
where no index set occurs more than twice, and that contains no null subexpressions.

We thus have,
\begin{theorem}
\label{thm:notmorethantwice}
A regular expression of linkages has a unique evaluation as a linkage
$\K_{\cdots , {\pm D_i} ,\cdots }$ where $D_i$ are the index sets that occur
only once in the expression.
\end{theorem}
\begin{remark}
The condition `no null subexpressions' can be relaxed,
if we permit a special `object' $\K_{\emptyset}$ such that\\ 
1. $\K_{\emptyset}=\Ks\lrar \K'_S,$ for any set $S$ and any
$\Ks,\K'_S,$ and\\  
2. $\K_{\emptyset}\lrar \Kp=\Kp,$ for any $\Kp.$
\end{remark}
A regular expression of linkages is best represented by means of a diagram,
where individual linkages are represented by nodes, with edges
corresponding to index sets common to two linkages.

This diagram would look different from the usual diagram for maps.
In the case of maps, the usual diagram would simply be a directed line 
going through a sequence of nodes, say $T_i, i=1, \cdots, k,$
corresponding to the composed map $T_1\times \cdots \times T_k.$
The diagram of the dual would correspond to the composed map $T^T_k\times \cdots \times T^T_1$
and the line would be directed in the opposite direction.

In the case of linkages there is no need of a directed line for describing
simple matched composition, the index set specification being adequate.
We need only distinguish `$\lrar$' from `$\rightleftharpoons$'.
But $\Ksp\rightleftharpoons\Kpq$ can be treated as $\Ksp\lrar \K_{(-P)Q}.$
The dual of $\Ksp\lrar \Kpq$ would be $\Ksp^d\lrar  \K^d_{(-P)Q}.$

If the maps were regarded as linkages, there would be undirected bold  edges
corresponding to `$\lrar$' and undirected but, say, dotted edges corresponding to `$\rightleftharpoons$'.
If a particular edge is bold in the diagram of the primal expression,
 it would be dotted in the diagram of the dual expression.

Linkages permit much greater flexibility
while losing none of the essential properties of dualization.
Linkages can simultaneously relate vectors over several sets:
consider, for instance, a regular expression of the kind
 $\K_{SPQ}\lrar \K_{STW}\rightleftharpoons \Kp.$
There is no natural analogue of this expression in the case of maps.
But the dual is still simple :
$(\K_{SPQ}\lrar \K_{STW}\rightleftharpoons \Kp)^d= (\K_{SPQ}\lrar \K_{STW}\lrar \K_{(-P)})^d= \K^d_{(-S)PQ}\lrar \K^d_{STW}\lrar \K^d_P.$

Directed lines are needed corresponding to situations where 
the implicit inversion theorem (Theorem \ref{thm:inversevsnl1}) is applicable.
\ref{sec:diagram} contains a description of such diagrams and their
essential invariance under dualization.

\section{Number lattices invertibly linked through  regular vector spaces} 
\label{sec:uniquelinkage}
\subsection{Regular vector spaces}
We saw in Subsection \ref{subsec:impinv},  that when generalized number lattices 
are invertibly linked as in Theorem \ref{thm:inversevsnl2}, 
a vector in the number lattice part of each is uniquely linked to 
a vector in that of the other.
We show in this section that when the vector space linking the two generalized number lattices
has special properties (such as having totally unimodular 
representative matrices), the length of 
a vector in one number lattice bounds the length of the 
corresponding vector in the other. In particular
`short' vectors of one are linked to `short' vectors of the other.
The incidence matrix of a graph is totally unimodular.
The ideas of this section are immediately applicable when the linking 
 vector spaces are spanned by or are orthogonal to the rows of the incidence matrix of a graph.

We need some preliminary definitions and  lemmas.

The \nw{column matroid} of a matrix on column set $S$ is the family of independent
subsets of $S.$ 
A set of columns is linearly independent in a matrix iff it is linearly independent in any other matrix that is row equivalent to it. Therefore, any two representative matrices 
of a vector space $\V_S$ have the same column matroid.
We call this the matroid of $\V_S$ and denote  it by $\M(\V_S).$
A maximal independent subset of columns of a representative matrix 
of $\V_S$ is called a  \nw{base} of the matroid $\M(\V_S).$

A matrix is said to be \nw{totally unimodular} iff
all its subdeterminants are $0,\pm{1}.$ 
A vector space which has a totally unimodular representative matrix
is called a \nw{regular} vector space.
We remind the reader  that a standard representative matrix is a representative matrix which can be put in the form $(I\ |\ K)$ after column
permutation.
It is clear that corresponding to every base $B$ of $\M(\V_S),$
there is a standard representative matrix of $\V_S$
with the unit submatrix corresponding to $B.$ 
We will call this the standard representative matrix of $\V_S$ 
with respect to $B.$

\begin{definition}
Let $C_S$ be a representative matrix of $\Vs.$
We say a base $B$ of $\M(\Vs)$ is picked according to a priority
sequence $(S_1, \cdots , S_k), S_i\subseteq S, S_i\cap S_j= \emptyset , i\ne  j,$ iff
$B$ is a  maximal independent set of columns $C_S$ with
elements from $S_1, \cdots, S_i, i< k,$ picked before elements from $S_{i+1}.$
The standard representative matrix built with respect to a base of $\M(\Vs)$ picked according to a priority
sequence $(S_1, \cdots , S_k)$ would simply be referred to as being built according to the priority
sequence $(S_1, \cdots , S_k).$
\end{definition}

The following lemma is well known and its routine proof is omitted (\cite{tutte}).
\begin{lemma}
\label{lem:standuni}
Every standard representative matrix of a regular vector space
is totally unimodular.
If $(I\ |\ K),$ is a standard representative matrix of $\Vs,$
then $(-K^T\ |\ I)$ is a standard representative matrix of $\Vs^{\perp}.$
Therefore, if $\Vs$ is regular, so is $\Vs^{\perp}.$

\end{lemma}

The following is an easy consequence
\begin{lemma}
\label{lem:basetobase}
Let a standard representative matrix of a regular vector space $\Vs$
with respect to a base $B_1$ of $\M(\V_S)$ be given.
Let $|B_1|=dim(\Vs)=m.$
Then we can compute the standard representative matrix of $\Vs$ with respect to any
other base $B_2$ of $\M(\V_S),$ which could be  specified in terms of a priority sequence of elements from $S,$
in $O(m^2|S|)$ time with all
the numbers occurring during intermediate computations being equal to
$0, \pm 1.$
%
%
\end{lemma}

The proof is given in \ref{sec:basetobase}.

The following simple result is useful for relating length 
of vectors in generalized number lattices linked through regular vector spaces.

\begin{lemma}
\label{lem:normBT}
Let $Q\equiv (Q_T\ |\ Q_B)$ be a totally unimodular matrix with $Q_B$ 
as a column permutation of the identity matrix. 
Let $x$ be linearly dependent on the rows of $Q.$
Let $x_T,x_B$ denote the vectors $x|_T,x|_B$ respectively.
Then
$$||x_T||\leq
 \sqrt{(|T|\times|B|)}\times ||x_B||.$$
\end{lemma}
\begin{proof}
We have $x_T^T= x_B^T Q_T,$ so that
$||x_T||^2=x_T^Tx_T=x_B^TQ_TQ_T^Tx_B.$
Now the maximum magnitude of an entry in $Q_TQ_T^T$ must be less or equal
to $|T|$ since entries of $Q_T $ are  $0,\pm{1}.$
It follows that  the maximum magnitude eigenvalue $\lambda_{max}$ of $Q_TQ_T^T$
must satisfy $|\lambda_{max}|\leq 
|T|\times |B|.$\\
Therefore, using Cauchy$-$Schwarz inequality,\\ $||x_T||^2=  ||x_B^TQ_TQ_T^Tx_B||\leq ||x_B^T||\times ||Q_TQ_T^Tx_B||\leq |\lambda_{max}| ||x_B||^2\leq |T|\times |B|\times||x_B||^2.$
The result follows.
\end{proof}
The next result states that a vector,  of the kind in the above lemma, that
is linked to a given vector, is easy to find.
\begin{lemma}
\label{lem:normSP1}
Let $\V_{SP}$ be a regular vector space on $S\uplus P.$

Let $x_P\in \V_{SP} \circ P.$ Then there exists $(x_S,x_P)\in \V_{SP}$
such that $||x_S||\leq \sqrt{|S|\times|P|}||x_P||.$
Moreover, if  a standard representative matrix is available for
$\V_{SP},$ 
 built  according to the priority sequence $(P,S),$ then $x_S$ can be computed in $\tilde{O}(mnlog(M))$ time, where $m=r(\Vsp\circ P),$
$log(M)$ is the maximum bit size of an entry in $x_P.$
and $n=|S|-r(\Vsp\times S)=r(\Vsp^{\perp}\circ S).$
The maximum bit size of an entry in $x_S$ is $O(log(mM)).$ 
\end{lemma}
\begin{proof}
The standard representative matrix with respect the base $B$ of $\M(\Vsp),$
since it is  picked according to the priority sequence $(P,S),$ must have the form
\begin{align}
\label{eqn:specialrep1}
{Q_{SP}}\equiv
\ppmatrix{
        Q_{1S} & \vdots\vdots  & \0_{1P}\\
        Q_{2S} & \vdots\vdots  &Q_{2P}} 
\equiv
\ppmatrix{
        I_{1S_1}&\vdots  Q_{1S_2}&\vdots  \vdots&  \0_{1P_1}& \vdots  & \0_{1P_2}\\
        \0_{2S_1}&\vdots  Q_{2S_2}&\vdots  \vdots& I_{2P_1}& \vdots  &  Q_{2P_2}}
\end{align}
We know, by Lemma \ref{lem:standuni}, that this matrix is totally unimodular so that all its entries are
$0,\pm{1}.$
Now let
$x_P\in \V_{SP}\circ P.$ We know, from the discussion in Subsection \ref{subsec:visible},
 that
$x_P^T=\lambda^T(I_{2P_1}\ |\  Q_{2P_2}).$ 
Therefore, $\lambda^T = x_{P_1}^T.$
Let us pick $(x_S,x_P)$ as
\begin{align}
\label{eqn:specialrep1}
(x_S^T\ |\ x_P^T) = (\0^T\ | \ x_{P_1}^T) 
\ppmatrix{
        Q_{1S} & \vdots\vdots  & \0_{1P}\\
        Q_{2S} & \vdots\vdots  &Q_{2P}} 
=
(\0^T\ | \ x_{P_1}^T) \ppmatrix{
        I_{1S_1}&\vdots  Q_{1S_2}&\vdots  \vdots&  \0_{1P_1}& \vdots  & \0_{1P_2}\\
        \0_{2S_1}&\vdots  Q_{2S_2}&\vdots  \vdots& I_{2P_1}& \vdots  &  Q_{2P_2}}
\end{align}
This yields $x_S^T= x_{P_1}^T(\0_{2S_1}\ \vdots \ | Q_{2S_2}).$

By Lemma \ref{lem:normBT}
, we know that\\
$||x_S||\leq
 \sqrt{(|S_2|\times|P_1|)}\times ||x_{P_1}||\leq \sqrt{(|S|\times|P|)}\times ||x_{P}||.$

It is clear that $x_S$ can be computed in $O((|P_1|\times |S_2|)log(M))$ 
time, where $log(M)$ is the maximum bit size of an entry in $x_P.$
Therefore, since $m=r(\Vsp\circ P)=|P_1|, n=|S_2|=|S|-r(\Vsp\times S)=r(\Vsp^{\perp}\circ S)
,$ the computation
can be done in  $O(mnlog(M))$ time.

Since $x_S^T= x_{P_1}^T(\0_{2S_1}\ \vdots \  Q_{2S_2}),$
 $Q_{2S_2}$ has $0,\pm 1$ entries and $m$ rows, and maximum bit size of an entry in $x_P$ is $O(log(M)),$
we must have that 
the maximum bit size of an entry in $x_S$ is 
$O(log(mM)).$

\end{proof}
%
%
%
%
%

\subsection{Approximate successive minima bases}
An important problem associated with number lattices is the computation 
of approximately shortest vectors of various kinds in the lattice. Even the 
approximate (with respect to a factor polyomial in the dimension of the lattice) versions of these problems are known to be hard (\cite{khot},\cite{arora}).
In this subsection, we examine what can be inferred from short vectors of
a number lattice about short vectors of a linked number lattice, when 
the linking is through a regular vector space.
\begin{definition}
Let $L_S$  be a number lattice with basis matrix $B_S$ of rank $n.$
For $i=1, \cdots m,$   we define the $i^{th}$  successive minimum as
$\lambda _i(L_S)= inf \{r, dim(span(L_S\cap \overline{B}(0,r)))\geq i\},$
where $\overline{B}(0,r)\equiv \{x\in \Re^{|S|},||x||\leq r\}$ is the closed 
ball of radius $r$ around $0_S.$
\end{definition}
It is clear from the definition of the succesive minima that 
if $ i\leq j$ then $\lambda _i(L_S)\leq \lambda _j(L_S), $ and that there
exists a basis matrix for $L_S$ with $i^{th}$ row having length 
$\lambda _i(L_S).$
\begin{definition}
We say $B_S$ is an SM-basis matrix for $L_S$ iff the $i^{th}$  row of $B_S$
has length $\lambda _i(L_S).$\\
Let $\mathbf{\alpha}\equiv (\alpha_1, \cdots ,\alpha_n).$  We say $B_S$ is an $\alpha$SM-basis matrix for $L_S$ iff the $i^{th}$  row of $B_S$
has length $\leq \alpha_i\lambda _i(L_S).$
If we wish to not explicitly mention $\alpha,$ $\alpha$SM-basis matrices are also referred to as {\bf reduced basis} matrices.
\end{definition}

\begin{remark}
It is immediate from the definition that if $B_S$ is an SM-basis matrix for $L_S,$ 
then its $i^{th} $ row has length no more than the length of its $j^{th} $ row,
where $i\leq j.$
However, this is not necessarily true in the case of an $\alpha$SM-basis matrix.
\end{remark}

The following lemma is an immediate consequence of the definition of an
SM-basis matrix.

\begin{lemma}
Let  $ \hat{B}_S$  be an SM-basis matrix for the number lattice $L_S,$
and let $B_S$  be any basis matrix of $L_S.$ 
Then there exists a matrix $B'_S,$  obtained by permuting the rows 
of $B_S$ such that the length of the $i^{th}$ row of $B'_S$ is less or equal to that of the $i^{th}$ row of $ \hat{B}_S.$  
\end{lemma}


We now present a way of relating lengths of vectors in the number lattice 
parts of two invertibly linked generalized number lattices.
We need the following routine lemma whose  proof is relegated to the appendix.
\begin{lemma}
\label{lem:perm}
Let $\nu(\cdot), \mu(\cdot)$ be permutations of the set $\{1, \cdots ,n\}.$
Then, for every $j\in \{1, \cdots ,n\}$ there  exists an $i\in \{1, \cdots ,n\}$
such that $\nu(i)\leq j \leq \mu(i).$
\end{lemma}
\begin{theorem}
\label{thm:inverselength}

Let $\Vsp$ be a regular vector space, and
let $\Vp\equiv \Vsp\times P, \Vs\equiv \Vsp\times S.$

Let $\Kp, \Ks$ be generalized number lattices with
$\Kp\subseteq \Vsp\circ P, \Ks\equiv \Vsp\lrar \Kp.$
Let $\Kp= \Vp+L_{P},$ where $\Vp,L_{P},$ are orthogonal.

\begin{enumerate}
\item $\Ks=\Vs+L_{S},$ where $\Vs,L_{S},$ are orthogonal.
Further $L_S\subseteq \Vsp\lrar L_P$ and $L_P\subseteq \Vsp\lrar L_S.$ 
\item
Let  $x_P\in L_P, x_P\ne 0_P.$ 
 Then, there is a unique ${x}_S\in L_S$
 such
that  $({x}_S,x_P)\in \Vsp.$ 
Further,  $x_S\ne 0_S,$ 
$||{x}_S||\leq
  \sqrt{(|S|\times|P|)}\times ||x_{P}||.$ 
Similarly, if $x_S\in L_S, x_S\ne 0_S,$
 then, there is a unique ${x}_P\in L_P$
 such
that  $({x}_S,x_P)\in \Vsp.$
Further,  $x_P\ne 0_P,$
$||{x}_P||\leq
  \sqrt{(|S|\times|P|)}\times ||x_{S}||.$

%
%
%

\item Let $B_P$ be a basis matrix for $L_P,$ with rows $x_{1P}, \cdots, x_{nP}.$
Let ${x}_{iS}, i=1, \cdots ,n,$ be the unique vector in $L_S$ such that $({x}_{iS},x_{iP})\in \Vsp, i=1, \cdots , n,$
and let $B_S$ be the basis matrix for  $L_S,$ with ${x}_{iS}, i=1, \cdots n,$ as  the
$i^{th}$ row. Let the rows of $B_P,B_S$ be permuted in order of 
increasing length to yield the matrices $B'_P,B"_S.$ 
Let $B'_P$ have rows ${x}'_{1P}, \cdots, {x}'_{nP}$
and let $B"_S$ have rows ${x}"_{1S}, \cdots, {x}"_{nS}.$
Then,
$$\frac{1}{\sqrt{|S|\times |P|}}||{x}'_{iP}||\leq ||{x}"_{iS}||\leq \sqrt{|S|\times |P|}||{x}'_{iP}||.$$
\item  
Let $\hat{B}_P$ be an $\alpha $SM- basis matrix for $L_P$ with rows $\hat{x}_{1P}, \cdots, \hat{x}_{nP},$ and with $||\hat{x}_{iP}||\leq ||\hat{x}_{jP}||, i\leq j.$


Let $\hat{B}_S$ be any  basis matrix for $L_S$ with rows $\hat{x}_{1S}, \cdots, \hat{x}_{nS},$ and with $||\hat{x}_{iS}||\leq ||\hat{x}_{jS}||, i\leq j.$
Then,
$$\frac{1}{\alpha _i\sqrt{|S|\times |P|}}||\hat{x}_{iP}||\leq ||\hat{x}_{iS}||.$$


\item 
Let $\hat{B}_P$ be an $\alpha $SM- basis matrix for $L_P$ with rows $\hat{x}_{1P}, \cdots, \hat{x}_{nP},$ and with $||\hat{x}_{iP}||\leq ||\hat{x}_{jP}||, i\leq j.$
Let ${x}_{iS}, i=1, \cdots ,n,$ be the unique vector in $L_S$ such that $({x}_{iS},\hat{x}_{iP})\in \Vsp, i=1, \cdots ,n,$
and let $B_S$ be the matrix with ${x}_{iS}, i=1, \cdots ,n,$ as  the
$i^{th}$ row.
Then $B_S$ is a $\beta $SM-basis matrix for $L_S,$\\
where $\mathbf{\beta} \equiv (\beta_1, \cdots , \beta_n)=(\alpha_1(|S|\times|P|), \cdots , \alpha_n(|S|\times|P|)).$

\end{enumerate}
\end{theorem}
\begin{proof}
1. From Theorem \ref{thm:inversevsnl2}, we know that there exists $L'_S\subseteq \Ks$
such that $\Ks=\Vs+L'_S,L'_S\cap \Vs=\0_S, L_P\subseteq \Vsp\lrar L'_S.$
Let $L_S$ be the projection of $L'_S$ onto $\Vs^{\perp}.$
Then $\Vs,L_S$ are orthogonal and $L_S+\Vs =  L'_S+\Vs =\Ks.$


Next, since $0_P\in L_P,$ we must have that $\Vsp\times S\subseteq \Vsp \lrar L_P$ and it is given that $L'_S \subseteq \Vsp\lrar L_P$ so that 
$\Vsp \times S+L'_S\subseteq \Vsp \lrar L_P.$
Therefore,
$L_S\subseteq \Vsp \times S+L'_S\subseteq \Vsp\lrar L_P.$

Let $x_P\in \Vsp\lrar L'_S.$ Then there exists $x'_S\in L'_S$ such 
that $(x'_S,x_P) \in \Vsp.$  For any $x"_S\in \Vsp\times S,$
we have $(x"_S,0_P)\in \Vsp$ so that $(x'_S+x"_S,x_P) \in \Vsp.$
But $L'_S\subseteq L_S+\Vsp\times S$ which means that $x"_S$ can be chosen 
so that $x'_S+x"_S\in L_S.$ Thus $x_P\in \Vsp\lrar L_S.$ But $\Vsp\lrar L'_S\supseteq L_P.$ Therefore,
$\Vsp\lrar L_S\supseteq L_P.$

Thus, $\Ks=\Vs+L_S,$ with $\Vs,L_S$ orthogonal, with $\Vsp\lrar L_P\supseteq L_S,$
and with $\Vsp\lrar L_S\supseteq L_P.$\\

2. Let $x_P\in L_P, x_P\ne 0_P.$ Since $ \Vsp\circ P\supseteq \Kp\supseteq  L_P,$
we must have that $x_P\in \Vsp\circ P.$
From Lemma \ref{lem:normSP1}, it is clear that there exists a vector ${x}'_S$
such that $({x}'_S,x_P)\in \Vsp$
and such that $||{x}'_S||\leq
  \sqrt{(|S|\times|P|)}\times ||x_{P}||.$
Since $\Vsp\lrar \Kp= \Ks,$
we must have that  ${x}'_S\in \Ks.$
Since $\Ks=\V_S+L_S,$ and $\V_S,L_S$ are orthogonal to each other,
we have ${x}'_S=x"_S+{x}_S,{x"}_S \in \Vs, {x}_S\in L_S,$
$ ||{x}_S||\leq ||x'_S|| \leq
  \sqrt{(|S|\times|P|)}\times ||x_{P}||.$
Now $x"_S \in \Vs=\Vsp\times S,$ so that
$(x"_S,0_P)\in \Vsp.$ Therefore $(({x'}_S,x_P)-(x"_S,0_P))= ({x}_S,x_P)\in \Vsp.$

By Theorem \ref{thm:inversevsnl2}, part 3(a),
${x}_S$ is the unique vector in $L_S$ such that $({x}_S,x_P)\in \Vsp$
and since $x_P$ is nonzero, so is ${x}_S.$\\ 

The above argument clearly works with $S,P$ interchanged so that
the result is also true with $S,P$ interchanged.

3.  Let $\mu(\cdot), \nu(\cdot)$ be permutations of $\{1, \cdots , n\}$
such that $x_{iP}= x'_{\mu(i)P}, x_{iS}= x"_{\nu(i)S}, i=1 , \cdots , n.$
From the previous part of the theorem we have that
$$||x_{iP}||\leq \sqrt{(|S|\times|P|)}||x_{iS}||; ||x_{iS}||\leq \sqrt{(|S|\times|P|)}||x_{iP}||.$$ 
Let $j\in \{1, \cdots ,n\}.$ Then, by Lemma \ref{lem:perm}, there exists 
$i\in \{1, \cdots ,n\}$ such that $\nu(i)\leq j \leq \mu(i).$ We then have 
$$ ||x'_{jP}||\leq|| x'_{\mu(i)P}||=||x_{iP}||\leq \sqrt{(|S|\times|P|)}||x_{iS}||=\sqrt{(|S|\times|P|)}||x"_{\nu(i)S}||\leq \sqrt{(|S|\times|P|)}||x"_{jS}||.$$ 

Interchanging $S,P$
and $\mu(\cdot),\nu(\cdot)$ in the above argument we get 
$ ||x"_{jS}||\leq \sqrt{(|S|\times|P|)}||x'_{jP}||.$\\

%

4. Let $\hat{B}_S$ be any  basis matrix for $L_S$ with rows $\hat{x}_{1S}, \cdots, \hat{x}_{nS},$ and with $||\hat{x}_{iS}||\leq ||\hat{x}_{jS}||, i\leq j.$
Let $x_{iP}, i= 1, \cdots ,n$ be the unique vectors in $L_P$
such that $(\hat{x}_{iS}, {x}_{iP})\in \Vsp.$ Let $x"_{iP}\in \{x_{1P}, \cdots , x_{nP}\}, i= 1, \cdots ,n$
be such that $||x"_{jP}||\leq ||x"_{iP}||, j\leq i.$
Let $B"_P$ be the matrix with its $i^{th}$ row vector as
 $x"_{iP}, i= 1, \cdots ,n.$
We then have,  by the previous part, since $\hat{B}_P$ is an $\alpha$SM basis
matrix,
$$\frac{1}{\alpha _i\sqrt{(|S|\times|P|)}}||\hat{x}_{iP}|| \leq  \frac{1}{\sqrt{(|S|\times|P|)}}||x"_{iP}|| \leq ||\hat{x}_{iS}||.$$


5. Let $\tilde{B}_S$ be an SM-basis matrix for $L_S$ with rows
$\tilde{x}_{1S}, \cdots , \tilde{x}_{nS}.$
We  have 
$$ ||{x}_{iS}||\leq \sqrt{|S|\times |P|}||\hat{x}_{iP}||\leq 
 \alpha_i  \sqrt{|S|\times |P|}\sqrt{|S|\times |P|}||\tilde{x}_{iS}||
= \alpha_i {(|S|\times |P|)}||\tilde{x}_{iS}||.$$

\end{proof}
\begin{remark}
In Theorem \ref{thm:inverselength} we work with 
$\Vp= \Vsp\times P.$
Our primary interest is in linking number lattices.
So if $\Kp=\Vp+L_P$ and $\Vp\supseteq \Vsp\times P,L_P\cap\Vp=\0_P$
and $L_P, \Vsp\times P$ are orthogonal, we redefine our generalized
number lattice on $P$ to be $\hat{\K}_P\equiv \Vsp\times P +L_P$
and work with $\hat{\K}_P$ in place of $\Kp.$
Theorem \ref{thm:inverselength} would then be immediately applicable.
\end{remark}

\section{Approximate successive minima bases of dual number lattices}
\label{sec:approximatedual}
We show in this section that from certain approximate successive minima bases
 of lattices, such bases can be built for dual number lattices, efficiently.

\subsection{ LLL- reduced bases}
A special class of $\alpha SM-basis$ matrices that can be built in polynomial
time is the LLL- reduced basis. We give the definition and important 
properties of this basis matrix below.
\begin{definition}
Let $B_P
\equiv (b_1, \cdots , b_m)^T,$
be an $(m\times n)$  basis matrix of number lattice $L_P.$
We say $b^*_1, \cdots , b^*_m$ is the Gram-Schmidt orthogonalization
of $b_1, \cdots , b_m$ iff
\begin{enumerate}
\item $b_1^*=b_1;$
\item $b_i^*=b_i- \Sigma_{j=1}^{i-1}\alpha_{ij}b_j^*,$ where 
$\alpha_{ij}=\frac{<b_i,b_j^*>}{<b_j,b_j^*>}.$
\end{enumerate}
\end{definition}
It can be seen that $b^*_1, \cdots , b^*_m$ are orthogonal vectors
and have the same span as $b_1, \cdots , b_m.$ 

We define $B_P^*\equiv (b^*_1, \cdots , b^*_m)^T.$


It is clear that $B_P= KB_P^*,$ where $K$ is a lower triangular matrix
with $K_{ii}=1, K_{ij}=  \alpha _{ij},j\leq i\leq m.$

Note that, if  $\hat{B}_P\equiv \hat{K}\hat{B}_P^*,$
where $\hat{K}$ is a lower triangular matrix with $1'$s along the diagonal
and the rows of $\hat{B}_P^*$
are orthogonal to each other,
 then the rows $\hat{b}_i^*, i=1, \cdots m$
 of $\hat{B}_P^*$ constitute the Gram-Schmidt orthogonalization
of the rows $\hat{b}_i, i=1, \cdots m,$
of $\hat{B}_P.$

Let us rewrite $B_P=KB_P^*$ as $B_P=KD(D^{-1}B_P^*),$
where
$D$ is  a diagonal matrix with\\ $D_{ii}= ||b^*_i||, i=1, \cdots ,m.$
Observe that the rows of $D^{-1}B_P^*$ have length $1$
and are orthogonal to each other.
Let $F\equiv KD.$

\begin{definition}
We say $B_P$ is {\bf LLL-reduced},  
or, equivalently, satisfies the size reduction and the Lovasz condition 
with respect to its  Gram-Schmidt orthogonalization,
with $\delta \in (\frac{1}{4},1),$
iff
\begin{itemize}
\item (size reduction) $|\alpha_{ij}|\leq \frac{1}{2}, 1\leq i \leq m;$
\item (Lovasz condition) $(\alpha_{(i+1)i})^2||b^*_{i}||^2+ ||b^*_{i+1}||^2 \geq \ \delta ||b^*_{i}||^2, i.e., |F_{(i+1)i}|^2+ |F_{(i+1)(i+1)}|^2 \geq \ \delta (|F_{ii}|^2 ).$
\end{itemize}
\end{definition}

The following is an important property of LLL-reduced bases (Proposition 1.12 of \cite{LLL}).
\begin{theorem}
For $\delta =\frac{3}{4},$ an LLL reduced basis  
$B_P
\equiv (b_1, \cdots , b_m)^T$ of $L_P,$
satisfies\\  $< b_j, b_j>\ \leq 2^{m-1}(\lambda _j(L_P))^2.$
Therefore, an  LLL-reduced basis for $L_P$ is an $\alpha$SM-basis matrix with
$\alpha\equiv (\alpha_1, \cdots, \alpha_m), \alpha_i=2^{\frac{m-1}{2}},
i\leq m.$
\end{theorem}

LLL-reduced bases can be built in polynomial time from an integral basis of $L_P.$
The worst case complexity of the LLL-algorithm is $\tilde{O}(m^4n(log(\hat{M}))^2),$ where $m$ is the dimension of $L_P,$
$n=|P|,$   $\hat{M}$ is the maximum 
norm of a row in the original integral basis of $L_P$ (\cite{LLL}).
For comparison purposes we use the (simply stated) complexity of the original LLL-algorithm.
Faster algorithms are available (\cite{phong}).

We will show below, that once we have an LLL-reduced basis of $L_P,$
reduced bases, as good as the LLL-basis, of lattices linked to $L_P,$
can be found more efficiently (Theorem \ref{thm:linkedlength}).
This would be true even if one compares with the algorithms in \cite{phong}.

\subsection{LLL-reduced basis of dual from primal}
\label{subsec:LLLdual}

Let $\Kp=L_P+\Vp,$ where $L_P$ is an integral lattice orthogonal to $\Vp.$
Let $\Vp$ have $B_P\equiv  (b_1, \cdots , b_m)^T $ as a basis matrix.
Let $\Kp^d= L_P^{(2)}+\Vp^{(2)},$ where $L_P^{(2)},\Vp^{(2)}$ are orthogonal. By Theorem \ref{thm:gennlchar},
 we know that
$\Vp^{(2)}=(L_P+\Vp)^{\perp}$ and $L_P^{(2)}$ has $B_P^{(2)}$ as a basis matrix,
where $row(B_P^{(2)})=row(B_P)$ and $B_P(B_P^{(2)})^T= \ identity \ matrix\  I_m.$

It may be noted that a basis matrix of $ L_P^{(2)}$ would in general be 
rational rather than integral. However, for construction of LLL-reduced basis, it can be treated as an integral
matrix multiplying every entry by  $det(B_PB_P^T).$
This integer has size $ \tilde{O}(\hat{M}^m),$ i.e., bit size
$ \tilde{O}(mlog(\hat{M})),$ where $m$ is the dimension of $L_P$
and $\hat{M}$ is the maximum size of $||B_P(i)||.$

The resulting matrix would have 
maximum norm $M'$ of a row 
of bit length 
 $log(M')= \tilde{O}(mlog(\hat{M})).$

The LLL-reduced basis for such a matrix can be computed, if one proceeds from 
a basis of $L_P^{(2)},$ in   
$\tilde{O}(m^4n(log(M'))^2)= \tilde{O}(m^6n(log(\hat{M}))^2)$ where $m$ is the dimension of $L_P,$ $n=|P|,$ 
$\hat{M}$ is the maximum size of $||B_P(i)||.$

There are two components to the LLL-algorithm.
The final reduced basis has to satisfy the size reduction condition
and the Lovasz condition. The latter is more involved. If it is not satisfied
we have to swap the successive bases which violate it and begin all over again.
 We will show that given an LLL-reduced basis for 
$L_P,$ one can build one for $ L_P^{(2)}$ without having to verify
the Lovasz condition. 
The method given below is not new. Although not explicitly stated in
 the lecture on dual lattices in  \cite{regev}, it is  essentially immediate from  ideas  there.



We have 
$B_P=F(D^{-1}B_P^*).$

Let
$B_P^{{(2+)}}\equiv (K^{-1})^TD^{-1}(D^{-1}B_P^*)= (F^{-1})^T(D^{-1}B_P^*).$ 

It is clear that 
$B_P(B_P^{{(2+)}})^T=I_m, $
since $(D^{-1}B_P^*)(D^{-1}B_P^*)^T =I_m,$ 
and further, that $row(B_P)=row(B_P^{{(2+)}}).$

Let $B_P^{{(2)}}$ be the matrix obtained from  $B_P^{{(2+)}}$ by reversing the order 
of its rows. Let $E,H,G,$ be obtained respectively by reversing the order of the columns 
and the order of the rows of $D^{-1},(K^{-1})^T, (F^{-1})^T.$
For computation of $G$ from $F$ one may use the algorithm due to Dixon  
(\cite{dixon}) which has complexity $\tilde{O}(m^4log(\hat{M})),$ where
$\hat{M}$ is the maximum size of $||B_P(i)||.$ 
(Faster algorithms are available (\cite{storjohanninvert}) but the 
expressions are not simple to state.)


Let $T_k$ denote the $k\times k$ matrix obtained by reversing the order 
of rows of $I_k.$  Note that $C'=T_kC$ has the rows of $C$ in reverse order.
In particular this means that $T_k^2=I_k.$

We then have $$B_P^{{(2)}}=T_mB_P^{{(2+)}}; E= T_mD^{-1}T_m; H=T_m(K^{-1})^TT_m; G=T_m(F^{-1})^TT_m=(T_m(K^{-1})^TT_m)(T_mD^{-1}T_m).$$

It is clear that $B_P^{{(2)}}=GEB_P^*=HE^2B_P^*.$ 

Since the rows of $E^2B_P^*$ are orthogonal to each other and $H$
is a lower triangular matrix with $1'$s along the diagonal,
we have the following lemma.
\begin{lemma}
\label{lem:claim}
The Gram-schmidt orthogonalization of $B_P^{{(2)}}$ is $E^2B_P^*.$ 
The Lovasz condition for $B_P^{{(2)}}$ with respect to its  Gram-schmidt orthogonalization is 
$|G_{(i+1)i}|^2+ |G_{(i+1)(i+1)}|^2 \geq \ \delta (|G_{ii}|^2 ).$
\end{lemma}

We now have
\begin{theorem}
\label{thm:lovcond}
\begin{enumerate}
\item If  $B_P$ satisfies the size reduction  condition  with respect to its Gram-schmidt orthogonalization  $B_P^*,$ then
$B_P^{{(2)}}$ satisfies the size reduction condition for the entries $(i+1,i),
i=1, \cdots , m-1,$ 
with respect to its  Gram-schmidt orthogonalization  $E^2B_P^*.$ 
\item If  $B_P$ satisfies the Lovasz condition  with respect to its Gram-schmidt orthogonalization  $B_P^*,$ then
$B_P^{{(2)}}$ satisfies the Lovasz condition 
with respect to its  Gram-schmidt orthogonalization  $E^2B_P^*.$
\end{enumerate}
\end{theorem}
The proof is relegated to \ref{sec:LLLdual}.

While the matrix $B_P^{{(2)}}=HE^2B_P^*,$ satisfies the Lovasz condition,
it may not satisfy the size reduction condition, since the entries $H_{ij}, i> j+1$ are not
necessarily of magnitude less or equal to $\frac{1}{2}.$
We therefore require one iteration, without swapping, of the LLL-algorithm,
for performing the size reduction. 
The matrix, if scaled by $det(B_PB_P^T)$ would become integral.
The maximum norm $M'$ of a row of this matrix would be $\tilde{O}((\hat{M})^m),$
where $\hat{M}$ is the  maximum norm of a row of $B_P.$
If no swapping is involved, the size reduction alone
takes $\tilde{O}(m^2nlog(M')) = \tilde{O}(m^3nlog(\hat{M}))$ time, where $n=|P|$	
(\cite{phong},\cite{regev}).
The final LLL-reduced matrix is obtained  by scaling down the 
result of the size reduction by $det(B_PB_P^T).$
We saw earlier, that computation of $G$ from $F$ can be done in
$\tilde{O}(m^4log(\hat{M})) $ time.
Let $M$ be the maximum magnitude of an entry in $L_P.$
Then it is clear that $log(\hat{M})=O(log(|P|M))=O(log(nM)).$
 The multiplication
$GEB_P^*$ can be done in $\tilde{O}((m^{\theta }n)log(\hat{M}))=
\tilde{O}((m^{\theta }n)log(n{M}))$ time.
Since $\theta < 3,$
the overall computation of 
 $B_P^{(2)"}$  from $B_P$ can be performed in $\tilde{O}((m^4+m^3n)log(n{M}))$ time.

The above  is a substantial improvement over computing using the LLL-algorithm
directly on the basis $B_P^{{(2)}}$ of $L^{(2)}_P$
(complexity $\tilde{O}(m^6n(log(n{M}))^2)$).
We summarize this discussion in  part 1 of Theorem \ref{thm:linkedlength}.
\section{$\alpha$SM-basis matrices for number lattices through
\label{sec:summary}
 linking and dualization}
We summarize in this section our results on approximate SM-bases 
of number lattices related through linking by regular vector spaces 
and by dualization.

\begin{theorem}
\label{thm:linkedlength}
Let $\Vsp$ be a regular vector space, $\Kp $ be a generalized number lattice with
$\Kp= \Vp+L_{P},$ where $L_{P}$ is an integral lattice orthogonal to $\Vp.$ 
Let $ \Vsp\times P=\Vp$   and let
$\Vsp\circ P= span(\Kp).$

Let $\Ks\equiv \Vsp\lrar \Kp,$
 with
$\Ks= \Vs+L_{S},$ where $\Vs,L_{S},$ are orthogonal.


Let $\Kp^d= \Vp^{(2)}+L_{P}^{(2)},$ where $\Vp^{(2)},L_{P}^{(2)},$ are orthogonal.
Let $B_P$ be an LLL-reduced basis for $L_P$ available with  its Gram-Schmidt 
orthogonalization. 
Let $m\equiv  dim(L_P),n\equiv |P|,$ and let $M $ be the maximum  size of an entry 
in $B_P. $

Let standard representative matrix $Q_{SP}$  of $\Vsp,$  
which is  built according to the priority $(P,S),$ be available.
We then have the following.
\begin{enumerate}
\item If $L_P$ is integral, an $\alpha$SM-basis matrix  $B_P^{(2)"}$
for $L_{P}^{(2)},$ where $\alpha \equiv (\alpha_1, \cdots ,\alpha _m),
\alpha _i= 2^{\frac{m-1}{2}},i =1, \cdots, m$ can be built in  
$\tilde{O}((m^4+m^3n)log(n{M}))$ time.
Further, the maximum bit size of an entry in $B_P^{(2)"}$ would be
$\tilde{O}(mlog(nM)).$

\item A $\beta$SM-basis matrix $B_S$
for $L_S, $ where $\beta \equiv (\beta_1, \cdots ,\beta _{m}),
\beta _i= (|S|\times |P|)2^{\frac{m-1}{2}},i =1, \cdots, m$ can be built 
in 
time \\
$\tilde{O}(r^4log(|S|)) + \tilde{O}(r^2mlog((|S|\times |P|)M))+ O((r(\Vsp^{\perp}\circ S))(r(\Vsp\circ P))m log(M)),$\\
where $r\equiv min (r(\Vsp\times S), r(\Vsp^{\perp}\circ S)).$
 Further, the maximum bit size of an entry in $B_S$ would be
$\tilde{O}(log((|S|\times |P|)M)).$

\item If $L_P$ is an integral matrix, and an LLL-reduced basis for 
$L_{P}^{(2)}$ has already been computed as in part 1 above, a $\beta$SM-basis matrix $B^{(2)}_S$
for $L_{S}^{(2)},$ where $\beta \equiv (\beta_1, \cdots ,\beta _{m}),
\beta _i= (|S|\times |P|)2^{\frac{m-1}{2}},i =1, \cdots, m,$
can be built in time\\
$
{O}((r(\Vsp))^2 |S\uplus P|)+\tilde{O}(\hat{r}^4log(|S|)) + \tilde{O}(\hat{r}^2m^2log(M)+\hat{r}^2mlog(|S|\times |P|))$\\$+ O((r(\Vsp\circ S)(r(\Vsp^{\perp}\circ P)m^2 log(M)),$\\
{where} $\hat{r}\equiv min (r(\Vsp^{\perp}\times S), r(\Vsp\circ S)).$
 Further, the maximum bit size of an entry in $B^{(2)}_S$ would be
$\tilde{O}(mlog(M)+log((|S|\times |P|))).$



\end{enumerate}
\end{theorem}

We need the following lemma for the proof of Theorem \ref{thm:linkedlength}.  

\begin{lemma}
\label{lem:projcomplexity}
Let $\Vsp,\Kp,L_P,$ be as in the statement of Theorem \ref{thm:linkedlength}.
Let $Q_{SP}$ be a standard representative matrix of a regular 
vector space $\Vsp$  built according to the priority sequence $(P,S).$
Let $x_P \in L_P.$
Then,
\begin{enumerate} 
\item to build a standard representative matrix $Q'_{SP}$ of 
$\Vsp$  according to the priority sequence $(S,P)$
takes ${O}((r(\Vsp))^2 |S\uplus P|) $ time;
\item to find a  vector $x_S$ such that $(x_S,x_P)\in \Vsp ,$  
$||x_S||\leq \sqrt{(|S|\times |P|)}\times ||x_P||$ and such that 
maximum bit length of an entry in $x_S$ is $O(log(|P|M)),$ 
where $log(M)$ is the maximum bit length of an entry in $x_P$
takes $O((r(\Vsp^{\perp}\circ S))(r(\Vsp\circ P)) log(M))$
time.
\item to project $x_S$ onto $\Vs,$
where $\Vs$ stands for $\Vsp\times S$ or $\Vsp^{\perp}\circ S$
takes\\ 
$\tilde{O}((r^4log(|S|)) + \tilde{O}((r^2(log((|S|\times |P|)M)),$ time
where $r \equiv min (r(\Vsp\times S), r(\Vsp^{\perp}\circ S)).$
The projections $x'_S,x"_S$ respectively of $x_S$ onto $\Vs, \Vs^{\perp}$ 
have maximum bit size $O(log((|S|\times |P|)M)).$
\end{enumerate} 
\end{lemma}
\begin{proof}
1. This follows from Lemma \ref{lem:basetobase}.

2. This follows from Lemma \ref{lem:normSP1}.

3. Note that a standard representative matrix, say $Q_S,$  of $\Vsp\times S$ is visible
in the standard representative matrix $Q_{SP}$ of $\Vsp.$
and a standard representative matrix, say $Q'_S,$  of $\Vsp\circ S$ is visible
in the standard representative matrix $Q'_{SP}$ of $\Vsp.$

Let $x'_S,x"_S$ be projections of $x_S$ onto $\Vsp\times S,$
$(\Vsp\times S)^{\perp},$  respectively. 
Since $Q_{SP}$ is totally unimodular, so is $Q_S.$ 
We then have,
$$ x'_S=Q_S^T(Q_SQ_S^T)^{-1}(Q_Sx_S); x"_S=x_S-x'_S.$$
The maximum magnitude of entries in $Q_SQ_S^T$ is $|S|.$
Therefore the time for inversion (\cite{dixon}) is \\
$\tilde{O}((r(\Vsp\times S))^4log(|S|).$

From Lemma \ref{lem:normSP1}, the maximum bit size of an entry  in $x_S$ is
$O(log(|P|M)).$
The maximum bit size of an entry  in $x'_S,x"_S$ can be seen to be $O(log((|S|\times |P|)M)).$
The multiplication time can be seen to be 
$ \tilde{O}((r(\Vsp\times S))^2log((|S|\times |P|)M)).$

Instead of computing $x'_S$ first, we could have computed $x"_S.$
We would then have had to interpret $Q_S$ as a standard representative matrix 
for $\Vsp^{\perp}\circ S= (\Vsp\times S)^{\perp}.$
If $Q'_S=(I|K)$ is a standard representative matrix of $\Vsp\times S,$
then 
$Q"_S\equiv (-K^T|I)$ is a standard representative matrix for $(\Vsp\times S)^{\perp}=\Vsp^{\perp}\circ S.$
Thus $Q"_S$ is a totally unimodular matrix.
Therefore,  the above calculation
holds for computing $x"_S$ first with $Q_S$ denoting a standard representative
matrix of $\Vsp^{\perp}\circ S.$

Therefore the time for inversion can be taken to be 
$\tilde{O}(r^4log(|S|)),$ 
and the multiplication time can be seen to be
$ \tilde{O}(r^2log((|S|\times |P|)M)),$ where $r\equiv min (r(\Vsp\times S), r(\Vsp^{\perp}\circ S)).$ 

\end{proof}

{\it Proof of Theorem \ref{thm:linkedlength}.}

1. This is clear from the discussion at the end of Subsection \ref{subsec:LLLdual}.

2. Given $x_P\in L_P,$ to find $x_S$ such that $(x_S,x_P)\in \Vsp$
takes $\tilde{O}((r(\Vsp^{\perp}\circ S))(r(\Vsp\circ P))log(max|x_P(i)|)= \tilde{O}((r(\Vsp^{\perp}\circ S)))(r(\Vsp\circ P)log(M))$ time,
by Lemma \ref{lem:normSP1}. This has to be done as many times as there are 
rows of $B_P.$ To project $x_S$ onto $\Vsp\times S$
takes 
 $\tilde{O}((r(\Vsp\times S))^4log(|S|)) + \tilde{O}((r(\Vsp\times S))^2log((|S|\times |P|)M))$ time.
The first term, corresponds to the inversion which has to be done only once.
The multiplication has to be done as many times as there are rows of $B_P.$
Therefore, to compute $B_S$ takes time $$\tilde{O}(m(r(\Vsp^{\perp}\circ S))(r(\Vsp\circ P))log(M))+\tilde{O}((r(\Vsp\times S))^4log(|S|)) + \tilde{O}((r(\Vsp\times S))^2mlog((|S|\times |P|)M)).$$
If the projection of $x_S$ had been performed onto $(\Vsp\times S)^{\perp}$
the computation time would be 
given by the above expression with $r(\Vsp\times S)$ being replaced by
$r((\Vsp\times S)^{\perp})=r(\Vsp^{\perp}\circ S).$

3. 
We have,
$\Kp^d= \V^{(2)}_P+L^{(2)}_P,$ with $\V^{(2)}_P,L^{(2)}_P,$ 
orthogonal and $\Ks^d= \V^{(2)}_S+L^{(2)}_S,$ with $\V^{(2)}_S,L^{(2)}_S,$
orthogonal.
We are given that $\Vsp\circ P=span(\Kp).$ 
Further, $\V^{(2)}_P= (\Vp+L_P)^{\perp}=(\Kp)^{\perp}=(\Vsp\circ P)^{\perp}=
\Vsp^{\perp}\times P$ and $span(L_P)=span(L^{(2)}_P).$ 
By Theorem \ref{thm:gennlchar},
$span(\Kp^d)= \Vp^{\perp}=(\Vsp\times P)^{\perp}=\Vsp^{\perp}\circ P.$
We have, by IDT (Theorem \ref{thm:idt}), $\Ks^d=(\Vsp\lrar \Kp)^d= (\Vsp)^{\perp}\lrar \Kp^d.$         

Since $(\Vsp)^{\perp}\lrar \Kp^d= \Ks^d, \Vsp^{\perp}\times P=\V^{(2)}_P, \Vsp^{\perp}\circ P\supseteq \Kp^d,$
Theorem \ref{thm:inverselength}
 is applicable.

To build a standard representative matrix $Q'_{SP}$ of
$\Vsp$  according to the priority sequence $(S,P)$ from $Q_{SP}$
takes ${O}((r(\Vsp))^2 |S\uplus P|) $ time.
 A standard representative matrix  of
$\Vsp^{\perp}$  according to the priority sequence $(P,S)$
immediately becomes available.
Within this matrix standard representative matrices are visible for
 $\Vsp^{\perp} \times S, (\Vsp^{\perp} \times S)^{\perp}= \Vsp\circ S . $
The maximum bit size of an entry of 
$L^{(2)}_P$ is $\tilde{O}(mlog(M)).$ Therefore, by the computation 
in part 2 above, using $\Vsp^{\perp}$ in place of $\Vsp,$ $(mlog(M)+log(|S|\times|P|))$ in place of $log(M)$ for maximum 
bit size, and 
taking $r\equiv min (r(\Vsp\circ S), r(\Vsp^{\perp}\times S)),$
the result follows.
$\al$

\begin{remark}
We do not have a general method for building a reduced basis for the dual 
lattice from a general reduced basis of the primal lattice. For the LLL-reduced
basis, fortunately we have an efficient method for building a reduced basis
for the dual.
If $L_P$ is a number lattice on $P$ and $\Vsp$ is a vector space on $S\uplus P,$
satisfying appropriate conditions relative to $L_P,$
we can efficiently build a reduced basis for $L_S,$ the number lattice part of 
$\Vsp\lrar L_P,$ from a reduced basis for $L_P.$
However, there appears no easy way of building the reduced basis of the dual of $L_S,$
starting from an LLL-reduced basis of $L_P,$
except through the  use of the implicit duality theorem as in part 3 of Theorem \ref{thm:linkedlength}.
\end{remark}
\section{Closest and shortest vectors}
\label{sec:closest}
In Section \ref{sec:uniquelinkage}, we studied the case where a vector
in a number lattice is uniquely lnked to that in another.
In this section we examine the situation where the vector linked
to a given vector is not unique.
We examine questions of the kind `what is a shortest vector linked
to a given vector?' We show that the answer is related  
to the 
notion of a `closest vector'
defined below.
%
%
%

Let $L_P\subseteq \Q^{|P|}$ and let $x_P\in \Q^{|P|}.$
We say $\hat{x}_P\in L_P$ is {\bf closest in $L_P$ to $x_P$} iff
 whenever $x'_P\in L_P,$ we have $||\hat{x}_P-x_P||
\leq ||{x'}_P-x_P||.$
We say $\hat{x}_P\in L_P$ is {\bf $\alpha-$closest in $L_P, \alpha \in \Q,$ to $x_P$} iff
 whenever $x'_P\in L_P,$ we have $||\hat{x}_P-x_P||
\leq \alpha||{x'}_P-x_P||.$

The {\bf closest vector problem} is to find a vector  closest in $L_P$ to 
a given vector $x_P.$
This problem is known to be NP-Hard (\cite{arora}). However, there is a polynomial time
algorithm available (\cite{babai1}) for finding an $\alpha-$closest 
vector in $L_P$ to a given vector, for $\alpha = 2^{\frac{m}{2}}.$
This algorithm is given below.


Let $B_P$ be an LLL-reduced basis matrix for $L_P.$
Let $B_P=(b_{1P}, \cdots,  b_{mP})^T=KB^*_P= K (b^*_{1P}, \cdots , b^*_{mP})^T,$ where $B^*_P$ is the Gram-Schmitt orthogonalization 
for $B_P.$ Let $x_P\in \Q^{|P|}.$ 
Then the vector 
$$(c_1|\cdots |c_m)B_P, c_j=\lceil{\frac{<x_P,b^*_{jP}>}{<b^*_{jP},b^*_{jP}>}}\rfloor,
j=1, \cdots , m,
$$
is $2^{\frac{m}{2}}-$closest in $L_P$ to $x_P.$


We now examine some natural questions about short vectors which
are related to the closest vector problem.

{\bf Question 1.} {\it Suppose $L_{PQ}$ is a number lattice with $(x_P,x_Q)\in L_{PQ}.$
What is a shortest vector in $L_{PQ}$ whose restriction to $P$ is $x_P?$}

We reformulate this question in terms of closest vectors in the lemma below.
\begin{lemma}
Let $L_{PQ}$ be a number lattice with $(x_P,x_Q)\in L_{PQ}.$
\begin{enumerate}
\item
A vector $\hat{x}_Q$ is a closest vector to $x_Q$ in $L_{PQ}\times Q$
 iff\\
$(x_P,x_Q-\hat{x}_Q)$ is the shortest vector in $L_{PQ}$ whose restriction to $P$ is $x_P,$
i.e.,\\
$(x_P,{x}_Q-\hat{x}_Q)\in L_{PQ}$ and 
whenever $(x_P,x'_Q)\in L_{PQ},||x_Q-\hat{x}_Q||\leq  ||x'_Q||.$
\item The vector  $\hat{x}_Q$ is an $\alpha -$closest vector to $x_Q$ in $L_{PQ}\times Q,$
iff \\
$(x_P,{x}_Q-\hat{x}_Q)\in L_{PQ}$ and
whenever $(x_P,x'_Q)\in L_{PQ},||x_Q-\hat{x}_Q||\leq \alpha ||x'_Q||.$
\end{enumerate}
\end{lemma}
\begin{proof}
Let $(x_P,x_Q)\in  L_{PQ}.$ Then $(x_P,x_Q-x"_Q)\in  L_{PQ}$
iff $x"_Q\in L_{PQ}\times Q.$

We have that 
$\hat{x}_Q$ is a closest vector to $x_Q$ in $L_{PQ}\times Q,$

iff $\langle (x_Q-x"_Q),(x_Q-x"_Q)  \rangle 
\geq \langle (x_Q-\hat{x}_Q)),(x_Q-\hat{x}_Q)\rangle ,$ whenever $x"_Q \in L_{PQ}\times Q,$ i.e., 

 iff $\langle (x_P,x_Q-x"_Q),(x_P,x_Q-x"_Q)\rangle \geq 
\langle (x_P,x_Q-\hat{x}_Q),(x_P,x_Q-\hat{x}_Q)\rangle,$  whenever 
$(x_P,x_Q-x"_Q)\in L_{PQ}.$

%
The proof of the second part is similar.
\end{proof}

{\bf Question 2.} {\it Let $L_P$ be a number lattice.
Let $L'_P$ be the projection of $L_P$ onto a vector space $\Vp.$
Given a vector $x'_P\in L'_P$ what is a shortest vector in $L_P$
whose projection onto $\Vp$ is $x'_P?$}

We can reduce this to Question 1 above as follows.

Let $B_P$ be a basis matrix for  $L_P,$ with ${x}_{iP}, i=1, \cdots ,m,$ as  the
$i^{th}$ row. Resolve ${x}_{iP}$ as ${x}_{iP}={x'}_{iP}+{x"}_{iP},i=1, \cdots ,m,
{x'}_{iP}\in \Vp, {x"}_{iP}\in \Vp^{\perp}.$
Let $P',P"$ be  disjoint copies of $P$ and let $\hat{x}_{P'},\hat{x}_{P"}$ denote copies 
respectively on $P',P"$ of $\hat{x}_P\in \F_P.$
Let $B_{P'P"}$ be the matrix  with $({x'}_{iP'}, {x"}_{iP"}), i=1, \cdots m,$ as  the
$i^{th}$ row. 
Let $L_{P'P"}$ be the number lattice on $P'\cup P"$ generated by the rows 
of $B_{P'P"}.$ It is clear that $x_P\in L_P$ iff $({x'}_{P'}, {x"}_{P"})\in L_{PP"}, $ where ${x'}_{P}, {x"}_{P"}$ are respectively the projections of
$x_P$ onto $\Vp,\Vp^{\perp}.$
We have that $||x_P||^2= ||{x'}_{P'}||^2+||{x"}_{P"}||^2.$
Further since the rows of $B_P$ are linearly independent, so are the rows of 
$B_{P'P"}.$
Therefore, $B_{P'P"}$
is a basis matrix for $L_{P'P"}.$

The projection $L'_{P}$ of $L_P$ onto $\Vp$ has the copy
$L_{P'P"}\circ P'.$
We have seen that the lengths of corresponding vectors in $L_P,$
$L_{P'P"}$ are the same.
Therefore, Question 2 can be rephrased as `given $x'_{P'}\in L_{P'P"}\circ P'$
what is a shortest vector $({x'}_{P'}, {x"}_{P"})\in L_{P'P"}?,$'
which is Question 1.


We next examine what information about $L_P,$ in terms of short vectors, can be 
garnered from $\Ks\equiv \Vsp\lrar L_P.$
%

We need a preliminary lemma before we phrase the next question.
\begin{lemma}
\label{lem:q3}
Let $\Vsp$ be a vector space, $L_P\subseteq \Vsp\circ P,$ be a number lattice,
with $span(L_P) \supseteq \Vsp\times P.$
Let $\Ks\equiv \Vsp\lrar L_P.$ Then $\Ks=\Vsp\times S+L_S,$
where $\Vsp\times S,L_S,$ are orthogonal and $L_S$ is a number lattice.
\end{lemma}
\begin{proof}
Let $L'_P$ be the projection of $L_P$ onto $(\Vsp\times P)^{\perp}.$
Let $\Kp\equiv L'_P+\Vsp\times P.$
Since $span(L_P) \supseteq \Vsp\times P,$
 we have that $span(L_P)=span(L'_P+\Vsp\times P)=span(\Kp),$ 
and that $L_P+\Vsp\times P= L'_P+\Vsp\times P= \Kp. $
It is clear that $\Vsp\lrar \Kp\supseteq \Vsp\lrar L_P.$
Let $x_S\in \Vsp\lrar\Kp.$ Then there exists $\hat{x}_P\in \K_P$ such that 
$(x_S,\hat{x}_P)\in \Vsp.$ Now $\Kp=L_P+\Vsp\times P$ so that 
$\hat{x}_P=x_P+x"_P, x_P\in L_P,x"_P\in \Vsp\times P.$
Therefore $(x_S,\hat{x}_P)+(0_S,x"_P)= (x_S,x_P)\in \Vsp.$ Thus,  $\Vsp\lrar L_P\supseteq\Vsp\lrar \Kp$
and therefore  $\Vsp\lrar L_P=\Vsp\lrar \Kp.$
From Theorem \ref{thm:inverselength}, part 1, it follows that 
$\Ks=\Vsp\times S+L_S,$
where $\Vsp\times S,L_S,$ are orthogonal and $L_S$ is a number lattice.
\end{proof}

{\bf Question 3.} {\it Let $\Vsp$ be a regular vector space, $L_P\subseteq \Vsp\circ P,$ be a number lattice, with $span(L_P) \supseteq \Vsp\times P.$
Let $\Kp=L_P+\Vsp\times P,$ with $L_P,\Vsp\times P,$ not necessarily orthogonal.
 Let $\Ks\equiv \Vsp\lrar L_P.$  
Given $x_S\in L_S,$ what is a shortest vector  $x_P\in L_P$
such that $(x_S,x_P)\in \Vsp ?$} 

First let $L'_{P}$  be the projection of $L_P$ onto $(\Vsp\times P)^{\perp}.$
Let $\K_P\equiv L'_{P}+\Vsp\times P.$ It is clear that $\K_P= L_P+\Vsp\times P.$
By Lemma \ref{lem:q3}, $\Ks\equiv \Vsp\lrar\K_P=\Vsp\lrar L_P$
and we  can write $\Ks=\Vsp\times S+L_S,$ where  $L_S$ is orthogonal to $\Vsp\times S.$
%
Since $\Vsp\circ S\supseteq \Ks, \Vsp\times S\subseteq \Ks, $ by Theorem \ref{thm:inversevsnl2}, we have that $\Vsp\lrar \Ks=\Kp.$
By Theorem \ref{thm:inverselength}, if $x_S\in L_S$ there is a unique
vector $x'_P\in L'_P$ such that $(x_S,x'_P)\in \Vsp.$
Further, since $(x_S,x'_P)\in \Vsp,$ a vector $(x_S,x_P)\in \Vsp$ iff $x_P-x'_P\in \Vsp\times P.$

Question 3. now reduces to Question 2. We find the unique $x'_P\in L'_P$
such that $(x_S,x'_P)\in \Vsp.$
Next we find a shortest vector $x_P\in L_P$ whose projection 
onto $(\Vsp\times P)^{\perp}$ is $x'_P.$
Note that $x_P-x'_P\in \Vsp\times P,$  so that $(x_S,x_P)\in \Vsp.$
Therefore $x_P$ is a shortest vector in $L_P,$   
such that $(x_S,x_P)\in \Vsp .$

\begin{remark}
Questions 1, 2 and 3 can be rephrased in terms of `$\alpha -$' closest rather 
than `shortest' and the answers are  similar.  
\end{remark}

\section{Conclusion}
\label{sec:conclusion}
We have shown that ideas that have proved useful in the study of 
electrical networks, viz., implicit inversion theorem (IIT) (Theorem \ref{thm:inversevsnl1})
and implicit duality theorem (IDT) (Theorem \ref{thm:idt}), which are basic to  implicit linear algebra, can be used with profit
 to study number lattices linked through regular vector spaces and 
through dualization.

We have built new self dual number lattices from old by using IDT, in a manner analogous to
building new reciprocal electrical networks from old.

Using IIT, we have related properties of number 
lattices invertibly linked through a regular vector space and using IDT,
to that of duals of such lattices.
We have shown that reduced bases for such number lattices 
can be built efficiently starting from such a basis for one of them.

We have shown that the short vector problem under certain additional restrictions
can be solved by solving an appropriate closest vector problem.


\appendix

\section{Variations on the problem of finding basis for a number lattice}
\label{sec:Var}

1.  Consider the problem of finding the basis for the number lattice $\hat{L}_S$ of integral solutions 
to the equation
\begin{align}
\label{eqn:integral}
\bbmatrix{x_1^T&\vdots&x_2^T}\bbmatrix{A_S\\ \cdots \\B}=0,
\end{align}
where $A_S$ has full row rank and  rows of $A_S$ generate the rows of $B.$
The SL-algorithm finds a unimodular matrix 
\begin{align*}
R= \bbmatrix{R_{11}& \vdots &R_{12}\\
R_{21}& \vdots &R_{22}}\ 
\mbox{such that}\ 
\bbmatrix{R_{11}& \vdots &R_{12}\\
R_{21}& \vdots &R_{22}}
 \bbmatrix{A_S\\ \cdots \\B}= \bbmatrix{\hat{A}_S\\ \cdots \\  \0}.
\end{align*}

We claim that $(R_{21}\ \vdots \ R_{22}) $ is a basis matrix for the 
number lattice $\hat{L}_S.$

To see this, first it is clear that rows of $(R_{21}\ \vdots \ R_{22}) $
are in $\hat{L}_S.$

Next, since $R$ is unimodular, we can write any integral solution to Equation \ref{eqn:integral}
as 
\begin{align*}
\bbmatrix{x_1^T&\vdots&x_2^T}=  \bbmatrix{\lambda_1^T&\vdots&\lambda_2^T}
\bbmatrix{R_{11}& \vdots &R_{12}\\
R_{21}& \vdots &R_{22}}, \  \bbmatrix{\lambda_1^T&\vdots&\lambda_2^T}\ \mbox{integral}.
\end{align*}
If $\lambda_1^T$ is nonzero 
\begin{align}
\bbmatrix{x_1^T&\vdots&x_2^T}\bbmatrix{A_S\\ \cdots \\B}= \bbmatrix{\lambda_1^T&\vdots&\lambda_2^T}\bbmatrix{\hat{A}_S\\ \cdots \\  \0} 
\ne  0,
\end{align}
 since $\hat{A}_S$ has full row rank.
Therefore $\lambda_1^T =0$ and any integral solution to Equation \ref{eqn:integral}
can be written as 
\begin{align*}
\bbmatrix{x_1^T&\vdots&x_2^T}=  \lambda_2^T
\bbmatrix{
R_{12}^T& \vdots &R_{22}^T}, \  \lambda_2^T\ \mbox{integral}.
\end{align*}
Thus, $(R_{21}\ \vdots \ R_{22}) $ is a basis matrix for the 
number lattice $\hat{L}_S.$


2. The problem of finding a basis for the collection of all integral vectors contained in
a vector space $\V_S$
is equivalent to finding a  basis matrix
for the integral solutions to the equation
$x^TQ_S^T=0,$ where $Q_S$ is a basis matrix for $\Vs^{\perp}.$

3. Finding a basis for the intersection of a number lattice with a 
vector space can be handled similarly. Let $C_S$ be a basis matrix for number lattice $L_S.$ Let $\Vs, \Vs^{\perp}$ be  complementary orthogonal vector spaces with representative matrices $B_S, Q_S$ respectively.  A vector $x_S^T$ belongs to $L_S\cap \Vs,$
iff $x_S^T= \lambda ^TC_S, \lambda \ \mbox{integral}$ and $x_S^TQ_S^T=0.$ 
We first find a basis matrix $\hat{C}$ for integral solutions to 
$\lambda ^T C_SQ_S^T=0$ and then compute $\hat{C}C_S.$ We claim that this matrix is a basis  matrix   
for $L_S\cap \Vs.$

The rows of $\hat{C}C_S$
belong to $L_S\cap \Vs,$
since they  are  linear combinations of rows of $C_S$  
orthogonal to $\Vs^{\perp}.$ 
To see the converse, any vector $x_S^T$ in $L_S\cap \Vs,$ 
must satisfy  $x_S^T= \lambda ^TC_S, \ \lambda \ \mbox{integral and}
\ \lambda ^T C_SQ_S^T=0.$
Any integral vector in the solution space of $\lambda ^T C_SQ_S^T=0,$
can be written as $\sigma^T\hat{C}, \ \sigma  \ \mbox{integral}.$
Therefore  any vector $x_S^T$ in $L_S\cap \Vs,$ can be written 
$\sigma^T\hat{C}C_S, \ \sigma  \ \mbox{integral}.$

\section{Proof of Lemma \ref{lem:Kgenminor}}

1. Suppose $\KSP \lrar \KPQ = \KSQ$ and 
${\bf 0}_{PQ} \in \KPQ .$
\\
It is clear
from the definition of the matched composition operation that
$\KSP \circ S \supseteq \KSQ\circ S.$  
\\
Since 
${\bf 0}_{PQ} \in \KPQ,  $  if 
$\fS \oplus {\bf 0}_{P} \in \KSP,$ we must have that $\fS \oplus {\bf 0}_{Q} \in \KSQ.$
Thus, $\KSP \times S \subseteq \KSQ\times S.$ 

2. On the other hand suppose
$ \KSP \times S \subseteq \KSQ \times S$ and $\KSP\circ S \supseteq \KSQ \circ S.$
Let $\hat{\K}_{PQ} \equiv \KSP \lrar \KSQ,$ i.e., $\hat{\K}_{PQ}$ is the collection of all vectors $\fP\oplus \fQ $ s.t. for some
vector $\fS,$ 
$\fS\oplus \fQ \in \KSQ$, $\fS \oplus \fP \in \KSP.$  
\\
Since $\KSP$ is closed under subtraction,
it  contains the zero vector, the negative of every vector in it and is closed
under addition. 
\\
Since  ${\bf 0}_S \oplus{\bf 0}_P\in \KSP,$ we must have that   ${\bf 0}_S \in \KSP \times S$ and therefore   ${\bf 0}_S \in \KSQ\times S.$
It follows that 
${\bf 0}_S \oplus {\bf 0}_{Q} \in \KSQ. $ 
\\
Hence, by definition of $\hat{\K}_{PQ},
{\bf 0}_{PQ} \in \hat{\K}_{PQ}.$ Further, since both $\KSP, \KSQ,$ are closed under addition, so is $\hat{\K}_{PQ}$ since $\hat{\K}_{PQ}= \KSP \lrar \KSQ.$\\
We will now show that $\KSP\lrar \hat{\K}_{PQ}= \KSQ.$\\
Let $\fS\oplus \fQ \in \KSQ.$
Since $\KSP\circ S \supseteq \KSQ \circ S,$ we must have that $\fS\oplus \fP \in \KSP,$ for some $\fP.$
By the definition of $\hat{\K}_{PQ},$ we have that $\fP\oplus \fQ \in \hat{\K}_{PQ}.$ 
\\
Hence,
$\fS\oplus \fQ \in \KSP \lrar \hat{\K}_{PQ}.$ 
\\
Thus,
$\KSP \lrar \hat{\K}_{PQ} \supseteq \KSQ.$
\\
Next, let $\fS\oplus \fQ\in \KSP \lrar \hat{\K}_{PQ},$ i.e., for some $\fP, \fS \oplus \fP \in \KSP$ and $\fP\oplus \fQ \in \hat{\K}_{PQ}.$
\\
We know, by the definition of $\hat{\K}_{PQ},$ that there exists $\fS'\oplus \fQ \in \KSQ$ s.t. $\fS' \oplus \fP \in \KSP.$
\\
Since $\KSP$ is closed under subtraction, we must have, $(\fS - \fS') \oplus
{\bf 0}_{P} \in \KSP.$  
\\
Hence, $\fS - \fS' \in \KSP \times S
\subseteq \KSQ\times S.$  
\\
Hence $(\fS - \fS') \oplus
{\bf 0}_{Q} \in \KSQ.$ 
\\
Since $\KSQ$ is closed under addition and
$\fS'\oplus \fQ \in \KSQ$,
\\
it follows that $(\fS - \fS')\oplus {\bf 0}_{Q} + \fS'\oplus \fQ = \fS \oplus \fQ$ also
belongs to $\KSQ$.  
\\
Thus, $\KSP \lrar \hat{\K}_{PQ} \subseteq \KSQ.$
\\
3. From parts 1 and 2 above, the equation $\KSP\lrar \KPQ= \KSQ$ can be satisfied by some $\hat{\K}_{PQ}$ if and only if 
$\KSP \circ S \supseteq \KSQ\circ S$ and  $\KSP \times S \subseteq \KSQ\times S.$ 
\\
Next, let $\hat{\K}_{PQ}$ satisfy the equation  $\KSP\lrar \hat{\K}_{PQ} =\KSQ $ and be closed under addition. 
\\
From part 2, we know that if $\hat{\K}_{PQ}$ satisfies $\KSP \circ P \supseteq \hat{\K}_{PQ}\circ P$ and  $\KSP \times P \subseteq \hat{\K}_{PQ}\times P,$
\\
then $\KSP\lrar (\KSP\lrar \hat{\K}_{PQ}) =\hat{\K}_{PQ}.$ But $\hat{\K}_{PQ}$ satisfies $\KSP\lrar \hat{\K}_{PQ} =\KSQ$ and satisfies $\KSP \circ P \supseteq \hat{\K}_{PQ}\circ P$ and  $\KSP \times P \subseteq \hat{\K}_{PQ}\times P,$

It follows that for any such $\hat{\K}_{PQ},$ we have $\KSP\lrar \KSQ=\hat{\K}_{PQ}.$ 
\\
This proves that $\hat{\K}_{PQ}\equiv \KSP\lrar \KSQ$
is the only solution to the equation $\KSP\lrar \KPQ =\KSQ, $ under the condition $\KSP \circ P \supseteq \KPQ\circ P$ and  $\KSP \times P \subseteq \KPQ\times P.$ 
$\al$

\section{Diagrams of expressions of linkages}
\label{sec:diagram}
A regular expression of linkages is best represented by means of a diagram,
where nodes correspond to individual linkages, and edges 
correspond to index sets common to two linkages.
\begin{example} 
Consider the regular expression 
$$\E\equiv \lrar _{A,B,C,P,Q,R,T.V,W}(\K_{S_1QRT},\K_{S_2QPW},\K_{S_3RPV},\K_{S_4ATB},\K_{S_5(-A)VC},
\K_{S_6CBW}).$$
\begin{figure}
\begin{center}
 \includegraphics[width=6in]{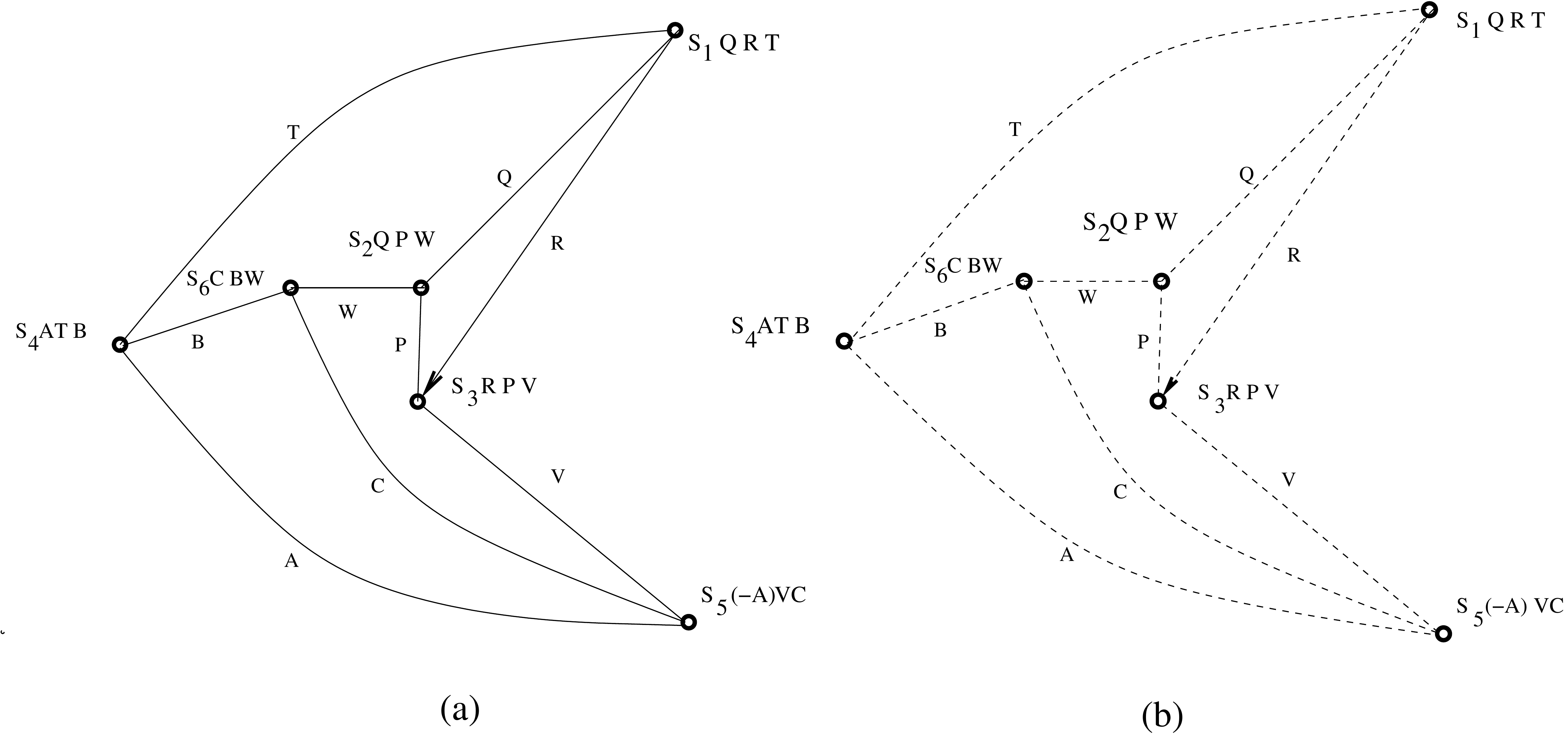}
 \caption{ A linkage diagram and its dual
}
 \label{fig:d1}
\end{center}
\end{figure}
Figure \ref{fig:d1} (a) shows the diagram for this expression. (Individual linkages are represented only by the subscript.
For instance ${(S_5(-A)VC)}$ represents $\K_{(S_5(-A)VC)}.$)
Let $N$ be a subset of the nodes of the diagram.
The {\bf subexpression of $\E$ on $N$}  will have as diagram the subgraph 
with nodes $N$ and edges with both endpoints in $N.$

We need use only $`\lrar$' or only $`\rightleftharpoons$' since the other operation  can be handled by
changing the sign of one of the index sets as in $\K_{(S_5(-A)VC)}.$
We will make edges corresponding to $`\lrar$' bold and those corresponding to $`\rightleftharpoons$' dotted.
In the present expression we have used only $`\lrar$'.

Suppose in the subexpression $\K_{LR}\lrar \K_{RM}$
we have $\K_{LR}\circ R\supseteq \K_{RM} \circ R, \K_{LR}\times R\subseteq \K_{RM} \times R.$
We then make the edge corresponding to $R$ directed from $\K_{LR}$
to $\K_{RM}.$ In the present expression, we have taken $\K_{S_1QRT}, \K_{S_3RPV},$ to satisfy $\K_{S_1QRT}\circ R\supseteq \K_{S_3RPV}\circ R,
\K_{S_1QRT}\times R\subseteq \K_{S_3RPV}\times R,$
so that in the diagram there is a directed edge labelled $R,$ from the node  ${S_1QRT}$
to the node ${S_3RPV}.$

Let us define the dual to the above expression to be 
$$\rightleftharpoons _{A,B,C,P,Q,R,T.V,W}(\K^d_{S_1QRT},\K^d_{S_2QPW},\K^d_{S_3RPV},\K^d_{S_4ATB},\K^d_{S_5(-A)VC},
\K^d_{S_6CBW}).$$

The diagram for this dual expression is shown in Figure \ref{fig:d1} (b).
It is identical to that of the primal except that the a node $\K_S$
has been replaced by $\K^d_S$ and  a bold edge corresponding to an index set
$P$  has been replaced by a dotted edge corresponding to $P.$
Further, directed edges retain their direction in the diagram of the dual.
This is because $(\K_{LR}\lrar \K_{RM})^d \ = \ \K^d_{LR}\rightleftharpoons \K^d_{RM}$
and
$\K_{LR}\circ R\supseteq \K_{RM} \circ R, \K_{LR}\times R\subseteq \K_{RM} \times R$
is equivalent to 
$\K^d_{LR}\times R\subseteq \K^d_{RM} \times R, \K^d_{LR}\circ R\supseteq \K^d_{RM} \circ R.$

Let the above `primal' expression have the evaluation
$\K_{S_1S_2 S_3 S_4 S_5 S_6}.$

From IDT (Theorem \ref{thm:idt}),
it follows that 
the dual expression evaluates to $\K^d_{S_1S_2 S_3 S_4 S_5 S_6}.$

Finally, if every one of the linkages in the expression is a full dimensional
number lattice,
by Theorem \ref{thm:idt2}, every subexpression will evaluate to a 
full dimensional number lattice.
\end{example}
Let $\K_{LR}, \K_{RM}$ be as above with a directed bold edge from $\K_{LR}$ to
 $\K_{RM},$
in the diagram of the expression.

Let us divide the nodes of the diagram into two sets $N_1,N_2$ with
$\K_{LR}\in N_1$ and $\K_{RM}\in N_2,$
where nodes in $N_1$ have no edges between them.
The subexpression on $N_1$ will have the form
$\K_{GHLR}\equiv \oplus_{i,j}\{\K_{G_iH_j}\} \oplus \K_{LR},$ where the $G_i,H_j,L,R,M$ are mutually disjoint,
no index set occurs more than once and $G\equiv \uplus G_i, H\equiv \uplus H_j.$
Let the subexpression on $N_2$ evaluate to $\K_{CHRM}.$

The original expression can be simplified to 
$\K_{GHLR}\lrar \K_{CHRM}.$ The diagram of this reduced expression
will have  only two nodes but many edges corresponding to the $H_j$ and $R.$\\
We will show that  in the reduced diagram, there will be a directed
edge corresponding to $R,$ from node $\K_{GHLR}$ to node $\K_{CHRM},$
in addition to the edges corresponding to the $H_j.$

Firstly, it can be shown that $\K_{LR}\circ R \supseteq \K_{CHRM}\circ R$
and $\K_{LR}\times R\subseteq\K_{CHRM}\times R$
as follows.\\
By Theorem \ref{thm:inversevsnl1},
 we have that $\K_{YZ}\circ Z\supseteq (\K_{XY}\lrar \K_{YZ})\circ Z$ and $\K_{YZ}\times Z\subseteq (\K_{XY}\lrar \K_{YZ})\times Z.$
By repeated application of this result, if necessary, we can infer that
$\K_{RM}\circ R \supseteq \K_{CHRM}\circ R$
and $\K_{RM}\times R\subseteq\K_{CHRM}\times R.$
But we have $\K_{LR}\circ R\supseteq \K_{RM}\circ R$ and 
$\K_{LR}\times R\subseteq \K_{RM}\times R.$
Thus, $\K_{LR}\circ R \supseteq \K_{CHRM}\circ R$
and $\K_{LR}\times R\subseteq\K_{CHRM}\times R.$

Next, since $\K_{GHLR}= \oplus_{i,j}\{\K_{G_iH_j}\} \oplus \K_{LR},$ 
we must have $\K_{GHLR}\circ R=\K_{LR}\circ R$ and
$\K_{GHLR}\times R=\K_{LR}\times R.$
It follows that $\K_{GHLR}\circ R=\K_{LR}\circ R\supseteq \K_{CHRM}\circ R$ and $\K_{GHLR}\times R=\K_{LR}\times R\subseteq \K_{CHRM}\times R.$ 
Therefore, in the reduced diagram there will be a directed
edge corresponding to $R,$ from node $\K_{GHLR}$ to node $\K_{CHRM}.$ 

The dual expression will have a corresponding reduced expression 
$\K^d_{GHLR}\rightleftharpoons \K^d_{CHRM}.$
Here again the arrow for the edge corresponding to $R$ will be from
$\K^d_{GHLR}$ to $\K^d_{CHRM}.$

If the directed edge is from $N_2$ to $N_1,$ in general, the diagram of 
the dual reduced expression may only have an undirected edge.

In the case of the expression 
$$\E\equiv \lrar _{A,B,C,P,Q,R,T.V,W}(\K_{S_1QRT},\K_{S_2QPW},\K_{S_3RPV},\K_{S_4ATB},\K_{S_5(-A)VC},
\K_{S_6CBW}),$$
we can take $N_1, $ for instance, to be
$N_1\equiv \K_{S_1QRT}\oplus \K_{S_5(-A)VC},$ since there is no edge between
the nodes. 

We then have the reduced expression
$$ \E_{red}\equiv \K_{S_1QRTS_5(-A)VC}\lrar
\K_{S_2S_3S_4S_6ACTQRV}$$
$$=
 (\K_{S_1QRT}\oplus \K_{S_5(-A)VC})\lrar 
(\K_{S_2QPW}\lrar\K_{S_3RPV}\lrar\K_{S_4ATB}\lrar \K_{S_6CBW}).$$

\begin{figure}
\begin{center}
 \includegraphics[width=4.75in]{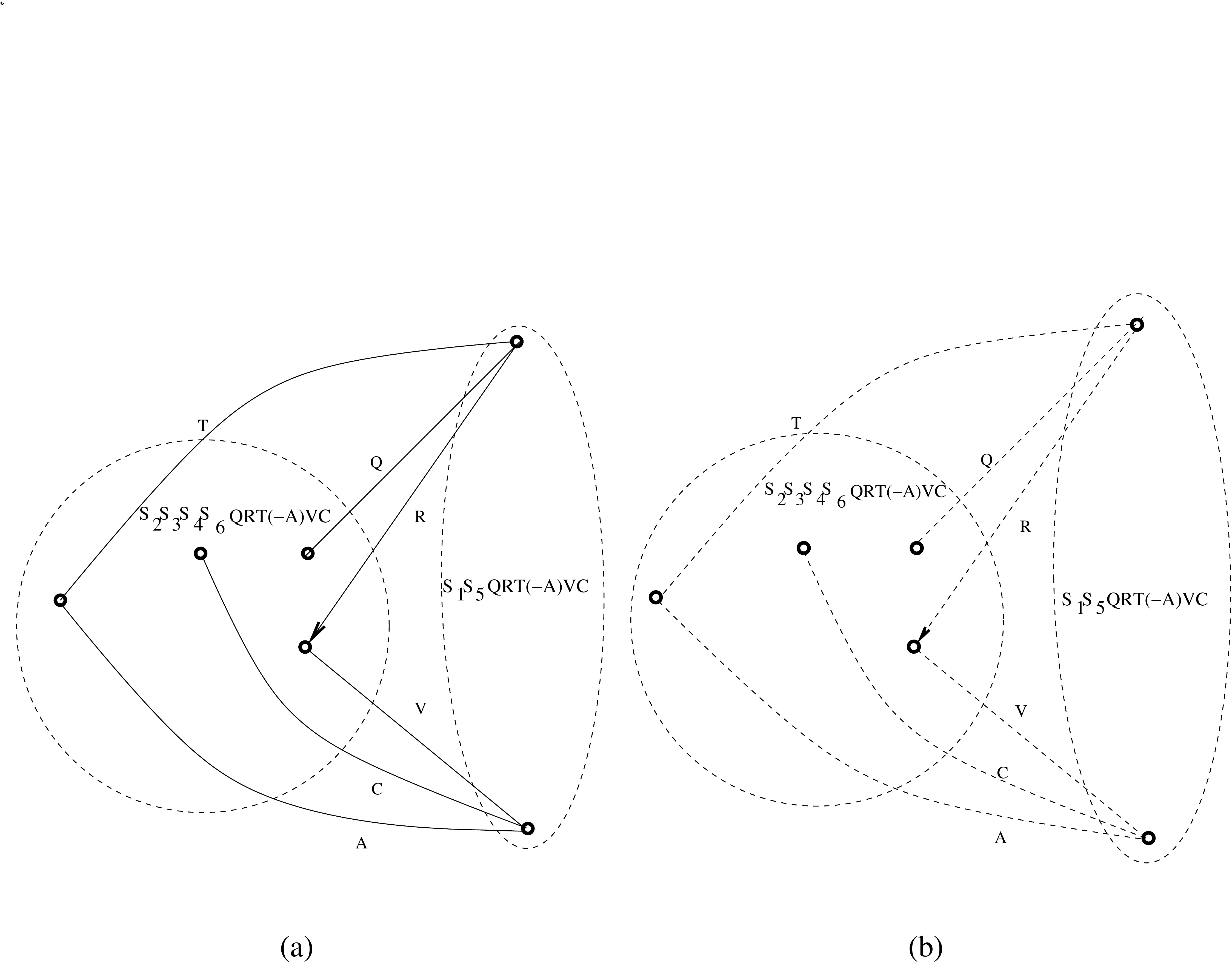}
 \caption{ The reduced linkage diagram and its dual
}
 \label{fig:d2}
\end{center}
\end{figure}

Figure \ref{fig:d2}(a) shows the diagram of this reduced expression.
Here, the nodes have been
split into $N_1,N_2,$ with $N_1$ to the right. 
We have

$N_2\equiv \K_{S_2S_3S_4S_6ACTQRV}=(\K_{S_2QPW}\lrar\K_{S_3RPV}\lrar\K_{S_4ATB}\lrar \K_{S_6CBW}).$

The internal edges (which appear in Figure \ref{fig:d1}), labelled $B,P,W$ involving nodes of $N_2$ have been deleted,
there are undirected edges labelled $T,Q,V,C,A,$
and a directed edge 
 from $N_1$ to $N_2,$ labelled $R$ corresponding to the directed edge
from $\K_{S_1QRT}$ to $\K_{S_3RPV}$ in Fig \ref{fig:d1}(a).
Its meaning here is that 
$$\K_{S_1QRTS_5(-A)VC}\circ R\supseteq \K_{S_2S_3S_4S_6ACTQRV}\circ R, 
\K_{S_1QRTS_5(-A)VC}\times R\subseteq \K_{S_2S_3S_4S_6ACTQRV}\times R.$$

Figure \ref{fig:d2}(b), shows the diagram of the dual of the reduced 
expression. It is identical to the diagram in Figure \ref{fig:d2}(a),
except that we have to interpret the node, say $\alpha\beta \gamma$
as $\K^d_{\alpha\beta \gamma},$ the dotted edges as corresponding to
`$\rightleftharpoons$' and the directed dotted edge
 from $N_1$ to $N_2,$ labelled $R$ as corresponding to the directed dotted edge
from $\K^d_{S_1QRT}$ to $\K^d_{S_3RPV}$ in Fig \ref{fig:d1}(b).
Its meaning here is that
$$\K^d_{S_1QRTS_5(-A)VC}\circ R\supseteq \K^d_{S_2S_3S_4S_6ACTQRV}\circ R, 
\K^d_{S_1QRTS_5(-A)VC}\times R\subseteq \K^d_{S_2S_3S_4S_6ACTQRV}\times R.$$
 
These ideas extend to the case where the set of nodes of the diagram is partitioned into $\{N_1,\cdots , N_k\},$  if there are internal edges,  only the block of nodes $N_k$ contains them,
 and finally,  the directed edges are either between the node sets $N_1,\cdots , N_{k-1},$ or into $N_k.$
\section{Proof of Lemma \ref{lem:basetobase}}
\label{sec:basetobase}

It is sufficient to  prove the result when the priority sequences 
for $B_1, B_2,$ 
are of the form\\
$(e_1, \cdots , e_{i-1}, e, e', e_{i+2}, \cdots e_{|S|})$ for $B_1,$ and of the form   
$(e_1, \cdots , e_{i-1}, e', e, e_{i+2}, \cdots e_{|S|})$
 for  $B_2.$ 
Without loss of generality we can assume the standard representative 
matrix of $\Vs$ with respect to the base $B_1$ can be written as
\begin{align}
\label{eqn:basetobase}
\bbmatrix{I&\vdots &\0&\vdots&k_1&\vdots & K_1\\
0&| &1&|&k_2&|& K_2}
\end{align}
where $k_2= \pm 1$ and the order of columns is $B_1\cap B_2, e, e', (S-B_1)\cap (S-B_2).$
To convert this into a standard representative matrix of $\Vs$ with respect
to $B_2$ we have to add or subtract the last row from the earlier rows
in order to convert the entries in the column $e'$ in these earlier rows
from $k_1$ to the $\0$ column. This results in a standard representative
matrix and therefore must be totally unimodular. So the entries in the 
columns corresponding to $e,(S-B_1)\cap (S-B_2)$ in this new matrix 
must be $0,\pm 1$ and therefore all numbers encountered in this
 computation are $0,\pm 1.$ This computation is $O(|B_1\cap B_2||S|)$
when $|B_1-B_2|=1 $ and in general has to be repeated $|B_1-B_2|$
times so that the computation is finally $O(|B_1|^2|S|)=O(m^2|S|).$
 $\al$

\section{Proof of Lemma \ref{lem:perm}}
It is sufficient to prove the following statement  for the pair of permutations,
$(id(\cdot), \rho(\cdot)), \rho\equiv \mu\nu^{-1}.$   

{\it Statement}
Let $S\equiv \{t, \cdots, t+n\}, t\in \mathbb{Z},$ and let $\rho(\cdot)$ act on $S.$
Let $S'\equiv \{j, i\leq j\leq \rho(i),i\in S\}.$
Then, $S'= S.$

Note that if the statement is true for $t=t'$ it is also true for $t= t'+ s, s \in \mathbb{Z}.$

The statement is clearly true for $n=1.$
Let the statement be true for $n\leq k.$

Now, let $n=k+1.$
Let $T_r\equiv \{t, \cdots , t+r\}\subseteq S.$ We have $T_0\subseteq S'.$

We will show that if $0\leq r\leq k,$ then $T_{r+1}\equiv \{t, \cdots , t+r+1\}\subseteq S'.$

If $\rho(T_r)=T_r,$ then $\rho(S-T_r)=S-T_r.$
So the statement is true since it is true for $n\leq k$
and we have $0\leq r\leq k,0\leq k-r \leq k.$
(We are in this case dealing with two permutations
on sets of smaller size.)

Suppose $\rho(T_r)\ne T_r.$
Then there exists $m, 0\leq m\leq r$ such that 
$\rho(t+m)>t+r.$ Therefore, 
$t+m< t+r+1 \leq \rho(t+m),$ so that 
$t+r+1\in S',$ i.e., $T_{r+1}\equiv \{t, \cdots , t+r+1\}\subseteq S.$

We conclude that $S=T_{k+1} \subseteq S'.$

Thus the statement is true for n=k+1.
$\ \ \ \ \ \ \ \ \ \ \ \ \ \ \ \  \ \al$

\section{LLL-reduced basis for dual number lattice}
\label{sec:LLLdual}
The proof of Theorem \ref{thm:lovcond} is based on the following simple lemma which can be verified by direct  calculation.
\begin{lemma}
\label{lem:triangleinv}
Let $F$ be a lower triangular matrix with \\$F_{ii}=||b^*_i||,
i=1, \cdots ,m \ \mbox{and}\ F_{(i+1)i}=\alpha_{(i+1)i}||b^*_i||, i=1, \cdots ,m-1 .$
Then, $G\equiv T_m(F^{-1})^TT_m$ is a lower triangular matrix satisfying
$$G_{ii}=\frac{1}{||b^*_{m-i+1}||},
i=1, \cdots ,m \ \mbox{and}\ G_{(i+1)i}=-\alpha_{(m-i+1)(m-i)}\frac{1}{||b^*_{m-i+1}||},
i=1, \cdots ,m-1.$$
\end{lemma}
{\it Proof of Theorem \ref{thm:lovcond}}\\
1. This is clear from Lemma \ref{lem:triangleinv}.
\\ 
2. The Lovasz condition for $B_S$ is
$$|F_{(i+1)i}|^2+|F_{(i+1)(i+1)}|^2\geq \delta(|F_{ii}|^2), i.e., 
\ |\alpha_{(i+1)i}|^2 ||b^*_i||+ ||b^*_{i+1}||^2\geq \delta(||b^*_i||^2), i=1, \cdots ,m-1 .$$
Dividing the above
inequality throughout by $||b^*_i||^2\times ||b^*_{i+1}||^2,$
we get
$$|\alpha_{(i+1)i}|^2||b^*_{i+1}||^{-2}+|||b^*_{i}||^{-2}\geq \delta (||b^*_{i+1}||^{-2}),  i=1, \cdots ,m-1 .$$

Therefore,
$$|\alpha_{(m-i+1)(m-i)}|^2||b^*_{m-i+1}||^{-2}+|||b^*_{m-i}||^{-2}\geq \delta (||b^*_{m-i+1}||^{-2}), i=1, \cdots ,m-1 ,$$
$i.e.,
|G_{(i+1)i}|^2+ |G_{(i+1)(i+1)}|^2\geq \delta(|G_{ii}|^2), i=1, \cdots ,m-1 ,$
which proves the required result. $\al$


%

\bibliographystyle{elsarticle1-num}
\bibliography{references}

\begin{thebibliography}{10}
\expandafter\ifx\csname url\endcsname\relax
  \def\url#1{\texttt{#1}}\fi
\expandafter\ifx\csname urlprefix\endcsname\relax\def\urlprefix{URL }\fi
\expandafter\ifx\csname href\endcsname\relax
  \def\href#1#2{#2} \def\path#1{#1}\fi

\bibitem{arora} S.~Arora, L.~Babai, J.~Sweedyk, The hardness of approximate optima in lattices, codes, and systems of linear equations, J. Comput. Syst. Sci. 54.(1997)  317--331.
\bibitem{babai1} L.~Babai, On Lovasz’ lattice reduction and the nearest lattice point problem.
Combinatorica
, 6(1):(1986) 1--13.

\bibitem{banas} W.~Banaszczyk, New bounds in some transference theorems in the geometry of numbers, Mathematische Annalen, 296(4):625–635, 1993.

 
\bibitem{cassels} J.~W.~S.~Cassels.
An introduction to the geometry of numbers
.  Springer-Verlag, New York, 1971.


\bibitem{dixon}J.~D.~Dixon,
Exact Solution of Linear Equations
Using p-Adic Expansions,
Numer. Math. (40) (1982) 137--141.

\bibitem{Forney2004}
G.~D. Forney, M.~D. Trott, {The Dynamics of Group Codes : Dual Abelian Group
  Codes and Systems}, IEEE Transactions on Information Theory 50~(11) (2004)
  1--30.

\bibitem{hafner}J.~L.~Hafner and K.~S.~McCurley, Asymptotically fast triangularization of matrices
over rings, SIAM J. Comput. 20 (1991), no. 6, 1068--1083.



%


\bibitem{kannan} R.~Kannan,  Algorithmic
Geometry of numbers, Annual Reviews of Computer Science, 2
 (1987) 231--267. Annual Review Inc. , Palo Alto,
California.
\bibitem{khot}S.~Khot, Hardness of approximating the shortest vector problem in lattices, J. ACM. 52  (2005) 789--808.

\bibitem{lagarias}J.~C.~ Lagarias, H.~ W.~ Lenstra, Jr., and C.~P.~ Schnorr. Korkin-Zolotarev bases and successive minima of
a lattice and its reciprocal lattice. Combinatorica, 10 (1990) 333--348.

\bibitem{LLL} J. ~K. ~Lenstra, H. ~W. ~Lenstra  and L. ~Lovasz, Factoring polynomials with rational coeffidents,
Math. Ann. 261 (1982), 515--534.

\bibitem{micciwasser}  D.~Micciancio  and  S.~Goldwasser.
Complexity  of  Lattice  Problems:   a  cryptographic  perspective
,
Kluwer Academic
Publishers, Boston, Massachusetts, 2002.

\bibitem{micci} D.~Micciancio. Lecture notes on lattice algorithms and applications, 2014. Available at
http://cseweb.ucsd.edu/ ̃daniele/classes.html, last accessed 17 Oct 2014.



\bibitem{HNarayanan1986a}
H.~Narayanan, {On the decomposition of vector spaces}, Linear Algebra and its
  Applications 76 (1986) 61--98.


\bibitem{HNarayanan}
H.~Narayanan, {A Unified Construction of Adjoint Systems and Networks}, Circuit
  Theory and Applications 14 (1986) 263--276.

\bibitem{narayanan1987topological}
H.~Narayanan, {Topological transformations of electrical networks},
  International Journal of Circuit Theory and Applications 15~(3) (1987)
  211--233.

\bibitem{narayanan2002some}
H.~Narayanan, Some applications of an implicit duality theorem to connections
  of structures of special types including dirac and reciprocal structures,
  Systems \& Control Letters 45~(2) (2002) 87 -- 95.


\bibitem{HNarayanan1997}
H.~Narayanan, Submodular Functions and Electrical Networks, Annals of Discrete
  Mathematics, vol. 54, North Holland, Amsterdam, 1997.
(Open 2nd edition,
 http://www.ee.iitb.ac.in/$\tilde{\  }$ hn/ , 2009.)

\bibitem{HN2000}
H.~Narayanan, On the Duality between Controllability and Observability in Behavioural Systems Theory, Proc. International Conference on Communications Control and Signal Processing, CCSP 2000, Bangalore, July 25-July 28,2000, 183--186.
(https://www.ee.iitb.ac.in/~hn/otherpapers.htm)
\bibitem{HN2000a}
H.~Narayanan, Matroids representable over Modules, Electrical Network Topology and Behavioural Systems Theory, Research report of the EE Dept., IIT Bombay, May 2000.
(https://www.ee.iitb.ac.in/~hn/otherpapers.htm)

\bibitem{peikert} C.~Peikert. A decade of lattice cryptography, Cryptology ePrint Archive,
Report 2015/939, 2015. http://eprint.iacr.org/.

\bibitem{phong} Phong~Q.~Nguyen and  D.~Stehle,
An LLL algorithm  with quadratic complexity, 
Siam J. Comput.
(39) (2009)   874--903.

\bibitem{regev}O.~ Regev, Lattices in computer science,  Lecture 
notes of a course given in Tel Aviv University (2004).

\bibitem{STHN2014}
Siva Theja, H. ~Narayanan,
On the notion of generalized minor in topological network theory and matroids,
Linear Algebra and its Applications 458 (2014) 1--46.

\bibitem{storjohann}
A.~Storjohann and G.~Labahn, Asymptotically fast computation of Hermite normal forms of integer matrices, ISSAC'96 (Zurich,Switzerland,1996),ACM, 259--266.

\bibitem{storjohanninvert}
 A.~ Storjohann, On the complexity of inverting integer and polynomial matrices,
Computational Complexity,
(24)  (2015) 777--821

\bibitem{tutte} W.~T. Tutte, Lectures on matroids, { Journal of 
Research of the National Bureau of Standards}, {vol. B69} (1965) 1--48. 







\end{thebibliography}

\end{document}